\documentclass{amsart}
\usepackage{amssymb}
\usepackage[nobysame]{amsrefs}
\usepackage{mathrsfs}
\usepackage[usesnames,svgnames]{xcolor}
\usepackage[sc]{mathpazo}
\usepackage{tikz,pgf}
	\pgfdeclarelayer{background}
	\pgfsetlayers{background,main}
\usepackage{enumerate}
\usepackage{todonotes}
\usepackage{subcaption}
\usepackage{mathtools}
\usepackage{dsfont} 
\usepackage[colorlinks=true,
	urlcolor=MidnightBlue,
	citecolor=DarkGreen]{hyperref}
\newtheorem{theorem}{Theorem}[section]
\newtheorem{proposition}[theorem]{Proposition}
\newtheorem{conjecture}[theorem]{Conjecture}
\newtheorem{lemma}[theorem]{Lemma}
\newtheorem{corollary}[theorem]{Corollary}
\newtheorem{definition}[theorem]{Definition}
\newtheorem{question}[theorem]{Question}
\newtheorem{assumption}[theorem]{Assumption}
\theoremstyle{remark}
\newtheorem{remark}[theorem]{Remark}
\newtheorem{example}[theorem]{Example}

\newcommand{\defn}[1]{{\color{DarkGreen}\emph{#1}}}
\newcommand{\defs}{\overset{\mathrm{def}}{=}}
\newcommand{\ie}{\text{i.e.}\;}
\newcommand{\red}[2]{\mathrm{Red}_{#1}(#2)}

\newcommand{\CC}{\mathscr{C}}
\newcommand{\EE}{\mathscr{E}}
\newcommand{\HH}{\mathscr{H}}
\newcommand{\II}{\mathscr{I}}
\newcommand{\MM}{\mathscr{M}}
\newcommand{\PP}{\mathcal{P}}

\newcommand{\xb}{\mathbf{x}}

\newcommand{\st}{^{\text{st}}}

\renewcommand{\th}{^{\text{th}}}
\DeclareMathOperator{\rk}{rk}

\DeclareMathOperator{\rise}{Rise}
\DeclareMathOperator{\orb}{Orb}

\newcommand{\bg}{\mathfrak{B}}  
\newcommand{\id}{\mathds{1}}		
\renewcommand{\top}{\mathfrak{c}} 	
\newcommand{\gen}{A}			
\newcommand{\rtwo}{B}		
\newcommand{\grp}{G}			
\newcommand{\optfrak}[1]{#1}   
\author{Henri M{\"u}hle}
\author{Vivien Ripoll}
\address{Institut f{\"u}r Algebra, Technische Universit{\"a}t Dresden, Zellescher Weg 12--14, 01069 Dresden, Germany.}
\email{henri.muehle@tu-dresden.de}
\address{Fakult{\"a}t f{\"u}r Mathematik, Universit{\"a}t Wien, Oskar-Morgenstern-Platz 1, 1090 Vienna, Austria.}
\email{vivien.ripoll@univie.ac.at}
\thanks{HM is partially supported by a Public Grant overseen by the French National Research Agency (ANR) as part of the ``Investissements d'Avenir'' Program (Reference: ANR-10-LABX-0098) and by Digiteo project PAAGT (Nr. 2015-3161D).  VR is supported by the Austrian Science Foundation FWF, Grants Z130-N13 and F50-N15, the latter in the framework of the Special Research Program ``Algorithmic and Enumerative Combinatorics''.}
\title{Connectivity Properties of Factorization Posets in Generated Groups}

\begin{document}

\allowdisplaybreaks

\begin{abstract}
	We consider three notions of connectivity and their interactions in partially ordered sets coming from reduced factorizations of an element in a generated group.  While one form of connectivity essentially reflects the connectivity of the poset diagram, the other two are a bit more involved: Hurwitz-connectivity has its origins in algebraic geometry, and shellability in topology.  We propose a framework to study these connectivity properties in a uniform way.  Our main tool is a certain linear order of the generators that is compatible with the chosen element. 
\end{abstract}

\maketitle

\section{Introduction}
	\label{sec:introduction}
For any group $\grp$ the braid group $\bg_{n}$ on $n$ strands naturally acts on $n$-tuples of elements of $\grp$.  This \defn{Hurwitz action} is defined as follows: the $i\th$ standard generator of $\bg_n$ acts on an $n$-tuple of group elements by swapping the $i\th$ entry and the $(i+1)\st$ entry, and moreover conjugating one by the other, so that the product of all the elements remains unchanged (see \eqref{eq:hurwitz} for the precise definition).  This action goes back to \cite{hurwitz91ueber}, where it appeared in the study of branched coverings of Riemann surfaces and was applied to the case where $\grp$ is the symmetric group.  It also plays a role in the computation of the braid monodromy of projective curves~\cites{brieskorn88automorphic,kulikov00braid,libgober89invariants}.

It is a natural question to ask for the number of orbits in $\grp^{n}$ under Hurwitz action.  There are several results on conditions for two elements of $\grp^{n}$ to belong to the same Hurwitz orbit.  To name a few, \cite{benitzhak03graph} deals with the symmetric group, \cite{hou08hurwitz} investigates generalized quaternion groups and dihedral groups and \cite{sia09hurwitz} treats dihedral groups, dicyclic groups and semidihedral groups.  

Since the Hurwitz action preserves the multiset of conjugacy classes of a tuple, it is natural to study the action on $A^{n}$ for any subset $A\subseteq G$ that is closed under $G$-conjugation. Without loss of generality we may then assume that $\grp$ is generated by $\gen$ as a monoid. Some authors refer to this setting as an ``equipped group''~\cites{kharlamov03braid,kulikov13factorizations}.  We will, however, use the notion of a \defn{generated group} following \cites{bessis03dual,huang17absolute}.  A fundamental case is the study of the Hurwitz action on \defn{reduced $\gen$-factorizations} of some element $\top\in\grp$, and the question of the number of Hurwitz orbits.  

Let us write $\red{\gen}{\top}$ for the set of all reduced $\gen$-factorizations of $\top$.  It is well known that when $\grp$ is a finite irreducible real reflection group, $\gen$ its set of reflections, and $\top$ a Coxeter element, the Hurwitz action is transitive on $\red{\gen}{\top}$~\cite{deligne74letter}.  This was later generalized to finite irreducible complex reflection groups~\cite{bessis15finite}*{Proposition~7.6}.  It was shown recently that for $\grp$ a finite irreducible real reflection group the Hurwitz action on $\red{\gen}{\top}$ is transitive if and only if $\top$ is a (parabolic) quasi-Coxeter element~\cite{baumeister17on}*{Theorem~1.1}.  One direction of this equivalence was recently extended to affine Coxeter groups~\cite{wegener17on}*{Theorem~1.1}.  In a similar spirit, \cite{muehle18poset}*{Theorem~1.2} enumerates the Hurwitz orbits for elements of the alternating group generated by all $3$-cycles.

In \cites{bessis03dual,bessis06dual} and also \cite{brady01partial}, generated groups were equipped with an additional structure of a partially ordered set.  This construction has its origin already in \cite{garside69braid}.  More precisely, we consider the set of all prefixes (up to equivalence) of some element $\top\in\grp$, and we say that two such prefixes are comparable if one appears as a subword in some reduced $\gen$-factorization of the other.  From this point of view maximal chains in the resulting partial order correspond to reduced $\gen$-factorizations of $\top$ and the Hurwitz action can be seen as a method to pass from one chain to the other.  The number of Hurwitz orbits of $\red{\gen}{\top}$ can then be interpreted as a ``connectivity coefficient'' of this \defn{factorization poset} $\PP_{\top}(\grp,\gen)$.

This article revolves around the relation of the previously described \defn{Hurwitz-connectivity} to two other forms of connectivity of a poset: \defn{chain-connectivity} (motivated by graph theory) and \defn{shellability} (motivated by topology).  The main result of this article is the following uniform approach to proving Hurwitz-connectivity, chain-connectivity and shellability of $\PP_{\top}(\grp,\gen)$.  The statement of this result uses two notions that will be formally defined later in the article, namely a linear order of the generators that is compatible with $\top$ (Definition~\ref{def:compatible_order}), and a certain ``well-covered'' property (Definition~\ref{def:well_covered}).  The latter property asserts that for every generator that is not minimal with respect to a given linear order we can find a smaller generator such that both have a common upper cover in $\PP_{\top}(\grp,\gen)$.  

Let us fix the following notation for the upcoming three statements.  Let $\grp$ denote a group that is generated by $\gen\subseteq\grp$ as a monoid, where we suppose that~$\gen$ is closed under $\grp$-conjugation.  For $\top\in\grp$ let $\gen_{\top}\subseteq\gen$ denote the set of all generators that appear in at least one reduced $\gen$-factorization of $\top$.

\begin{theorem}\label{thm:main_result}
	If $\red{\gen}{\top}$ is finite, and the factorization poset $\PP_{\top}(\grp,\gen)$ admits a $\top$-compatible order $\prec$ of $\gen_{\top}$ and is totally well-covered with respect to $\prec$, then $\PP_{\top}(\grp,\gen)$ is chain-connected, Hurwitz-connected and shellable.
\end{theorem}

We want to emphasize that Theorem~\ref{thm:main_result} uniformly and simultaneously approaches the question whether a factorization poset is chain-connected, Hurwitz-connected or shellable.  We are well aware that it is far from trivial in full generality to establish that a factorization poset is well-covered and admits a compatible order.  However, for some special groups the framework presented here may provide a convenient method to reach uniform insights about the connectivity of the corresponding factorization posets.

By definition, a factorization poset can only be well-covered with respect to a given linear order of the generators.  We conjecture that we can weaken the assumptions of Theorem~\ref{thm:main_result} a bit; and we are able to prove this conjecture for the Hurwitz-connectivity part.

\begin{theorem}\label{thm:main_hurwitz}
	If $\red{\gen}{\top}$ is finite, and the factorization poset $\PP_{\top}(\grp,\gen)$ is chain-connected and admits a $\top$-compatible order of $\gen_{\top}$, then the Hurwitz action is transitive on $\red{\gen}{\top}$. 
\end{theorem}

\begin{conjecture}\label{conj:main_shellable}
	If $\red{\gen}{\top}$ is finite, and the factorization poset $\PP_{\top}(\grp,\gen)$ is totally chain-connected and admits a $\top$-compatible order of $\gen_{\top}$, then $\PP_{\top}(\grp,\gen)$ is shellable.
\end{conjecture}

The main part of this article deals with a proper study of the implications and non-implications between the three types of connectivity; see Figure~\ref{fig:the_big_picture} for an overview.  We also provide three other versions of Conjecture~\ref{conj:main_shellable}: Conjectures~\ref{conj:compatible_el_shellable} and \ref{conj:compatible_well_covered} are phrased in terms of factorization posets, whereas Conjecture~\ref{conj:compatible_linear_cycle_graph} is formulated in terms of a particular graph that represents the local structure of the factorization poset.  This tool enables us to prove a particular case of Conjecture~\ref{conj:main_shellable}; see Theorem~\ref{thm:generator_single_cycle}.

\begin{figure}
	\centering
	\hspace*{-1.5cm}\begin{tikzpicture}
		\def\x{1};
		\def\y{1};
		\def\s{1};
		\draw(7*\x,1*\y) node[draw,minimum width=4cm,minimum height=1cm,text width=3.75cm,text centered]{reduced cycle graph linear};
		\draw(10.5*\x,3*\y) node[draw,minimum width=4cm,minimum height=1cm,text width=3.75cm,text centered]{totally well-covered + compatible order};
		\draw(5*\x,3*\y) node[draw,minimum width=4cm,minimum height=1cm,text width=3.75cm,text centered]{totally chain-connected + compatible order};
		\draw(3*\x,5*\y) node[draw,minimum width=4cm,minimum height=1cm,text width=3.75cm,text centered]{chain-connected + compatible order};
		\draw(1*\x,7*\y) node[draw,minimum width=4cm,minimum height=1cm,text width=3.75cm,text centered]{locally Hurwitz-connected};
		\draw(6*\x,7*\y) node[draw,minimum width=4cm,minimum height=1cm,text width=3.75cm,text centered]{compatible order};
		\draw(11*\x,7*\y) node[draw,minimum width=4cm,minimum height=1cm,text width=3.75cm,text centered]{$\lambda_{\top}$ EL-labeling};
		\draw(3*\x,9*\y) node[draw,minimum width=4cm,minimum height=1cm,text width=3.75cm,text centered]{Hurwitz-connected};
		\draw(9*\x,9*\y) node[draw,minimum width=4cm,minimum height=1cm,text width=3.75cm,text centered]{shellable};
		\draw(6*\x,11*\y) node[draw,minimum width=4cm,minimum height=1cm,text width=3.75cm,text centered]{chain-connected};
		\draw(6*\x,13*\y) node[draw,minimum width=4cm,minimum height=1cm,text width=3.75cm,text centered]{totally chain-connected};
		\draw[double,->,red!90!black](2.5*\x,8.25*\y) -- (2*\x,7.75*\y) node[black,midway,sloped,below]{\tiny Ex.~\ref{ex:hurwitz_nonlocal}} node[midway,sloped]{\tiny $\backslash$};
		\draw[double,->,red!90!black](1.5*\x,7.75*\y) -- (2*\x,8.25*\y) node[black,midway,sloped,above]{\tiny Ex.~\ref{ex:abelian_quotient_disconnected}} node[midway,sloped]{\tiny $\backslash$};
		\draw[double,->](.75*\x,5.25*\y) .. controls (-2*\x,6.5*\y) and (-2*\x,7.5*\y) .. (.75*\x,8.75*\y) node[black,midway,sloped,above]{\tiny Thm.~\ref{thm:main_hurwitz}};
		\draw[double,->](4*\x,9.75*\y) -- (4.5*\x,10.25*\y) node[black,midway,sloped,above]{\tiny Prop.~\ref{prop:chain_connectivity_necessary}};
		\draw[double,->,red!90!black](5*\x,10.25*\y) -- (4.5*\x,9.75*\y) node[black,midway,sloped,below]{\tiny Ex.~\ref{ex:chain_nonhurwitz}} node[midway,sloped]{\tiny $\backslash$};
		\draw[double,->,red!90!black](7.5*\x,10.25*\y) -- (8*\x,9.75*\y) node[black,midway,sloped,above]{\tiny Ex.~\ref{ex:hurwitz_nonshellable}} node[midway,sloped]{\tiny $\backslash$};
		\draw[double,->](7.5*\x,9.75*\y) -- (7*\x,10.25*\y) node[black,midway,sloped,below]{\tiny Prop.~\ref{prop:chain_connectivity_necessary}};
		\draw[double,->,red!90!black](5.25*\x,9.25*\y) -- (6.75*\x,9.25*\y) node[black,midway,sloped,above]{\tiny Ex.~\ref{ex:hurwitz_nonshellable}} node[midway,sloped]{\tiny $\backslash$};
		\draw[double,->,red!90!black](6.75*\x,8.75*\y) -- (5.25*\x,8.75*\y) node[black,midway,sloped,below]{\tiny Ex.~\ref{ex:chain_nonhurwitz}} node[midway,sloped]{\tiny $\backslash$};
		\draw[double,->,red!90!black](4.5*\x,7.75*\y) -- (4*\x,8.25*\y) node[black,midway,sloped,below]{\tiny Ex.~\ref{ex:abelian_quotient_disconnected}} node[midway,sloped]{\tiny $\backslash$};
		\draw[double,->,red!90!black](4.5*\x,8.25*\y) -- (5*\x,7.75*\y) node[black,midway,sloped,above]{\tiny Ex.~\ref{ex:hurwitz_nonshellable}} node[midway,sloped]{\tiny $\backslash$};
		\draw[double,->](3.75*\x,6.75*\y) -- (3.25*\x,6.75*\y) node[black,midway,sloped,below]{\tiny Lem.~\ref{lem:compatible_rank_2_hurwitz}};
		\draw[double,->,red!90!black](3.25*\x,7.25*\y) -- (3.75*\x,7.25*\y) node[black,midway,sloped,above]{\tiny Ex.~\ref{ex:big_cycle_example}} node[midway,sloped]{\tiny $\backslash$};
		\draw[double,->](8.75*\x,6.75*\y) -- (8.25*\x,6.75*\y) node[black,midway,sloped,below]{\tiny Def.};
		\draw[double,->,red!90!black](8.25*\x,7.25*\y) -- (8.75*\x,7.25*\y) node[black,midway,sloped,above]{\tiny Ex.~\ref{ex:abelian_quotient_disconnected}} node[midway,sloped]{\tiny $\backslash$};
		\draw[double,->](10*\x,7.75*\y) -- (9.5*\x,8.25*\y) node[black,midway,sloped,below]{\tiny Def.};
		\draw[double,->,red!90!black](10*\x,8.25*\y) -- (10.5*\x,7.75*\y) node[black,midway,sloped,above]{\tiny Ex.~\ref{ex:chain_nonhurwitz}} node[midway,sloped]{\tiny $\backslash$};
		\draw[double,->,red!90!black](5.75*\x,11.75*\y) -- (5.75*\x,12.25*\y) node[black,midway,sloped,above]{\tiny Ex.~\ref{ex:hurwitz_nonshellable}} node[midway,sloped]{\tiny $\backslash$};
		\draw[double,->](6.25*\x,12.25*\y) -- (6.25*\x,11.75*\y) node[black,midway,sloped,above]{\tiny Def.};
		\draw[double,->,red!90!black](1.5*\x,9.75*\y) .. controls (1.5*\x,11.5*\y) and (2.5*\x,13*\y) .. (3.75*\x,13.25*\y) node[black,midway,sloped,above]{\tiny Ex.~\ref{ex:hurwitz_nonshellable}} node[midway,sloped]{\tiny $\backslash$};
		\draw[double,->,red!90!black](3.75*\x,12.75*\y) .. controls (3*\x,12.5*\y) and (2.25*\x,11.5*\y) .. (2*\x,9.75*\y) node[black,midway,sloped,below]{\tiny Ex.~\ref{ex:chain_nonhurwitz}} node[midway,sloped]{\tiny $\backslash$};
		\draw[double,->,red!90!black](8.25*\x,13.25*\y) .. controls (9.5*\x,13*\y) and (10.5*\x,11.5*\y) .. (10.5*\x,9.75*\y) node[black,midway,sloped,above]{\tiny Fig.~\ref{fig:dunce_hat} / Qu.~\ref{qu:chain_connected_nonshellable}} node[midway,sloped]{\tiny $\backslash$};
		\draw[double,->](10*\x,9.75*\y) .. controls (9.75*\x,11.5*\y) and (9*\x,12.5*\y) .. (8.25*\x,12.75*\y) node[black,midway,sloped,below]{\tiny Def. + Prop.~\ref{prop:chain_connectivity_necessary}};
		\draw[double,->](4*\x,5.75*\y) -- (4.5*\x,6.25*\y) node[black,midway,sloped,above]{\tiny Def.};
		\draw[double,->,red!90!black](5*\x,6.25*\y) -- (4.5*\x,5.75*\y) node[black,midway,sloped,below]{\tiny Ex.~\ref{ex:abelian_quotient_disconnected}} node[midway,sloped]{\tiny $\backslash$};
		\draw[double,->](4.5*\x,3.75*\y) -- (4*\x,4.25*\y) node[black,midway,sloped,below]{\tiny Def.};
		\draw[double,->,red!90!black](7*\x,7.75*\y) -- (7.5*\x,8.25*\y) node[black,midway,sloped,above]{\tiny Ex.~\ref{ex:abelian_quotient_disconnected}} node[midway,sloped]{\tiny $\backslash$};
		\draw[double,->,red!90!black](8*\x,8.25*\y) -- (7.5*\x,7.75*\y) node[black,midway,sloped,below]{\tiny Ex.~\ref{ex:chain_nonhurwitz}} node[midway,sloped]{\tiny $\backslash$};
		\draw[<->,double](11.25*\x,6.25*\y) -- (10.75*\x,3.75*\y) node[black,midway,sloped,below]{\tiny Thm.~\ref{thm:shellable_well_covered}};
		\draw[->,double](9.5*\x,6.25*\y) -- (6.5*\x,3.75*\y) node[black,midway,sloped,below]{\tiny Def. + Prop.~\ref{prop:chain_connectivity_necessary}};
		\draw[->,double](8.25*\x,2.75*\y) -- (7.25*\x,2.75*\y) node[black,midway,sloped,below]{\tiny Prop.~\ref{prop:well_covered_chain_connected}};
		\draw[->,double](8.25*\x,1.75*\y) -- (8.75*\x,2.25*\y) node[black,midway,sloped,above]{\tiny Prop.~\ref{prop:reduced_graph_linear_well_covered}};
		\draw[->,very thick,green!90!black](6*\x,3.75*\y) -- (9*\x,6.25*\y) node[black,midway,sloped,above]{\tiny Conj.~\ref{conj:compatible_el_shellable}};
		\draw[->,very thick,green!90!black](7.25*\x,3.25*\y) -- (8.25*\x,3.25*\y) node[black,midway,sloped,above]{\tiny Conj.~\ref{conj:compatible_well_covered}};
		\draw[->,very thick,green!90!black](5.25*\x,2.25*\y) -- (5.75*\x,1.75*\y) node[black,midway,sloped,above]{\tiny Conj.~\ref{conj:compatible_linear_cycle_graph}};
	\end{tikzpicture}
	\caption{Implications, non-implications and conjectures between the several properties of $\PP_{\top}$.  We have omitted a few arrows coming from transitivity.}
	\label{fig:the_big_picture}
\end{figure}
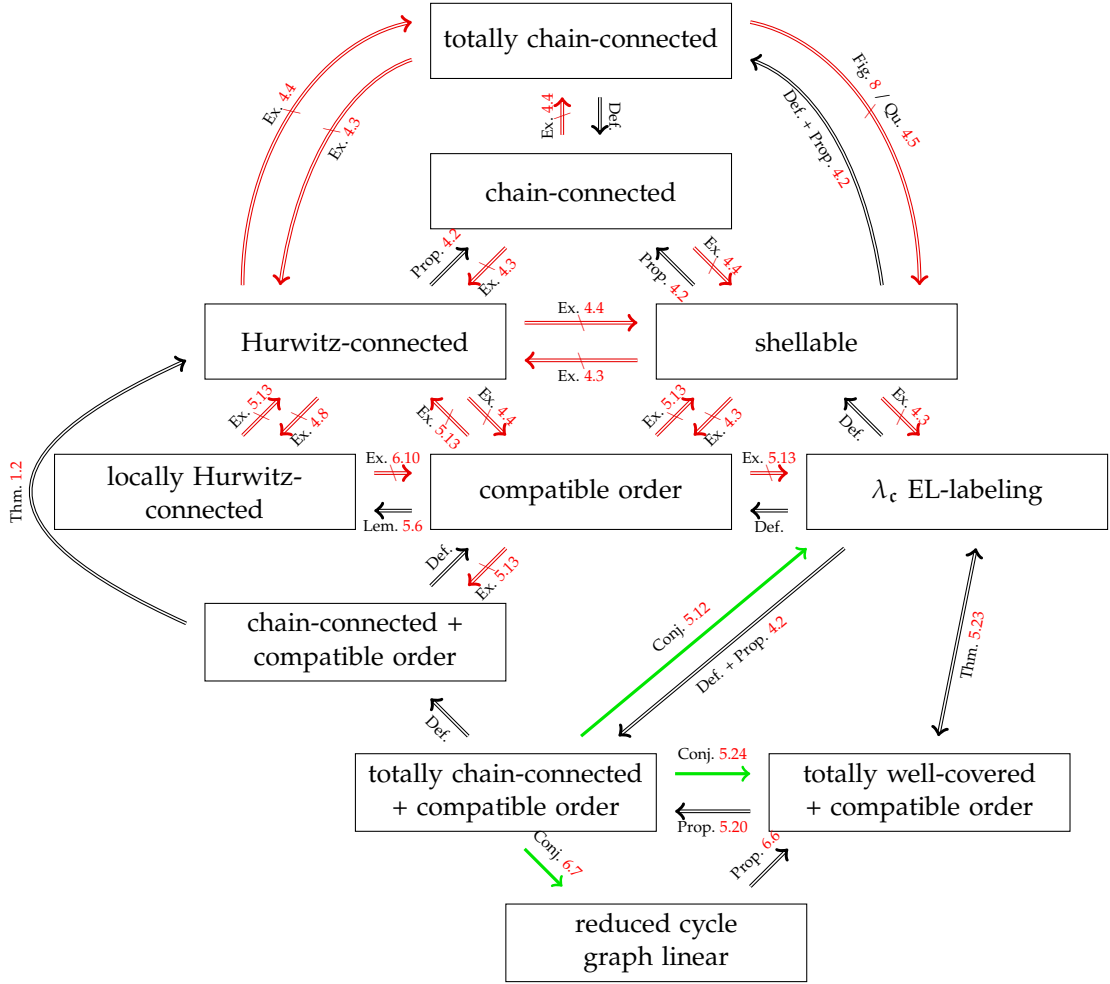

\smallskip

This article is organized as follows.  In Section~\ref{sec:connectivity} we formally define chain-connectivity, Hurwitz-connectivity and shellability.  In the process we recall the necessary background and define the needed concepts.  In Section~\ref{sec:motivating_example} we introduce a huge class of factorization posets arising from reflection groups.  We briefly recall the definitions and state that these factorization posets possess all three connectivity properties.  In Section~\ref{sec:interaction} we investigate relations between our three connectivity properties without any further assumptions.  The heart of this manuscript is Section~\ref{sec:compatible_orders} in which we define the notion of a compatible order of the generators and the ``well-covered'' property.  We prove Theorem~\ref{thm:main_hurwitz} and provide an equivalent formulation of Conjecture~\ref{conj:main_shellable}.  This section culminates in the proof of Theorem~\ref{thm:main_result}.  We conclude this manuscript with Section~\ref{sec:cycle_graph} in which we define a certain graph from which we can essentially recover the factorization poset, and we use this perspective to prove a particular case of our main conjecture.

\section{Three Notions of Connectivity}
	\label{sec:connectivity}
In this section we define the three notions of connectivity that we care about.  Each of the following three subsections serves at the same time as a preliminary section that introduces further necessary concepts and notions.
	
\subsection{Poset Terminology}
	\label{sec:poset_terminology}
In this section we recall the basic concepts from the theory of partially ordered sets, and we introduce a first notion of connectivity.

	\subsubsection{Basics}
	\label{sec:basics}
A partially ordered set (\defn{poset} for short) is a set $P$ equipped with a partial order $\leq$, and we usually write $\PP=(P,\leq)$.  

If $\PP$ has a least element $\hat{0}$ and a greatest element $\hat{1}$, then it is \defn{bounded}, and the \defn{proper part} of $\PP$ is the subposet $\overline{\PP}\defs\bigl(P\setminus\{\hat{0},\hat{1}\},\leq\bigr)$.  

Two elements $x,y\in P$ form a \defn{covering} if $x<y$ and there is no $z\in P$ with $x<z<y$.  We then write $x\lessdot y$, and equivalently say that $x$ \defn{is covered by} $y$ or that $y$ \defn{covers} $x$.  Let us define the set of coverings of $\PP$ by 
\begin{align}\label{eq:order_diagram}
	\EE(\PP) \defs \bigl\{(x,y)\in P\times P\mid x\lessdot y\bigr\}.
\end{align}

From now on we will only consider finite posets.  A \defn{chain} of $\PP$ is a totally ordered subset $C\subseteq P$ meaning that for every $x,y\in C$ we have $x<y$ or $y<x$.  If $C=\{x_{1},x_{2},\ldots,x_{k}\}$ with $x_{i}<x_{j}$ whenever $i<j$, we occasionally use the notation $C:x_{1}<x_{2}<\cdots<x_{k}$ to emphasize the order of the elements.  Moreover, a chain $C$ is \defn{maximal} if it is not contained properly in any other chain.  Let $\MM(\PP)$ denote the set of maximal chains of $\PP$.  A poset is \defn{graded} if all maximal chains have the same cardinality, and this common cardinality minus one is the \defn{rank} of~$\PP$, denoted by $\rk(\PP)$.

\begin{figure}
	\centering
	\begin{tikzpicture}\small
		\def\x{.75};
		\def\y{.75};
		\draw(2*\x,1*\y) node[circle,draw,fill=white,scale=.6](n1){};
		\draw(1*\x,2*\y) node[circle,draw,fill=white,scale=.6](n2){};
		\draw(3*\x,2*\y) node[circle,draw,fill=white,scale=.6](n3){};
		\draw(1*\x,3*\y) node[circle,draw,fill=white,scale=.6](n4){};
		\draw(3*\x,3*\y) node[circle,draw,fill=white,scale=.6](n5){};
		\draw(2*\x,4*\y) node[circle,draw,fill=white,scale=.6](n6){};
		\draw(n1) -- (n2);
		\draw(n1) -- (n3);
		\draw(n2) -- (n4);
		\draw(n3) -- (n5);
		\draw(n4) -- (n6);
		\draw(n5) -- (n6);
	\end{tikzpicture}
	\caption{A bounded graded poset that is not chain-connected.}
	\label{fig:bounded_disconnected}
\end{figure}
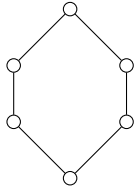

For $x,y\in P$ with $x\leq y$, the set $[x,y]\defs\{z\in P\mid x\leq z\leq y\}$ is an \defn{interval} of $\PP$.

\subsubsection{Chain-Connectivity}
	\label{sec:chain_connectivity}
The first notion of connectivity of a poset that springs to mind is the connectivity of its \defn{poset diagram}, \ie the graph $\bigl(P,\EE(\PP)\bigr)$.  Observe that this graph is trivially connected whenever $\PP$ is bounded.  However, the poset diagram of the proper part of a bounded poset need not be connected, see Figure~\ref{fig:bounded_disconnected}.  We are in fact interested in the following stronger version of connectivity.  

\begin{definition}\label{def:chain_graph}
	Let $\PP$ be a bounded graded poset, and define
	\begin{align}\label{eq:chain_connectivity}
		\II_{\text{chain}} \defs \bigl\{\{C,C'\}\mid C,C'\in\MM(\PP)\;\text{and}\;\lvert C\cap C'\rvert=\rk(\PP)\bigr\}.
	\end{align}
	The \defn{chain graph} of $\PP$ is the graph $\CC(\PP) \defs \bigl(\MM(\PP),\II_{\text{chain}}\bigr)$.
\end{definition}

In other words two maximal chains of $\PP$ are adjacent in the chain graph if they differ in exactly one element.  We call $\PP$ \defn{chain-connected} if $\CC(\PP)$ is connected.  Observe that the poset diagram of the proper part of a chain-connected poset is again connected as soon as the rank of $\PP$ is at least three. 

Moreover, if every interval of $\PP$ is chain-connected, then we call $\PP$ \defn{totally chain-connected}. It is easy to check (by induction) that a bounded graded poset is totally chain-connected if and only if the poset diagram of the proper part of every interval of rank $\geq 3$ is connected (we will not use this characterization in the following).

\subsection{Factorization Posets in Generated Groups}
	\label{sec:factorization_generated}
In this section we introduce the main construction that associates a bounded graded poset with each triple $(\grp,\gen,\top)$, where $\grp$ is a group generated as a monoid by the set $\gen\subseteq\grp$, and where~$\top$ is some element of $\grp$.

\subsubsection{Generated Groups}
	\label{sec:generated_groups}
Fix a group $\grp$ and a subset $\gen\subseteq\grp$ that generates $\grp$ as a monoid.  Let $\id$ denote the identity of $\grp$.  We then call the pair $(\grp,\gen)$ a \defn{generated group}, and we define the \defn{$\gen$-length} of $x\in\grp$ by
\begin{align}\label{eq:gen_length}
	\ell_{\gen} \defs \min\bigl\{k\in\mathbb{N}\mid x=a_{1}a_{2}\cdots a_{k},\;\text{where}\;a_{i}\in\gen\;\text{for}\;i\in[k]\bigr\},
\end{align}
where $[k]\defs\{1,2,\ldots,k\}$.  If $k=\ell_{\gen}(x)$, then any factorization $x=a_{1}a_{2}\cdots a_{k}$ with $a_{i}\in\gen$ for $i\in[k]$ is \defn{reduced}.  Let $\red{\gen}{x}$ denote the set of all reduced $\gen$-factorizations of $x\in\grp$.  In order to avoid confusion, we usually write the elements of $\red{\gen}{x}$ as tuples rather than as words over the alphabet $\gen$.  It follows immediately from the definition that $\ell_{\gen}$ satisfies the \defn{sub-additivity law}
\begin{align}\label{eq:sub_additivity}
	\ell_{\gen}(xy)\leq\ell_{\gen}(x)+\ell_{\gen}(y).
\end{align}
If $x,y\in\grp$ are such that equality holds in \eqref{eq:sub_additivity}, then we say that $x$ \defn{divides} $xy$.  In that case there exists a reduced $\gen$-factorization of $x$ that is a prefix of some reduced $\gen$-factorization of $xy$.  This gives immediately rise to the definition of the following partial order on $\grp$, the \defn{$\gen$-prefix order}:
\begin{align}\label{eq:prefix_order}
	x\leq_{\gen}y\quad\text{if and only if}\quad\ell_{\gen}(x)+\ell_{\gen}(x^{-1}y)=\ell_{\gen}(y).
\end{align}
Observe that $x$ divides $y$ if and only if $x$ lies on a geodesic from $\id$ to $y$ in the right Cayley graph of $(\grp,\gen)$.  The definition of the $\gen$-prefix order as given in \eqref{eq:prefix_order} has perhaps first appeared explicitly in \cite{brady01partial} in the case of the symmetric group, but the notion of divisibility goes back to \cite{garside69braid}.

The next lemma, which is well known to experts, describes the intrinsic recursive structure of the $\gen$-prefix order.  Its proof is essentially verbatim to the proof of \cite{bjorner05combinatorics}*{Proposition~3.1.6}, which treats a particular case.  If $x\leq_{\gen} y$, then we denote the interval of $(G,\leq_{\gen})$ generated by $x$ and $y$ by
\begin{displaymath}
	[x,y]_{\gen}\defs\{z\in G\mid x\leq_{\gen}z\leq_{\gen}y\}.
\end{displaymath}

\begin{lemma}\label{lem:bottom_intervals}
  Let  $x,y,z\in G$.
  \begin{enumerate}[(i)]
  \item If $x\leq_{\gen}y$, then the poset ${\bigl([\id,x^{-1}y]_{\gen},\leq_{\gen}\bigr)}$ is isomorphic to $\bigl([x,y]_{\gen},\leq_{\gen}\bigr)$.
  \item If $x\leq_{\gen}y\leq_{\gen}z$, then $x^{-1}y\leq_{\gen}x^{-1}z$ and the poset $\bigl([x^{-1}y,x^{-1}z],\leq_{\gen}\bigr)$ is isomorphic to $\bigl([y,z]_{\gen},\leq_{\gen}\bigr)$.
    \end{enumerate}
\end{lemma}
\begin{proof}
	For (i), We will show that the map 
	\begin{equation*}
		f_{x}\colon[\id,x^{-1}y]_{\gen}\to[x,y]_{\gen},\quad u\mapsto xu
	\end{equation*}
	is the desired poset isomorphism.
	
	Let $u\leq_{\gen}x^{-1}y$.  In view of \eqref{eq:sub_additivity} we obtain
	\begin{multline*}
		\ell_{\gen}(y) = \ell_{\gen}(x) + \ell_{\gen}(x^{-1}y) = \ell_{\gen}(x) + \ell_{\gen}(u) + \ell_{\gen}(u^{-1}x^{-1}y)\\
			\geq \ell_{\gen}(xu) + \ell_{\gen}\bigl((xu)^{-1}y\bigr) \geq \ell_{\gen}(y).
	\end{multline*}
	So both inequalities are equalities, and we obtain that $x\leq_{\gen}xu\leq_{\gen}y$, which implies that $f_{x}$ is well defined.
	
	Now, if $v\in[x,y]_{\gen}$, we obtain:
	\begin{multline*}
		\ell_{\gen}(x^{-1}v) = \ell_{\gen}(v)-\ell_{\gen}(x) = \ell_{\gen}(y)-\ell_{\gen}(v^{-1}y)-\ell_{\gen}(x)  = \ell_{\gen}(x^{-1}y)-\ell_{\gen}(v^{-1}y), 
	\end{multline*}
	which implies $x^{-1}v\leq_{\gen}x^{-1}y$.  Thus, the map
        \begin{equation*}
		g_{x}\colon[x,y]_{\gen}\to[\id,x^{-1}y]_{\gen},\quad v\mapsto x^{-1}v
	\end{equation*}
        is well defined. It is obviously the inverse of $f_x$, so $f_{x}$ is a bijection.
	
	Finally, let $g,h\in[\id,x^{-1}y]_{\gen}$.  From above we know that $xg\leq_{\gen}y$ and $xh\leq_{\gen}y$, and we obtain
	\begin{align*}
		g\leq_{\gen}h & \Longleftrightarrow \ell_{\gen}(h) = \ell_{\gen}(g) + \ell_{\gen}(g^{-1}h)\\
		& \Longleftrightarrow \ell_{\gen}(x^{-1}y) - \ell_{\gen}(h^{-1}x^{-1}y) = \ell_{\gen}(x^{-1}y) - \ell_{\gen}(g^{-1}x^{-1}y) + \ell_{\gen}(g^{-1}h)\\
		& \Longleftrightarrow \ell_{\gen}\bigl((xg)^{-1}y\bigr) = \ell_{\gen}\bigl((xh)^{-1}y\bigr) + \ell_{\gen}(g^{-1}h)\\
		& \Longleftrightarrow \ell_{\gen}(y) - \ell_{\gen}(xg) = \ell_{\gen}(y) - \ell_{\gen}(xh) + \ell_{\gen}\bigl((xg)^{-1}xh\bigr)\\
		& \Longleftrightarrow \ell_{\gen}(xh) = \ell_{\gen}(xg) + \ell_{\gen}\bigl((xg)^{-1}xh\bigr)\\
		& \Longleftrightarrow xg \leq_{\gen} xh.
	\end{align*}
	This completes the proof of statement (i).

        For (ii), note that we already proved that if $x\leq_{\gen}y\leq_{\gen}z$, then $x^{-1}y\leq_{\gen}x^{-1}z$: this is the proof above that $g_x$ is well defined (replacing $y$ with $z$ and $v$ with $y$). Now, using property (i) for different values of $x$ and $y$, we obtain that $\bigl([y,z]_{\gen},\leq_{\gen}\bigr)$ and $\bigl([x^{-1}y,x^{-1}z],\leq_{\gen}\bigr)$ are both isomorphic to ${\bigl([\id,y^{-1}z]_{\gen},\leq_{\gen}\bigr)}$. This concludes the proof of statement (ii).
\end{proof}

Let us from now on assume that $\gen$ is closed under $\grp$-conjugation.  In that case,~$\leq_{\gen}$ is in fact a subword order on $\red{\gen}{x}$.

\begin{proposition}[\cite{huang17absolute}*{Proposition~2.8}]\label{prop:subword_order}
	Let $\gen\subseteq\grp$ be a generating set closed under $\grp$-conjugation.  Let $g\in\grp$ with $\ell_{\gen}(g)=n$.  Fix $k\leq n$ and a list of integers $1\leq i_{1}<i_{2}<\cdots<i_{k}\leq n$.  For $x\in\grp$ with $\ell_{\gen}(x)=k$ the following are equivalent:
	\begin{enumerate}[(i)]
		\item $x\leq_{\gen}g$, \ie $\ell_{\gen}(g)=\ell_{\gen}(x)+\ell_{\gen}(x^{-1}g)$;
		\item there exists $(a_{1},a_{2},\ldots,a_{n})\in\red{\gen}{g}$ such that $x=a_{i_{1}}a_{i_{2}}\cdots a_{i_{k}}$.
	\end{enumerate}
\end{proposition}

For a fixed element $\top\in\grp$, we define the \defn{factorization poset} of $\top$ in $(\grp,\gen)$ by
\begin{align*}
	\PP_{\top}(\grp,\gen) \defs \bigl([\id,\top]_{\gen},\leq_{\gen}\bigr).
\end{align*}
Whenever it is clear from the context, we omit the group and the generating set.  
The three following lemmas (whose proofs are straightforward) will be used extensively in the rest of the paper. 

The maximal chains of $\PP_{\top}$ correspond bijectively to the reduced $\gen$-factorizations of $\top$.  As a consequence, the reduced $\gen$-factorizations of $\top$ completely determine $\PP_{\top}$.  To make this more precise, let us define a map
\begin{align}\label{eq:natural_labeling}
	\lambda_{\top}\colon\EE(\PP_{\top}) \to\gen,\quad (x,y)\mapsto x^{-1}y.
\end{align}

\begin{lemma}\label{lem:bijection_chains_words}
	The map $\lambda_{\top}$ extends to a bijection 
	\begin{equation*}\label{eq:chain_labeling}
		\begin{split}
			\MM(\PP_{\top})& \to \red{\gen}{\top}\\
			x_{0}\lessdot_{\gen}x_{1}\lessdot_{\gen}\cdots\lessdot_{\gen}x_n & \mapsto\bigl(x_{0}^{-1}x_{1},x_{1}^{-1}x_{2},\ldots,x_{n-1}^{-1}x_n\bigr),
		\end{split}
	\end{equation*}
	whose inverse is given by
	\begin{equation*}
		\begin{split}
			\red{\gen}{\top} &\to \MM(\PP_{\top})\\
			(a_{1},a_{2},\ldots,a_{n}) &\mapsto x_{0}\lessdot_{\gen}x_{1}\lessdot_{\gen}\cdots\lessdot_{\gen}x_n,
		\end{split}
	\end{equation*}
	where $x_{0}=\id$ and $x_{i}=a_{1}a_{2}\cdots a_{i}$ for $i\in[n]$.
\end{lemma}

The labeling from \eqref{eq:natural_labeling} is preserved under the bijection from Lemma~\ref{lem:bottom_intervals}.

\begin{lemma}\label{lem:label_preserving}
	Let $x,y,z\in[\id,\top]_{\gen}$ with $x\leq_{\gen}y\leq_{\gen}z$.  For every $g,h\in[x^{-1}y,x^{-1}z]_{\gen}$ with $g\lessdot_{\gen}h$ we have $\lambda_{\top}(g,h)=\lambda_{\top}\bigl(xg,xh\bigr)$.
\end{lemma}

The fact that $\gen$ is closed under $\grp$-conjugation implies that the isomorphism type of $\PP_{\top}$ depends only on the conjugacy class of $\top$.

\begin{lemma}\label{lem:conjugation_isomorphism}
  If $\gen$ is closed under $\grp$-conjugation, then $\ell_{\gen}$ is invariant under $\grp$-conjugation.  In other words $\ell_{\gen}(x)=\ell_{\gen}(gxg^{-1})$ for all $g,x\in\grp$.\\
  Moreover, for all $x,y,g \in \grp$, we have $x\leq_{\gen}y$ if and only if $gxg^{-1}\leq_{\gen}gyg^{-1}$.
\end{lemma}

Factorization posets are always self dual, too.

\begin{proposition}[\cite{huang17absolute}*{Proposition~2.5}]\label{prop:self_dual}
	Let $\gen$ be closed under $\grp$-conjugation.  For any $x,z\in\grp$ with $x\leq_{\gen}z$ the map 
	\begin{displaymath}
		K_{x,z}\colon\grp\to\grp,\quad y\mapsto xy^{-1}z
	\end{displaymath}
	restricts to an anti-automorphism of the interval $\bigl([x,z]_{\gen},\leq_{\gen}\bigr)$.
\end{proposition}

\begin{example}\label{ex:sym_4}
	Let $\grp=\mathfrak{S}_{4}$ be the symmetric group of permutations of $[4]$.  It is well known that $\grp$ is generated by its set of transpositions
	\begin{displaymath}
		T = \bigl\{(1\;2),(1\;3),(1\;4),(2\;3),(2\;4),(3\;4)\bigr\}.
	\end{displaymath}
	Since any transposition is an involution, $T$ generates $\mathfrak{S}_{4}$ as a monoid.  It is moreover easy to check that $T$ is closed under $\mathfrak{S}_{4}$-conjugation.  Let $\top=(1\;2\;3\;4)$ be a long cycle in $\mathfrak{S}_{4}$.  The factorization poset $\PP_{\top}(\mathfrak{S}_{4},T)$ is shown in Figure~\ref{fig:sym_4_poset}.  
	
	The reader is cordially invited to verify Lemmas~\ref{lem:bijection_chains_words} and~\ref{lem:label_preserving}.  Lemma~\ref{lem:conjugation_isomorphism} translates to this case as follows: if we replace $\top$ by any other long cycle $\top'$, then the map that adjusts the order of the letters in the cycles of the permutations occurring in Figure~\ref{fig:sym_4_poset} according to the relative order of $[4]$ in $\top'$, is a poset isomorphism.  Proposition~\ref{prop:self_dual} can be verified by rotating the poset diagram by 180 degrees.
	
	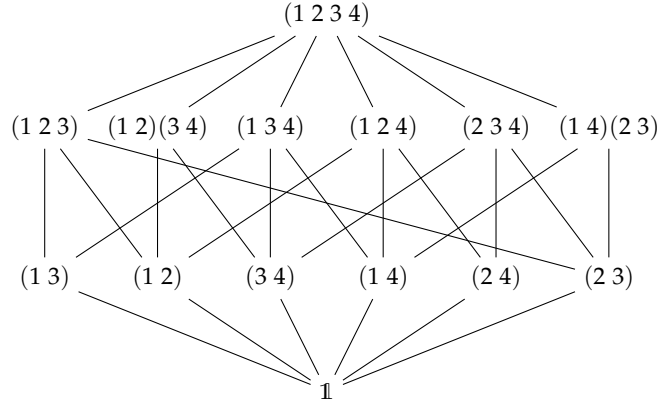
\begin{figure}
		\centering
		\begin{tikzpicture}\small
			\def\x{1.5};
			\def\y{2};
			\draw(3.5*\x,1.25*\y) node(n1){$\id$};
			\draw(1*\x,2*\y) node(n2){$(1\;3)$};
			\draw(2*\x,2*\y) node(n3){$(1\;2)$};
			\draw(3*\x,2*\y) node(n4){$(3\;4)$};
			\draw(4*\x,2*\y) node(n5){$(1\;4)$};
			\draw(5*\x,2*\y) node(n6){$(2\;4)$};
			\draw(6*\x,2*\y) node(n7){$(2\;3)$};
			\draw(1*\x,3*\y) node(n8){$(1\;2\;3)$};
			\draw(2*\x,3*\y) node(n9){$(1\;2)(3\;4)$};
			\draw(3*\x,3*\y) node(n10){$(1\;3\;4)$};
			\draw(4*\x,3*\y) node(n11){$(1\;2\;4)$};
			\draw(5*\x,3*\y) node(n12){$(2\;3\;4)$};
			\draw(6*\x,3*\y) node(n13){$(1\;4)(2\;3)$};
			\draw(3.5*\x,3.75*\y) node(n14){$(1\;2\;3\;4)$};
			\draw(n1) -- (n2);
			\draw(n1) -- (n3);
			\draw(n1) -- (n4);
			\draw(n1) -- (n5);
			\draw(n1) -- (n6);
			\draw(n1) -- (n7);
			\draw(n2) -- (n8);
			\draw(n2) -- (n10);
			\draw(n3) -- (n8);
			\draw(n3) -- (n9);
			\draw(n3) -- (n11);
			\draw(n4) -- (n9);
			\draw(n4) -- (n10);
			\draw(n4) -- (n12);
			\draw(n5) -- (n10);
			\draw(n5) -- (n11);
			\draw(n5) -- (n13);
			\draw(n6) -- (n11);
			\draw(n6) -- (n12);
			\draw(n7) -- (n8);
			\draw(n7) -- (n12);
			\draw(n7) -- (n13);
			\draw(n8) -- (n14);
			\draw(n9) -- (n14);
			\draw(n10) -- (n14);
			\draw(n11) -- (n14);
			\draw(n12) -- (n14);
			\draw(n13) -- (n14);
		\end{tikzpicture}
		\caption{The factorization poset $\PP_{\top}$ of the long cycle $\top=(1\;2\;3\;4)$ in the symmetric group $\mathfrak{S}_{4}$ generated by its transpositions.}
		\label{fig:sym_4_poset}
	\end{figure}
\end{example}

\subsubsection{The Hurwitz Action}
	\label{sec:hurwitz_action}
Perhaps the most important consequence of the assumption that $\gen$ is closed under $\grp$-conjugation is the existence of a braid group action on $\red{\gen}{x}$ (and thus in view of Lemma~\ref{lem:bijection_chains_words} also on $\MM(\PP_{\top})$).  Recall that the \defn{braid group} on $n$ strands can be defined via the group presentation
\begin{multline}\label{eq:braid_group}
	\bg_{n} \defs \bigl\langle\sigma_{1},\sigma_{2},\ldots,\sigma_{n-1}\mid \sigma_{i}\sigma_{i+1}\sigma_{i}=\sigma_{i+1}\sigma_{i}\sigma_{i+1}\;\text{for}\;i\in[n-2],\\
		\text{and}\;\sigma_{i}\sigma_{j}=\sigma_{j}\sigma_{i}\;\text{for}\;i,j\in[n-1]\;\text{with}\;\lvert i-j\rvert>1\bigr\rangle.
\end{multline}
Now fix $x\in\grp$ with $\ell_{\gen}(x)=n$.  For $i\in[n-1]$, we define an action of the braid group generator $\sigma_{i}$ on $\red{\gen}{x}$ by
\begin{equation}\label{eq:hurwitz}
  \begin{array}{crccl}
	& \sigma_{i}\cdot(a_{1},\ldots,a_{i-1}, & a_{i}, & a_{i+1}, & a_{i+2},\ldots,a_{n})\\
	\defs & (a_{1},\ldots,a_{i-1}, & a_{i+1}, & a_{i+1}^{-1}a_{i}a_{i+1}, & a_{i+2},\ldots,a_{n}).
\end{array}\end{equation}
We call such a relation a \defn{Hurwitz move}.

In other words, the generators of $\bg_{n}$ swap two consecutive factors of a reduced $\gen$-factorization of $x$ and conjugate one by the other, so that the product stays the same.  Since $\gen$ is closed under $\grp$-conjugation, $\sigma_{i}$ is indeed a map on $\red{\gen}{x}$, and it is straightforward to verify that this action respects the relations of \eqref{eq:braid_group}, and therefore extends to a group action of $\bg_{n}$ on $\red{\gen}{x}$: the \defn{Hurwitz action}.  Let us state one technical result that will be used in the last section of this paper.

\begin{lemma}\label{prop:reverse_same_orbit_size}
	Let $(\grp,\gen)$ be a generated group, where $\gen$ is closed under $\grp$-conjugation, and let $a,b\in\gen$.  Suppose that $\red{\gen}{ab}$ is finite.  Then the Hurwitz orbit of $\red{\gen}{ab}$ containing $(a,b)$ has the same size as the Hurwitz orbit of $\red{\gen}{ba}$ containing $(b,a)$. 
\end{lemma}
\begin{proof}
  Write $u:=ab$ and $v:=ba$. Note that $u$ and $v$ are conjugate, for example $v=bub^{-1}$. The map $x\mapsto bxb^{-1}$ clearly induces a bijection from $\red{\gen}{u}$ to $\red{\gen}{v}$ which is compatible with the Hurwitz action. The factorization $(a,b)$ of $u$ is sent, via this bijection, to the factorization $(bab^{-1},b)$ of $v$. Since $(b,a)$ can be obtained from $(bab^{-1},b)$ by a Hurwitz move, they are in the same Hurwitz orbit, hence the orbit of $\red{\gen}{v}$ containing $(b,a)$ is equinumerous to the orbit of $\red{\gen}{u}$ containing $(a,b)$.
\end{proof}

We can now define the second notion of connectivity used in this paper.  

\begin{definition}\label{def:hurwitz_graph}
	Let $\top\in\grp$, and define 
	\begin{align}\label{eq:hurwitz_connectivity}
		\II_{\text{hurwitz}} \defs \bigl\{\{\xb,\xb'\}\mid\xb,\xb'\in\red{\gen}{\top}\;\text{and}\;\xb'=\sigma_{i}\xb\;\text{for some}\;i\in[\ell_{\gen}(\top)-1]\bigr\}.
	\end{align}
	The \defn{Hurwitz graph} of $\top$ is the graph $\HH(\top)\defs\bigl(\red{\gen}{\top},\II_{\text{hurwitz}}\bigr)$.
\end{definition}

In view of Lemma~\ref{lem:bijection_chains_words} we may as well define the Hurwitz graph of $\top$ as a graph on the maximal chains of $\PP_{\top}$, and from this point of view it is clearly (isomorphic to) a subgraph of $\CC(\PP_{\top})$ (see Definition~\ref{def:chain_graph}).  We call $\PP_{\top}$ \defn{Hurwitz-connected} if $\HH(\top)$ is connected.  This is the case if and only if the braid group $\bg_{\ell_{\gen}(\top)}$ acts transitively on $\red{\gen}{\top}$.  In view of Lemma~\ref{lem:bijection_chains_words} we sometimes abuse notation and write $\HH(\PP_{\top})$ instead of $\HH(\top)$.  Figure~\ref{fig:sym_4_hurwitz_graph} shows the Hurwitz graph of the factorization poset from Figure~\ref{fig:sym_4_poset}.

\begin{figure}
	\centering
	\begin{tikzpicture}\small
		\def\x{2.5};
		\def\y{2.5};
		\draw(1*\x,2*\y) node(v1){$(14)(13)(12)$};
		\draw(2*\x,2*\y) node(v2){$(13)(34)(12)$};
		\draw(2*\x,1*\y) node(v3){$(34)(14)(12)$};
		\draw(2.5*\x,1.8*\y) node(v4){$(34)(12)(24)$};
		\draw(3*\x,1*\y) node(v5){$(34)(24)(14)$};
		\draw(3*\x,2*\y) node(v6){$(24)(23)(14)$};
		\draw(4*\x,2*\y) node(v7){$(23)(34)(14)$};
		\draw(3.2*\x,2.5*\y) node(v8){$(23)(14)(13)$};
		\draw(4*\x,3*\y) node(v9){$(23)(13)(34)$};
		\draw(3*\x,3*\y) node(v10){$(13)(12)(34)$};
		\draw(3*\x,4*\y) node(v11){$(12)(23)(34)$};
		\draw(2.5*\x,3.2*\y) node(v12){$(12)(34)(24)$};
		\draw(2*\x,4*\y) node(v13){$(12)(24)(23)$};
		\draw(2*\x,3*\y) node(v14){$(24)(14)(23)$};
		\draw(1*\x,3*\y) node(v15){$(14)(12)(23)$};
		\draw(1.8*\x,2.5*\y) node(v16){$(14)(23)(13)$};
		\draw(v1) -- (v2) -- (v3) -- (v4) -- (v5) -- (v6) -- (v7) -- (v8) -- (v9) -- (v10) -- (v11) -- (v12) -- (v13) -- (v14) -- (v15) -- (v16) -- (v1);
		\draw(v1) -- (v3) -- (v5) -- (v7) -- (v9) -- (v11) -- (v13) -- (v15) -- (v1);
		\draw(v2) -- (v10);
		\draw(v4) -- (v12);
		\draw(v6) -- (v14);
		\draw(v8) -- (v16);
	\end{tikzpicture}
	\caption{The Hurwitz graph of the long cycle $\top=(1\;2\;3\;4)$ in the symmetric group $\mathfrak{S}_{4}$ generated by its transpositions.}
	\label{fig:sym_4_hurwitz_graph}
\end{figure}
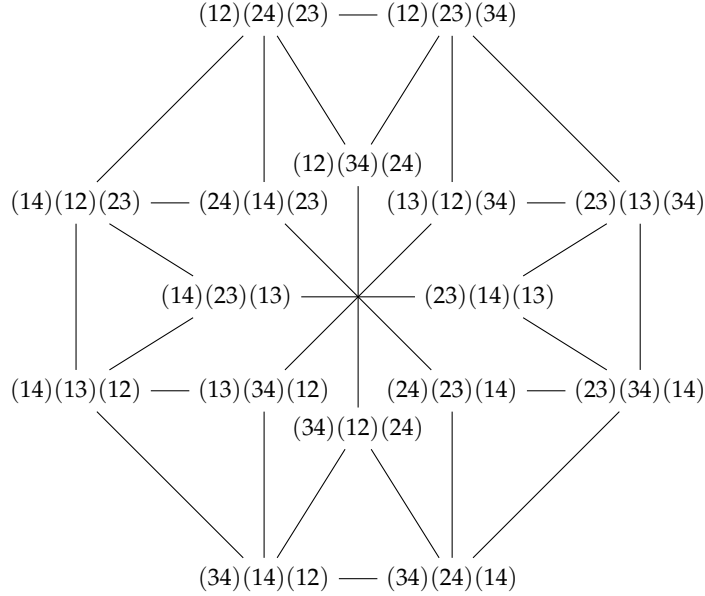

\subsection{Shellability of Posets}
	\label{sec:shellability}
The last notion of connectivity that will be important for this article has its origins in algebraic topology.  Recall that the set of chains of a graded poset $\PP$ forms a simplicial complex: the \defn{order complex} of $\PP$, denoted by $\Delta(\PP)$.  It is a consequence of Hall's Theorem \cite{stanley11enumerative}*{Proposition~3.8.5} that if~$\PP$ is graded and bounded, then the \defn{M{\"o}bius invariant} of $\PP$, \ie the value of the M{\"o}bius function of $\PP$ between least and greatest element, equals the reduced Euler characteristic of $\Delta(\overline{\PP})$.  Consequently, the combinatorics of $\PP$ provides some information on the topology of $\Delta(\overline{\PP})$.  

\subsubsection{Ordinary Shellability}
	\label{sec:ordinary_shellability}
A class of pure simplicial complexes with a particularly nice homotopy type are the \defn{shellable} simplicial complexes: their homotopy type is in fact that of a wedge of spheres~\cite{bjorner96shellable}*{Theorem~4.1}, the corresponding (co-)homology groups are torsion-free, and the Stanley-Reisner ring of such complexes is Cohen-Macaulay~\cite{bjorner80shellable}*{Appendix}.  

We phrase the definition of shellability directly in terms of a bounded graded poset $\PP$.  It can be transferred to pure simplicial complexes via the correspondence between maximal chains of $\PP$ and facets of $\Delta(\overline{\PP})$. 

\begin{definition}\label{def:poset_shellability}
	Let $\PP$ be a bounded graded poset.  A \defn{shelling} of $\PP$ is a linear order $\prec$ on $\MM(\PP)$ such that whenever two maximal chains $M,M'\in\MM(\PP)$ satisfy $M\prec M'$, then there exists $N\in\MM(\PP)$ with $N\prec M'$ and $x\in M'$ with the property that 
	\begin{equation}\label{eq:shelling}
		M\cap M'\subseteq N\cap M'=M'\setminus\{x\}.
	\end{equation}
\end{definition}

A poset that admits a shelling is \defn{shellable}.  We observe that a poset which is not chain-connected cannot be shellable, because then any linear order on $\MM(\PP)$ has a first occurence of two successive chains that lie in different connected components of $\CC(\PP)$, and these two chains forbid such an order to be a shelling of $\PP$ (see details in Proposition~\ref{prop:chain_connectivity_necessary}).  Moreover, every bounded poset of rank $\leq 2$ is shellable, and a bounded graded poset of rank $3$ is shellable if and only if its proper part is connected.  Bearing this in mind, we can view shellability as a sophisticated notion of connectivity.  

\subsubsection{Lexicographic Shellability}
	\label{sec:lexicographic_shellability}
There is a nice combinatorial way to establish shellability, by exhibiting a particular edge-labeling of the poset.  An \defn{edge-labeling} of $\PP$ is a map $\lambda\colon\EE(\PP)\to\Lambda$, where $\Lambda$ is an arbitrary partially ordered set.  An edge-labeling of $\PP$ naturally extends to a labeling of $\MM(\PP)$, where for $C:\hat{0}=x_{0}\lessdot x_{1}\lessdot\cdots\lessdot x_{n}=\hat{1}$ we set 
\begin{displaymath}
	\lambda(C)=\bigl(\lambda(x_{0},x_{1}),\lambda(x_{1},x_{2}),\ldots,\lambda(x_{n-1},x_{n})\bigr).
\end{displaymath}
A maximal chain $C\in\MM(\PP)$ is \defn{rising} if $\lambda(C)$ is weakly increasing with respect to the partial order on $\Lambda$.  A chain $C\in\MM(\PP)$ \defn{precedes} a chain $C'\in\MM(\PP)$ if $\lambda(C)$ is lexicographically smaller than $\lambda(C')$ with respect to the order on $\Lambda$.  An edge-labeling $\lambda$ of $\PP$ is an \defn{EL-labeling} if in every interval of $\PP$ there exists a unique rising maximal chain, and this chain precedes every other maximal chain in that interval.  A poset that admits an EL-labeling is \defn{EL-shellable}.  We have the following result.

\begin{theorem}[\cite{bjorner80shellable}*{Theorem~2.3}]\label{thm:el_shellable}
	Every EL-shellable poset is shellable.
\end{theorem}

In particular, if $\PP$ is an EL-shellable poset, then the lexicographic order on the set $\bigl\{\lambda(C)\mid C\in\MM(\PP)\bigr\}$ induces a shelling of $\PP$.  The converse of Theorem~\ref{thm:el_shellable} is not true, see for instance~\cites{vince85shellable,walker85poset}.  

Factorization posets coming from a generated group $(\grp,\gen)$ are naturally equip\-ped with the edge-labeling $\lambda_{\top}$ defined in \eqref{eq:natural_labeling}.  One way to establish shellability of $\PP_{\top}$ is thus to find a suitable linear order on $\gen$ such that $\lambda_{\top}$ is an EL-labeling.

\begin{example}\label{ex:free_abelian_3} 
	Let $\grp$ be the free abelian group of rank $3$ (isomorphic to $\mathbb{Z}^3$), generated by three pairwise commuting elements $r,s,t$.  Fix the element $\top=rst$.  The corresponding factorization poset is the boolean lattice shown in Figure~\ref{fig:free_abelian_3_poset}, and the corresponding chain graph is shown in Figure~\ref{fig:free_abelian_3_graph}.  Observe that this graph is isomorphic to the Hurwitz graph of $\top$.  Fix the linear order $r\prec s\prec t$.  Then, the reduced factorizations of $\top$ are (in lexicographic order):
	\begin{displaymath}
		(r,s,t) \prec (r,t,s) \prec (s,r,t) \prec (s,t,r) \prec (t,r,s) \prec (t,s,r).
	\end{displaymath}
	Observe that $(r,s,t)$ is the unique rising reduced factorization of $\top$.  In view of Lemma~\ref{lem:bijection_chains_words} this sequence of reduced factorizations corresponds to the following order on $\MM(\PP_{\top})$:
	\begin{displaymath}
		\{\id,r,rs,\top\} \prec \{\id,r,rt,\top\} \prec \{\id,s,rs,\top\} \prec \{\id,s,st,\top\} \prec \{\id,t,rt,\top\} \prec \{\id,t,st,\top\}.
	\end{displaymath}
	It is straightforward to check that this is a shelling of $\PP_{\top}$.  
	
	\begin{figure}
		\centering
		\begin{subfigure}[t]{.4\textwidth}
			\centering
			\begin{tikzpicture}\small
				\def\x{1};
				\def\y{1};
				\draw(2*\x,1*\y) node(n1){$\id$};
				\draw(1*\x,2*\y) node(n2){$r$};
				\draw(2*\x,2*\y) node(n3){$s$};
				\draw(3*\x,2*\y) node(n4){$t$};
				\draw(1*\x,3*\y) node(n5){$rs$};
				\draw(2*\x,3*\y) node(n6){$rt$};
				\draw(3*\x,3*\y) node(n7){$st$};
				\draw(2*\x,4*\y) node(n8){$rst$};
				\draw(n1) -- (n2);
				\draw(n1) -- (n3);
				\draw(n1) -- (n4);
				\draw(n2) -- (n5);
				\draw(n2) -- (n6);
				\draw(n3) -- (n5);
				\draw(n3) -- (n7);
				\draw(n4) -- (n6);
				\draw(n4) -- (n7);
				\draw(n5) -- (n8);
				\draw(n6) -- (n8);
				\draw(n7) -- (n8);
			\end{tikzpicture}
			\caption{The factorization poset $\PP_{rst}$ in the free abelian group generated by $\{r,s,t\}$.}
			\label{fig:free_abelian_3_poset}
		\end{subfigure}
		\hspace*{.1cm}
		\begin{subfigure}[t]{.4\textwidth}
			\centering
			\begin{tikzpicture}\small
				\def\x{1};
				\def\y{1};
				\draw(2*\x,1*\y) node(n1){$\{\id,r,rs,rst\}$};
				\draw(1*\x,2*\y) node(n2){$\{\id,s,rs,rst\}$};
				\draw(1*\x,3*\y) node(n3){$\{\id,s,st,rst\}$};
				\draw(2*\x,4*\y) node(n4){$\{\id,t,st,rst\}$};
				\draw(3*\x,3*\y) node(n5){$\{\id,t,rt,rst\}$};
				\draw(3*\x,2*\y) node(n6){$\{\id,r,rt,rst\}$};
				\draw(n1) -- (n2);
				\draw(n2) -- (n3);
				\draw(n3) -- (n4);
				\draw(n4) -- (n5);
				\draw(n5) -- (n6);
				\draw(n6) -- (n1);
			\end{tikzpicture}
			\caption{The chain graph of the factorization poset in Figure~\ref{fig:free_abelian_3_poset}.}
			\label{fig:free_abelian_3_graph}
		\end{subfigure}
		\caption{A factorization poset and its chain graph in the free abelian group on three generators.}
		\label{fig:free_abelian_3}
	\end{figure}
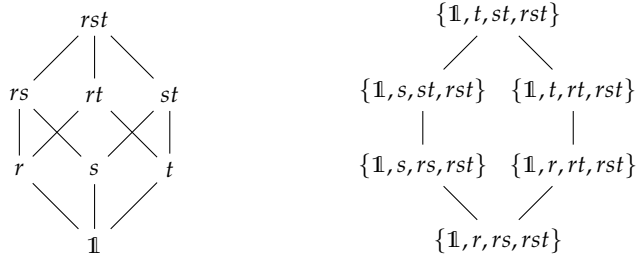
\end{example}

\section{The Motivating Example}
	\label{sec:motivating_example}
Perhaps the earliest occurrence of a factorization poset coming from a generated group is in \cite{brady01partial}.  That article deals with the symmetric group generated by its set of transpositions.  One of the main results of that article is that the corresponding factorization poset of a long ordered cycle is a lattice, which is in fact isomorphic to the lattice of noncrossing set partitions studied by Kreweras~\cite{kreweras72sur}.  Essentially the same isomorphism (phrased in a slightly different language) was described a few years earlier by Biane in \cite{biane97some}.  We have encountered this example in Example~\ref{ex:sym_4}.

At about the same time, other researchers have constructed similar posets coming from other reflection groups, see for instance \cites{reiner97non,bessis03dual}.  Nowadays, all of these constructions can be seen as instances of the following uniform construction.  

Fix a finite-dimensional complex vector space $V$ and consider the group $\optfrak{U}(V)$ of unitary transformations on $V$.  An element $t\in\optfrak{U}(V)$ is a \defn{reflection} if it has finite order and fixes a hyperplane pointwise, called the \defn{reflection hyperplane} of $t$.  Any subgroup of $\optfrak{U}(V)$ that is generated by reflections is a \defn{(complex) reflection group}.  If $\optfrak{W}$ is a reflection group and $T$ is its set of reflections, then $(\optfrak{W},T)$ is---naturally---a generated group.  It is easy to see that $T$ is closed under $\optfrak{W}$-conjugation.  For more background on reflection groups, we refer the interested reader to \cites{humphreys90reflection,lehrer09unitary}.

An element $w\in\optfrak{W}$ is \defn{regular} if it has an eigenvector in the complement of the reflection hyperplanes of $\optfrak{W}$.  A \defn{Coxeter element} is a regular element of some particular order.  We do not go into further detail here, and refer the interested reader to \cite{reiner17on} instead.  It is a consequence of \cite{lehrer99reflection}*{Theorem~C} that Coxeter elements exist in the case when $\optfrak{W}$ is both \defn{irreducible} (\ie $\optfrak{W}$ does not stabilize a proper subspace of $V$ other than $\{0\}$) and \defn{well-generated} (\ie the minimal number of reflections needed to generate $\optfrak{W}$ is $\dim V$).  

Let $\optfrak{W}$ be a finite irreducible well-generated complex reflection group, $T$ the set of reflections of $\optfrak{W}$, and $\top$ a Coxeter element of $\optfrak{W}$.  The factorization poset $\PP_{\top}(\optfrak{W},T)$ is the \defn{lattice of $\top$-noncrossing $\optfrak{W}$-partitions}, as a reference to the prototypical example of the lattice of noncrossing set partitions that arises when $\optfrak{W}$ is the symmetric group.  There exists a vast literature on the study of these posets, and we refer the reader for instance to \cites{armstrong09generalized,bessis03dual,bessis15finite,brady08non,reiner97non} and all the references given therein.

The next two results show that lattices of noncrossing partitions possess all of our three types of connectivity.

\begin{theorem}[\cites{deligne74letter,bessis15finite}]\label{thm:nc_hurwitz}
	Let $\optfrak{W}$ be a finite irreducible well-generated complex reflection group, let $T$ be its set of reflections, and let $\top\in\optfrak{W}$ be a Coxeter element.  The lattice $\PP_{\top}(\optfrak{W},T)$ of $\top$-noncrossing $\optfrak{W}$-partitions is Hurwitz-connected, and thus chain-connected.
\end{theorem}

\begin{theorem}[\cites{athanasiadis07shellability,muehle15el}]\label{thm:nc_shellable}
	Let $\optfrak{W}$ be a finite irreducible well-generated complex reflection group, let $T$ be its set of reflections, and let $\top\in\optfrak{W}$ be a Coxeter element.  The lattice $\PP_{\top}(\optfrak{W},T)$ of $\top$-noncrossing $\optfrak{W}$-partitions is shellable.  More precisely, the map $\lambda_{\top}$ defined in \eqref{eq:natural_labeling} is an EL-labeling of $\PP_{\top}(\optfrak{W},T)$ for a certain linear order of $T$.
\end{theorem}

To date, however, no uniform proofs of Theorems~\ref{thm:nc_hurwitz} and \ref{thm:nc_shellable} are available.  By a uniform proof, we mean a proof that does not rely on the classification of complex reflection groups.  Uniform proofs are known when $\optfrak{W}$ is a real reflection group and are the main results of \cite{deligne74letter} and \cite{athanasiadis07shellability}, respectively.  For the remaining complex reflection groups both Hurwitz-connectivity and shellability have been verified case by case.  

One of the main goals of our work is the creation of a uniform framework with which we can essentially verify both properties by the same means: a particular linear order of the chosen generating set.  Such a linear order---tailored to the case of reflection groups---plays a crucial role in \cite{athanasiadis07shellability} and \cite{muehle15el}, and can indeed be seen as a precursor to one of the main definitions of this article, Definition~\ref{def:compatible_order} below.

\section{Interaction of Different Types of Connectivity}
\label{sec:interaction}

In this section we want to investigate the implications between the three types of connectivity of a factorization poset $\PP_{\top}(\grp,\gen)$.

\begin{assumption}\label{disc:finiteness}
	From now on we assume that the factorization poset $\PP_{\top}(\grp,\gen)$ is finite. 
\end{assumption}

Note that this is equivalent to the finiteness of $\red{\gen}{\top}$, and also to the finiteness of $\gen_\top$ (the set of generators that are in $\PP_\top$).  We have already mentioned the following easy observation.

\begin{proposition}\label{prop:chain_connectivity_necessary}
	Every factorization poset that is Hurwitz-connected is also chain-connected.  Every shellable bounded graded poset is chain-connected.
\end{proposition}
\begin{proof}
	Let $\PP_{\top}(\grp,\gen)$ be a factorization poset. Lemma~\ref{lem:bijection_chains_words} states that the sets $\red{\gen}{\top}$ and $\MM\bigl(\PP_{\top}(\grp,\gen)\bigr)$ are in bijection, and this bijection identifies $\HH(\top)$ as a subgraph of $\CC\bigl(\PP_{\top}(\grp,\gen)\bigr)$.  Hence, if $\HH(\top)$ is connected, so is $\CC\bigl(\PP_{\top}(\grp,\gen)\bigr)$. 

	Now let $\PP$ be a bounded graded poset, and let $\bigl\lvert\MM(\PP)\bigr\rvert=s$.  Suppose that we can label the maximal chains of $\PP$ such that $M_{1}\prec M_{2}\prec\cdots\prec M_{s}$ is a shelling of $\PP$.  We prove by induction that for every $i\in[s]$ the set $\{M_{1},M_{2},\ldots,M_{i}\}$ induces a connected subgraph of $\CC(\PP)$.  The base case $i=1$ holds trivially, since the corresponding subgraph consists of a single vertex.  Now assume that the subgraph of $\CC(\PP)$ induced by $\{M_{1},M_{2},\ldots,M_{i}\}$ is connected, and consider $M_{i+1}$.  Since $i\geq 1$ it follows that $M_{1}\prec M_{i+1} $, and since $\prec$ is a shelling we can find some $j\leq i$ and some $x\in M_{i+1}$ such that $M_{1}\cap M_{i+1}\subseteq M_{j}\cap M_{i+1}=M_{i+1}\setminus\{x\}$.  By Definition~\ref{def:chain_graph} there is an edge connecting $M_{j}$ and $M_{i+1}$, which implies that $\{M_{1},M_{2},\ldots,M_{i+1}\}$ induces a connected subgraph of $\CC(\PP)$.  Consequently, $\CC(\PP)$ is connected.
\end{proof}

Neither of the converse statements in Proposition~\ref{prop:chain_connectivity_necessary} is true without further assumptions as the next examples illustrate.

\begin{example}\label{ex:chain_nonhurwitz}
	For the converse of the first statement consider for instance the finite group given by the presentation 
	\begin{multline}\label{eq:chain_nonhurwitz}
		\grp = \bigl\langle r,s,t,u\mid r^{2}=s^{2}=t^{2}=u^{2}=\id,rt=tr,su=us,\\
			rs=st=tu=ur,sr=ru=ut=ts\bigr\rangle.
	\end{multline}
	(This is in fact the dihedral group of order $8$ in the so-called dual Coxeter presentation.  See \cite{bessis03dual} for more background on related groups and presentations.)  The set $\gen=\{r,s,t,u\}$ is clearly closed under $\grp$-conjugation.  Now take the element $\top=rt$.  We have $\red{\gen}{\top}=\bigl\{(r,t),(t,r),(s,u),(u,s)\bigr\}$, which implies that $\PP_{\top}(\grp,\gen)$ is not Hurwitz-connected.  Since $\PP_{\top}(\grp,\gen)$ has rank two, however, it is trivially chain-connected and shellable.  See Figure~\ref{fig:chain_nonhurwitz} for an illustration.
	
	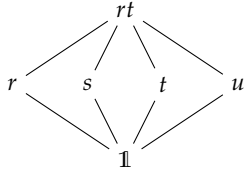
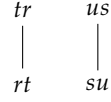
\begin{figure}
		\centering
		\begin{subfigure}[t]{.48\textwidth}
			\centering
			\begin{tikzpicture}\small
				\def\x{1};
				\def\y{1};
				\draw(2.5*\x,1*\y) node(n1){$\id$};
				\draw(1*\x,2*\y) node(n2){$r$};
				\draw(2*\x,2*\y) node(n3){$s$};
				\draw(3*\x,2*\y) node(n4){$t$};
				\draw(4*\x,2*\y) node(n5){$u$};
				\draw(2.5*\x,3*\y) node(n6){$rt$};
				\draw(n1) -- (n2);
				\draw(n1) -- (n3);
				\draw(n1) -- (n4);
				\draw(n1) -- (n5);
				\draw(n2) -- (n6);
				\draw(n3) -- (n6);
				\draw(n4) -- (n6);
				\draw(n5) -- (n6);
			\end{tikzpicture}
			\caption{The factorization poset $\PP_{rt}(\grp,\gen)$, where $\grp$ is given by the presentation in \eqref{eq:chain_nonhurwitz}.}
			\label{fig:chain_nonhurwitz_poset}
		\end{subfigure}
		\hspace*{.1\textwidth}
		\begin{subfigure}[t]{.4\textwidth}
			\centering
			\begin{tikzpicture}\small
				\def\x{1};
				\def\y{1};
				\def\s{1};
				\draw(1*\x,1*\y) node[scale=\s](n1){$rt$};
				\draw(1*\x,2*\y) node[scale=\s](n2){$tr$};
				\draw(2*\x,1*\y) node[scale=\s](n3){$su$};
				\draw(2*\x,2*\y) node[scale=\s](n4){$us$};
				\draw(n1) -- (n2);
				\draw(n3) -- (n4);
			\end{tikzpicture}
			\caption{The Hurwitz graph of the poset in Figure~\ref{fig:chain_nonhurwitz_poset}.}
			\label{fig:chain_nonhurwitz_graph}
		\end{subfigure}
		\caption{An example of a chain-connected factorization poset that is not Hurwitz-connected.}
		\label{fig:chain_nonhurwitz}
	\end{figure}
\end{example}

\begin{example}\label{ex:hurwitz_nonshellable}
	Next consider the infinite group given by the presentation
	\begin{multline}\label{eq:hurwitz_nonshellable}
		\grp = \bigl\langle r,s,t,u,v\mid r^{3}=s^{3}, t^{2}=u^{2}=v^{2}, rs=sr, tu=uv=vt, ut=tv=vu,\\
			rt=ts=sv=vr,rv=vs=su=ur,ru=us=st=tr\bigr\rangle.
	\end{multline}
	The set $\gen=\{r,s,t,u,v\}$ is closed under $\grp$-conjugation.  The factorization poset of $\top=rrrt$ is shown in Figure~\ref{fig:hurwitz_nonshellable_poset}, and the corresponding Hurwitz graph is depicted in Figure~\ref{fig:hurwitz_nonshellable_graph}.  By inspection of these figures we see that $\PP_{\top}(\grp,\gen)$ is Hurwitz-connected, but not shellable, since the subposet $\PP_{rrr}(\grp,\gen)$ is not chain-connected.

	\begin{figure}
		\centering
		\begin{subfigure}[t]{.48\textwidth}
			\centering
			\begin{tikzpicture}\small
				\def\x{1.1};
				\def\y{1.1};
				\draw(3.5*\x,1*\y) node(n1){$\id$};
				\draw(1.5*\x,2*\y) node(n2){$r$};
				\draw(2.5*\x,2*\y) node(n3){$v$};
				\draw(3.5*\x,2*\y) node(n4){$t$};
				\draw(4.5*\x,2*\y) node(n5){$u$};
				\draw(5.5*\x,2*\y) node(n6){$s$};
				\draw(1*\x,3*\y) node(n7){$rr$};
				\draw(2*\x,3*\y) node(n8){$rt$};
				\draw(3*\x,3*\y) node(n9){$rv$};
				\draw(4*\x,3*\y) node(n10){$rs$};
				\draw(5*\x,3*\y) node(n11){$ru$};
				\draw(6*\x,3*\y) node(n12){$ss$};
				\draw(1.5*\x,4*\y) node(n13){$rrt$};
				\draw(2.5*\x,4*\y) node(n14){$rrs$};
				\draw(3.5*\x,4*\y) node(n15){$rrr$};
				\draw(4.5*\x,4*\y) node(n16){$rss$};
				\draw(5.5*\x,4*\y) node(n17){$rrv$};
				\draw(3.5*\x,5*\y) node(n18){$rrrt$};
				\draw(n1) -- (n2);
				\draw(n1) -- (n3);
				\draw(n1) -- (n4);
				\draw(n1) -- (n5);
				\draw(n1) -- (n6);
				\draw(n2) -- (n7);
				\draw(n2) -- (n8);
				\draw(n2) -- (n9);
				\draw(n2) -- (n10);
				\draw(n2) -- (n11);
				\draw(n3) -- (n8);
				\draw(n3) -- (n9);
				\draw(n4) -- (n8);
				\draw(n4) -- (n11);
				\draw(n5) -- (n9);
				\draw(n5) -- (n11);
				\draw(n6) -- (n8);
				\draw(n6) -- (n9);
				\draw(n6) -- (n10);
				\draw(n6) -- (n11);
				\draw(n6) -- (n12);
				\draw(n7) -- (n13);
				\draw(n7) -- (n14);
				\draw(n7) -- (n15);
				\draw(n7) -- (n17);
				\draw(n8) -- (n13);
				\draw(n9) -- (n13);
				\draw(n9) -- (n17);
				\draw(n10) -- (n13);
				\draw(n10) -- (n14);
				\draw(n10) -- (n16);
				\draw(n10) -- (n17);
				\draw(n11) -- (n17);
				\draw(n12) -- (n13);
				\draw(n12) -- (n15);
				\draw(n12) -- (n16);
				\draw(n12) -- (n17);
				\draw(n13) -- (n18);
				\draw(n14) -- (n18);
				\draw(n15) -- (n18);
				\draw(n16) -- (n18);
				\draw(n17) -- (n18);
			\end{tikzpicture}
			\caption{The factorization poset $\PP_{rrrt}(\grp,\gen)$, where $\grp$ is given by the presentation in \eqref{eq:hurwitz_nonshellable}.}
			\label{fig:hurwitz_nonshellable_poset}
		\end{subfigure}
		\hspace*{.75cm}
		\begin{subfigure}[t]{.4\textwidth}
			\centering
			\begin{tikzpicture}\small
				\def\x{.55};
				\def\y{.55};
				\def\s{.67};
				\draw(3*\x,1*\y) node[scale=\s](n1){$rrts$};
				\draw(8*\x,1*\y) node[scale=\s](n2){$rtss$};
				\draw(1.5*\x,1.5*\y) node[scale=\s](n3){$rrrt$};
				\draw(9.5*\x,1.5*\y) node[scale=\s](n4){$tsss$};
				\draw(4*\x,2*\y) node[scale=\s](n5){$rvrs$};
				\draw(7*\x,2*\y) node[scale=\s](n6){$rsvs$};
				\draw(2.5*\x,2.5*\y) node[scale=\s](n7){$rrsv$};
				\draw(8.5*\x,2.5*\y) node[scale=\s](n8){$vrss$};
				\draw(1*\x,3*\y) node[scale=\s](n9){$rrvr$};
				\draw(10*\x,3*\y) node[scale=\s](n10){$svss$};
				\draw(3.5*\x,3.5*\y) node[scale=\s](n11){$rsrv$};
				\draw(7.5*\x,3.5*\y) node[scale=\s](n12){$vsrs$};
				\draw(2*\x,4*\y) node[scale=\s](n13){$rvsr$};
				\draw(9*\x,4*\y) node[scale=\s](n14){$srvs$};
				\draw(5.5*\x,4.5*\y) node[scale=\s](n15){$vssr$};
				\draw(4.5*\x,5.5*\y) node[scale=\s](n16){$srrv$};
				\draw(6.5*\x,5.5*\y) node[scale=\s](n17){$rssu$};
				\draw(5.5*\x,6.5*\y) node[scale=\s](n18){$urrs$};
				\draw(3.5*\x,7.5*\y) node[scale=\s](n19){$ursr$};
				\draw(7.5*\x,7.5*\y) node[scale=\s](n20){$srsu$};
				\draw(2*\x,7*\y) node[scale=\s](n21){$rsur$};
				\draw(9*\x,7*\y) node[scale=\s](n22){$surs$};
				\draw(1*\x,8*\y) node[scale=\s](n23){$rurr$};
				\draw(10*\x,8*\y) node[scale=\s](n24){$ssus$};
				\draw(2.5*\x,8.5*\y) node[scale=\s](n25){$usrr$};
				\draw(8.5*\x,8.5*\y) node[scale=\s](n26){$ssru$};
				\draw(4*\x,9*\y) node[scale=\s](n27){$srur$};
				\draw(7*\x,9*\y) node[scale=\s](n28){$susr$};
				\draw(1.5*\x,9.5*\y) node[scale=\s](n29){$trrr$};
				\draw(9.5*\x,9.5*\y) node[scale=\s](n30){$ssst$};
				\draw(3*\x,10*\y) node[scale=\s](n31){$strr$};
				\draw(8*\x,10*\y) node[scale=\s](n32){$sstr$};
				\draw(n3) -- (n1) -- (n2) -- (n4) -- (n10) -- (n24) -- (n30) -- (n32) -- (n31) -- (n29) -- (n23) -- (n9) -- (n13) -- (n21) -- (n27) -- (n28) -- (n22) -- (n14) -- (n6) -- (n5) -- (n18) -- (n19) -- (n25);
				\draw(n3) -- (n9) -- (n7) -- (n11) -- (n21) -- (n23) -- (n25) -- (n31) -- (n27) -- (n20) -- (n26) -- (n32) -- (n28) -- (n19) -- (n13)  -- (n15) -- (n12) -- (n8) -- (n10) -- (n14) -- (n16);
				\draw(n13) -- (n5) -- (n12) -- (n22) -- (n24) -- (n26);
				\draw(n8) -- (n2) -- (n6) -- (n17) -- (n20) -- (n14);
				\draw(n6) -- (n11) -- (n16) -- (n27);
				\draw(n18) -- (n22);
				\draw(n17) -- (n21);
				\draw(n1) -- (n7);
				\draw(n15) -- (n28);
			\end{tikzpicture}
			\caption{The Hurwitz graph of the poset in Figure~\ref{fig:hurwitz_nonshellable_poset} (the loops are not drawn).}
			\label{fig:hurwitz_nonshellable_graph}
		\end{subfigure}
		\caption{An example of a Hurwitz-connected factorization poset that is not shellable, because it contains an interval of rank $3$ that is not chain-connected.}
		\label{fig:hurwitz_nonshellable}
	\end{figure}
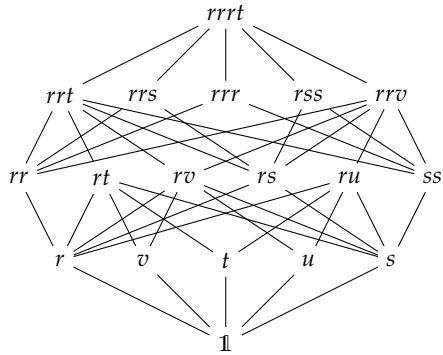
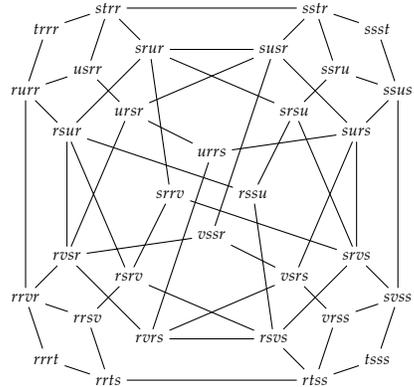
\end{example}

We observe that the example given in Figure~\ref{fig:hurwitz_nonshellable_poset} contains an interval which is not chain-connected, and this is the reason why it is not shellable.  But what happens if we exclude this situation, \ie if we assume that our factorization poset is totally chain-connected?  We are not aware of a factorization poset that is totally chain-connected, but not shellable.  

\begin{question}\label{qu:chain_connected_nonshellable}
	Does there exist a generated group $(\grp,\gen)$ and some $\top\in\grp$ such that $\red{\gen}{c}$ is finite and $\PP_{\top}(\grp,\gen)$ is totally chain-connected but not shellable?
\end{question}

An answer to Question~\ref{qu:chain_connected_nonshellable} would be of great importance within the framework presented here: we could either reduce the difficulty to prove that a factorization poset is shellable, or the group structure of such an example would exhibit a new obstruction to shellability.

Note that, for arbitrary graded posets, there are some well-known examples of totally chain-connected posets which are not shellable.  Consider for instance the poset $\PP$ in Figure~\ref{fig:dunce_hat}, which is reproduced from~\cite{bjorner82introduction}*{Page~16}.  The geometric realization of $\Delta(\overline{\PP})$ is the Dunce Hat, which is known to be non-shellable~\cite{hachimori08decompositions}*{Theorem~3}.  However, it can be verified that every interval of $\PP$ is chain-connected.  On the other hand, $\PP$ is not self dual and in view of Proposition~\ref{prop:self_dual} it cannot arise as a factorization poset in some generated group.

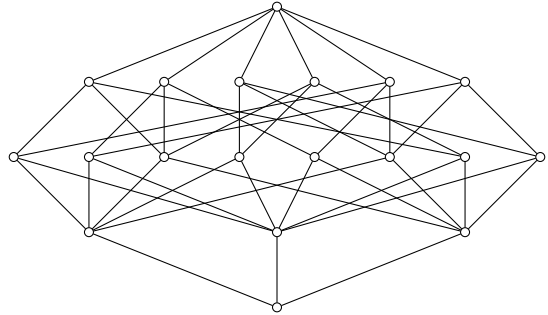
\begin{figure}
	\centering
	\begin{tikzpicture}\small
		\def\x{1};
		\def\y{1};
		\draw(4.5*\x,1*\y) node[circle,draw,fill=white,scale=.4](n1){};
		\draw(2*\x,2*\y) node[circle,draw,fill=white,scale=.4](n2){};
		\draw(4.5*\x,2*\y) node[circle,draw,fill=white,scale=.4](n3){};
		\draw(7*\x,2*\y) node[circle,draw,fill=white,scale=.4](n4){};
		\draw(2*\x,3*\y) node[circle,draw,fill=white,scale=.4](n5){};
		\draw(1*\x,3*\y) node[circle,draw,fill=white,scale=.4](n6){};
		\draw(4*\x,3*\y) node[circle,draw,fill=white,scale=.4](n7){};
		\draw(3*\x,3*\y) node[circle,draw,fill=white,scale=.4](n8){};
		\draw(6*\x,3*\y) node[circle,draw,fill=white,scale=.4](n9){};
		\draw(5*\x,3*\y) node[circle,draw,fill=white,scale=.4](n10){};
		\draw(7*\x,3*\y) node[circle,draw,fill=white,scale=.4](n11){};
		\draw(8*\x,3*\y) node[circle,draw,fill=white,scale=.4](n12){};
		\draw(3*\x,4*\y) node[circle,draw,fill=white,scale=.4](n13){};
		\draw(2*\x,4*\y) node[circle,draw,fill=white,scale=.4](n14){};
		\draw(4*\x,4*\y) node[circle,draw,fill=white,scale=.4](n15){};
		\draw(7*\x,4*\y) node[circle,draw,fill=white,scale=.4](n16){};
		\draw(6*\x,4*\y) node[circle,draw,fill=white,scale=.4](n17){};
		\draw(5*\x,4*\y) node[circle,draw,fill=white,scale=.4](n18){};
		\draw(4.5*\x,5*\y) node[circle,draw,fill=white,scale=.4](n19){};
		\draw(n1) -- (n2);
		\draw(n1) -- (n3);
		\draw(n1) -- (n4);
		\draw(n2) -- (n5);
		\draw(n2) -- (n6);
		\draw(n2) -- (n7);
		\draw(n2) -- (n8);
		\draw(n2) -- (n9);
		\draw(n3) -- (n5);
		\draw(n3) -- (n6);
		\draw(n3) -- (n7);
		\draw(n3) -- (n10);
		\draw(n3) -- (n11);
		\draw(n3) -- (n12);
		\draw(n4) -- (n8);
		\draw(n4) -- (n9);
		\draw(n4) -- (n10);
		\draw(n4) -- (n11);
		\draw(n4) -- (n12);
		\draw(n5) -- (n13);
		\draw(n5) -- (n16);
		\draw(n6) -- (n14);
		\draw(n6) -- (n17);
		\draw(n7) -- (n15);
		\draw(n7) -- (n18);
		\draw(n8) -- (n13);
		\draw(n8) -- (n14);
		\draw(n8) -- (n18);
		\draw(n9) -- (n15);
		\draw(n9) -- (n16);
		\draw(n9) -- (n17);
		\draw(n10) -- (n13);
		\draw(n10) -- (n17);
		\draw(n11) -- (n14);
		\draw(n11) -- (n18);
		\draw(n12) -- (n15);
		\draw(n12) -- (n16);
		\draw(n13) -- (n19);
		\draw(n14) -- (n19);
		\draw(n15) -- (n19);
		\draw(n16) -- (n19);
		\draw(n17) -- (n19);
		\draw(n18) -- (n19);
	\end{tikzpicture}
	\caption{A totally chain-connected poset that is not shellable.}
	\label{fig:dunce_hat}
\end{figure}

\medskip

We have seen in Example~\ref{ex:chain_nonhurwitz} that chain-connectivity of a factorization poset does not imply Hurwitz-connectivity.  However, we may add the following local criterion to make things work.  

\begin{definition}\label{def:locally_hurwitz}
	Let $\PP_{\top}(\grp,\gen)$ be factorization poset.  If $\bg_{2}$ acts transitively on $\red{\gen}{g}$ for every $g\leq_{\gen} c$ with $\ell_{\gen}(g)=2$, then we call $\PP_{\top}(\grp,\gen)$ \defn{locally Hurwitz-connected}.
\end{definition}

\begin{theorem}\label{thm:hurwitz_transitive}
	Let $\gen$ be closed under $\grp$-conjugation, and fix $\top\in\grp$ with $\ell_{\gen}(\top)=n$.  If $\PP_{\top}$ is chain-connected and locally Hurwitz-connected, then $\bg_n$ acts transitively on $\red{\gen}{\top}$, i.e., $\PP_{\top}$ is Hurwitz-connected.
\end{theorem}
\begin{proof}
  Since $\PP_{\top}$ is chain-connected, the chain graph of $\PP_{\top}$ is connected. Thus, by transitivity, it is sufficient to prove that any two chains that are neighbors in the chain graph, are connected in the Hurwitz graph (recall that in view of the canonical bijection of Lemma~\ref{lem:bijection_chains_words}, we can work on factorizations or on maximal chains equivalently).

  Consider $C$ and $C'$ two maximal chains that are neighbors in the chain graph. Say that $C$ corresponds to the reduced $\gen$-factorization $\xb=(a_{1},a_{2},\dots, a_{n})$, and $C'$ to the reduced $\gen$-factorization $\xb'=(a_{1}',a_{2}',\dots, a_{n}')$. The chains $C$ and $C'$ are the same except for one element, so there exists $i \in [n-1]$ such that $a_k=a_k'$ for any $k\neq i,i+1$, and  $a_ia_{i+1}=a_i' a_{i+1}'$. Denote by $g$ the element $a_i a_{i+1}$, which has length $2$ by construction. It is also a factor of a reduced expression of $\top$, so $g\leq_{\gen}\top$ (see Proposition \ref{prop:subword_order}). Because $\PP_{\top}$ is locally Hurwitz-connected, there exists $\omega_2\in\bg_{2}$ with $\omega_2\cdot(a_{i},a_{i+1})=(a_{i}',a_{i+1}')$.  This braid can be lifted to an element $\omega\in\bg_{n}$ in a natural way: the $i$-th and $i+1$-st strands are braided in $\omega$ like the first and second in $\omega_2$, and the other strands are not braided. By construction it satisfies $\omega\cdot\xb=\xb'$.  Since the argument holds for any two neighboring chains in the chain graph, it follows that the Hurwitz graph is also connected.
\end{proof}

The next example illustrates that being locally Hurwitz-connected is actually not a necessary condition for the Hurwitz-connectivity of $\PP_{c}$.  

\begin{example}\label{ex:hurwitz_nonlocal}
	Consider the infinite group given by the presentation
	\begin{multline}\label{eq:hurwitz_nonlocal}
		\grp = \bigl\langle r,s,t,u\mid r^{2}=s^{2},t^{2}=u^{2},rs=sr,tu=ut,\\
			rt=ts=su=ur,st=tr=ru=us\bigr\rangle.
	\end{multline}
	The set $\gen=\{r,s,t,u\}$ is closed under $\grp$-conjugation.  Figure~\ref{fig:hurwitz_nonlocal_poset} shows the factorization poset of $\top=rrt$, and the corresponding Hurwitz graph is depicted in Figure~\ref{fig:hurwitz_nonlocal_graph}.  We see that $\PP_{rrt}(\grp,\gen)$ is Hurwitz-connected, but the interval $\PP_{rr}(\grp,\gen)$ is not.

	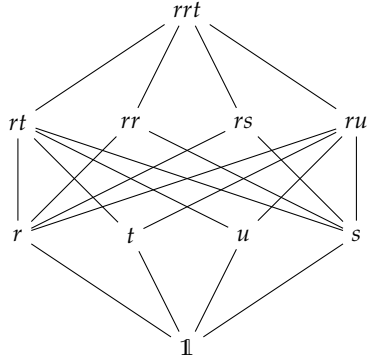
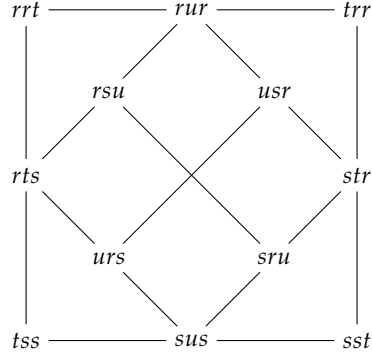
\begin{figure}
		\centering
		\begin{subfigure}[t]{.48\textwidth}
			\centering
			\begin{tikzpicture}\small
				\def\x{1.5};
				\def\y{1.5};
				\draw(2.5*\x,1*\y) node(n1){$\id$};
				\draw(1*\x,2*\y) node(n2){$r$};
				\draw(2*\x,2*\y) node(n3){$t$};
				\draw(3*\x,2*\y) node(n4){$u$};
				\draw(4*\x,2*\y) node(n5){$s$};
				\draw(1*\x,3*\y) node(n6){$rt$};
				\draw(2*\x,3*\y) node(n7){$rr$};
				\draw(3*\x,3*\y) node(n8){$rs$};
				\draw(4*\x,3*\y) node(n9){$ru$};
				\draw(2.5*\x,4*\y) node(n10){$rrt$};
				\draw(n1) -- (n2);
				\draw(n1) -- (n3);
				\draw(n1) -- (n4);
				\draw(n1) -- (n5);
				\draw(n2) -- (n6);
				\draw(n2) -- (n7);
				\draw(n2) -- (n8);
				\draw(n2) -- (n9);
				\draw(n3) -- (n6);
				\draw(n3) -- (n9);
				\draw(n4) -- (n6);
				\draw(n4) -- (n9);
				\draw(n5) -- (n6);
				\draw(n5) -- (n7);
				\draw(n5) -- (n8);
				\draw(n5) -- (n9);
				\draw(n6) -- (n10);
				\draw(n7) -- (n10);
				\draw(n8) -- (n10);
				\draw(n9) -- (n10);
			\end{tikzpicture}
			\caption{The factorization poset $\PP_{rrt}(\grp,\gen)$, where the prefix order on the group given by the presentation in \eqref{eq:hurwitz_nonlocal}.}
			\label{fig:hurwitz_nonlocal_poset}
		\end{subfigure}
		\hspace*{.1\textwidth}
		\begin{subfigure}[t]{.4\textwidth}
			\centering
			\begin{tikzpicture}\small
				\def\x{1.1};
				\def\y{1.1};
				\draw(1*\x,1*\y) node(n1){$tss$};
				\draw(3*\x,1*\y) node(n2){$sus$};
				\draw(5*\x,1*\y) node(n3){$sst$};
				\draw(2*\x,2*\y) node(n4){$urs$};
				\draw(4*\x,2*\y) node(n5){$sru$};
				\draw(1*\x,3*\y) node(n6){$rts$};
				\draw(5*\x,3*\y) node(n7){$str$};
				\draw(2*\x,4*\y) node(n8){$rsu$};
				\draw(4*\x,4*\y) node(n9){$usr$};
				\draw(1*\x,5*\y) node(n10){$rrt$};
				\draw(3*\x,5*\y) node(n11){$rur$};
				\draw(5*\x,5*\y) node(n12){$trr$};
				\draw(n1) -- (n2) -- (n3) -- (n7) -- (n12) -- (n11) -- (n10) -- (n6) -- (n1);
				\draw(n2) -- (n5) -- (n7) -- (n9) -- (n11) -- (n8) -- (n6) -- (n4) -- (n2);
				\draw(n4) -- (n9);
				\draw(n5) -- (n8);
			\end{tikzpicture}
			\caption{The Hurwitz graph of the poset in Figure~\ref{fig:hurwitz_nonlocal_poset} (the loops are not drawn).}
			\label{fig:hurwitz_nonlocal_graph}
		\end{subfigure}
		\caption{An example of a Hurwitz-connected factorization poset which has a rank-$2$ interval that is not Hurwitz-connected.}
		\label{fig:hurwitz_nonlocal}
	\end{figure}
\end{example}

\section{Compatible $\gen$-Orders}
	\label{sec:compatible_orders}
In this section we introduce our main tool: a linear order of $\gen$ that is compatible with $\top\in\grp$.  This concept is an algebraic generalization of the compatible reflection order introduced in \cite{athanasiadis07shellability}, and it also appeared in \cite{muehle15el} in the context of complex reflection groups.  We recall Assumption~\ref{disc:finiteness}: $\red{\gen}{\top}$ is assumed finite.

\subsection{Definition and Properties of Compatible Orders}
	\label{sec:def_compatible_orders}
Recall that $\gen_{\top}=\{a\in\gen\mid a\leq_{\gen}\top\}$.  Observe that we trivially have $\red{\gen}{\top}=\red{\gen_{\top}}{\top}$.

Let $\prec$ be any linear order on $\gen_{\top}$.  We say that a factorization $(a_{1},a_{2},\ldots,a_{\ell_{\gen}(\top)})\in\red{\gen}{\top}$ is \defn{$\prec$-rising} if $a_{i}\preceq a_{i+1}$ for all $i\in[\ell_{\gen}(\top)-1]$. We denote by $\rise(\top;\prec)$ the number of $\prec$-rising reduced $\gen$-factorizations of $\top$ for a given linear order $\prec$ on~$\gen_{\top}$.

\subsubsection{Hurwitz Orbits and Rising Factorizations}

The next statement relates these rising factorizations to the Hurwitz orbits of $\red{\gen}{\top}$, in the specific case when $\ell_{\gen}(\top)=2$.

\begin{proposition}\label{prop:compatible_order_rank_2}
	Let $\top\in\grp$ have $\ell_{\gen}(\top)=2$.  Let $\orb(\top)$ denote the number of Hurwitz orbits of $\red{\gen}{\top}$.  Then:
	\begin{displaymath}
		\orb(\top) = \min \bigl\{ \rise(\top;\prec) \mid \;\prec \text{ is a linear order on } \gen_{\top} \bigr\}.
	\end{displaymath}
	In particular, there exists a linear order $\prec$ of $\gen_{\top}$ such that the number of Hurwitz orbits of $\red{\gen}{\top}$ is equal to the number of $\prec$-rising reduced $\gen$-factorizations of $\top$.
\end{proposition}
\begin{proof}
	Any Hurwitz orbit of $\red{\gen}{\top}$ is by assumption finite, and therefore has the form
	\begin{equation}\label{eq:hurwitz_orbit}
		\top = a_{1}a_{2} = a_{2}a_{3} = \cdots = a_{p-1}a_{p} = a_{p}a_{1}.
	\end{equation}
	For any linear order $\prec$ on $\gen_{\top}$, at least one of the factorizations in \eqref{eq:hurwitz_orbit} is $\prec$-rising, because otherwise we would obtain the contradiction $a_{1}\succ a_{2}\succ\cdots\succ a_{p}\succ a_{1}$.  We thus obtain the inequality $\orb(\top)\leq\rise(\top;\prec)$ for any order $\prec$.
	
	It remains to show that equality can be achieved for some order $\prec$.  Since $\ell_{\gen}(\top)=2$, any $a\in\gen_{\top}$ appears in exactly one Hurwitz orbit of $\red{\gen}{\top}$.  We can therefore write $\gen_{\top}$ as the disjoint union of sets of the form $\bigl\{a_{1}^{(i)},a_{2}^{(i)},\ldots,a_{p_{i}}^{(i)}\bigr\}$, where $i$ ranges over the number of Hurwitz orbits of $\red{\gen}{\top}$.  Moreover, for any $i$ and any $k\in[p_{i}]$ we have $\top=a_{k}^{(i)}a_{k+1}^{(i)}$ (where we understand $p_{i}+1=1$).  We can thus consider a linear order $\prec$ on $\gen_{\top}$ satisfying $a_{1}^{(i)}\succ a_{2}^{(i)}\succ\cdots\succ a_{p_{i}}^{(i)}$ for all $i$, and we obtain $\rise(\top;\prec)=\orb(\top)$.
\end{proof}

If $\ell_{\gen}(\top)>2$, it is still true that in any Hurwitz orbit, there is at least one rising factorization, and hence, that the number of orbits is at most the number of rising factorizations (this is a consequence of Lemma~\ref{lem:lexicographically_smallest_rising}). However, Proposition~\ref{prop:compatible_order_rank_2} is not true in general, as it is not guaranteed that we can find a total order $\prec$ on $\gen_{\top}$ such that $\rise(\top;\prec)$ equals the number of Hurwitz orbits of $\red{\gen}{\top}$, as shown in the following example.

\begin{example}\label{ex:no_compatible_order}
	Consider the group $\grp$ from Example~\ref{ex:hurwitz_nonshellable} again, then we can check (by computer) that any linear order $\prec$ on $\gen=\{r,s,t,u,v\}$ produces at least two $\prec$-rising maximal chains in $\PP_{rrrt}$.  (In fact $\rise(rrrt;\prec)$ ranges between two and six.)  We have, however, already seen that the Hurwitz graph $\HH(rrrt)$ is connected; see Figure~\ref{fig:hurwitz_nonshellable_graph}.

	We note in the same example that it is not possible to find a linear order on $\gen$ such that for every $g\leq_{\gen}\top$ with $\ell_{\gen}(g)=2$ there is a unique $\prec$-rising reduced $\gen$-factorization of $g$.  (Observe that $\bg_{2}$ acts transitively on $\red{\gen}{g}$ for any such $g$.)  Take for instance $g=rt$.  We find $\red{\gen}{g}=\{rt,ts,sv,vr\}$.  If we suppose that $\prec$ is chosen in such a way that exactly one of the elements of $\red{\gen}{g}$ is $\prec$-rising, then there are four possibilities:
	\begin{itemize}
		\item $r\prec t$ implies $r\prec v\prec s\prec t$, which means that the two factorizations $(r,v)$ and $(v,s)$ in $\red{\gen}{rv}$ are $\prec$-rising;
		\item $t\prec s$ implies $t\prec r\prec v\prec s$, which means that the two factorizations $(r,v)$ and $(v,s)$ in $\red{\gen}{rv}$ are $\prec$-rising;
		\item $s\prec v$ implies $s\prec t\prec r\prec v$, which means that the two factorizations $(t,r)$ and $(s,t)$ in $\red{\gen}{ru}$ are $\prec$-rising;
		\item $v\prec r$ implies $v\prec s\prec t\prec r$, which means that the two factorizations $(t,r)$ and $(s,t)$ in $\red{\gen}{ru}$ are $\prec$-rising.
	\end{itemize}
\end{example}

This brings us to the main definition of this section.

\begin{definition}\label{def:compatible_order}
	A linear order $\prec$ on $\gen_{\top}$ is \defn{$\top$-compatible} if for any $g\leq_{\gen}\top$ with $\ell_{\gen}(g)=2$ there exists a unique $\prec$-rising reduced $\gen$-factorization of $g$.
\end{definition}

We have the following, immediate property.

\begin{lemma}\label{lem:compatible_restriction}
	If $\prec$ is a $\top$-compatible order of $\gen_{\top}$, then for any $g\leq_{\gen}\top$ the restriction $\prec_{g}$ of $\prec$ to $A_{g}$ is $g$-compatible.
\end{lemma}
\begin{proof}
	This follows from the fact that any element of length two which lies below $g$ also lies below $\top$.
\end{proof}

As a direct consequence of Proposition~\ref{prop:compatible_order_rank_2}, we obtain the following connection to the Hurwitz action in the case of an element of length 2.

\begin{corollary}\label{cor:equiv_compatible_loc_hurwitz}
	If $\ell_{\gen}(\top)=2$, then there exists a $\top$-compatible order of $\gen_{\top}$ if and only if $\PP_{\top}(\grp,\gen)$ is Hurwitz-connected.
\end{corollary}

We have seen in Example~\ref{ex:no_compatible_order} that this equivalence breaks down as soon as $\ell_{\gen}(\top)>2$.  We still have the following implication, though.

\begin{lemma}\label{lem:compatible_rank_2_hurwitz}
	If there exists a $\top$-compatible order of $\gen_{\top}$, then $\PP_{\top}(\grp,\gen)$ is locally Hurwitz-connected.
\end{lemma}
\begin{proof}
	Let $g\leq_{\gen}\top$ with $\ell_{\gen}(g)=2$.  If $\prec$ is $\top$-compatible, then its restriction to $\gen_g$ is $g$-compatible by Lemma~\ref{lem:compatible_restriction} and Corollary~\ref{cor:equiv_compatible_loc_hurwitz} implies that $\red{\gen}{g}$ has a unique Hurwitz orbit.
\end{proof}

\begin{example}
	If $\grp$ is abelian, then for every $\top\in\grp$, every linear order of $\gen_{\top}$ is $\top$-compatible.
\end{example}

\begin{example}\label{ex:sym_4_ctd_1}
	Let us continue Example~\ref{ex:sym_4}.  Let $\prec$ be the lexicographic order on the set $T$ of transpositions of $\mathfrak{S}_{4}$, \ie
	\begin{displaymath}
		(1\;2)\prec(1\;3)\prec(1\;4)\prec(2\;3)\prec(2\;4)\prec(3\;4).
	\end{displaymath}
	For $\top=(1\;2\;3\;4)$, the following table lists the elements of length two in $\PP_{\top}$ together with their sets of reduced $T$-factorizations.  Only the first factorizations per line are $\prec$-rising.\\
	\begin{center}\begin{tabular}{c||c}
		$g\leq_{T}\top$ & $\red{T}{g}$\\
		\hline\hline
		$(1\;2\;3)$ & $\bigl\{\bigl((1\;2),(2\;3)\bigr),\bigl((2\;3),(1\;3)\bigr),\bigl((1\;3),(1\;2)\bigr)\bigr\}$\\
		$(1\;2\;4)$ & $\bigl\{\bigl((1\;2),(2\;4)\bigr),\bigl((2\;4),(1\;4)\bigr),\bigl((1\;4),(1\;2)\bigr)\bigr\}$\\
		$(1\;3\;4)$ & $\bigl\{\bigl((1\;3),(3\;4)\bigr),\bigl((3\;4),(1\;4)\bigr),\bigl((1\;4),(1\;3)\bigr)\bigr\}$\\
		$(2\;3\;4)$ & $\bigl\{\bigl((2\;3),(3\;4)\bigr),\bigl((3\;4),(2\;4)\bigr),\bigl((2\;4),(2\;3)\bigr)\bigr\}$\\
		$(1\;2)(3\;4)$ & $\{\bigl((1\;2),(3\;4)\bigr),\bigl((3\;4),(1\;2)\bigr)\}$\\
		$(1\;4)(2\;3)$ & $\{\bigl((1\;4),(2\;3)\bigr),\bigl((2\;3),(1\;4)\bigr)\}$\\
	\end{tabular}\end{center}
	We thus conclude that $\prec$ is $\top$-compatible.  On the other hand, if we consider the following linear order
	\begin{displaymath}
		(1\;3)\prec'(1\;2)\prec'(1\;4)\prec'(2\;3)\prec'(2\;4)\prec'(3\;4),
	\end{displaymath}
	then we observe that $(1\;2\;3)$ has two $\prec'$-rising reduced $T$-factorizations, namely $\bigl((1\;2),(2\;3)\bigr)$ and $\bigl((1\;3),(1\;2)\bigr)$.
\end{example}

\subsubsection{Further Properties of Compatible Orders}

Let us collect a few more properties of $\top$-compatible orders.

\begin{lemma}\label{lem:compatible_cyclic}
	Let $A_{\top}=\{a_{1},a_{2},\ldots,a_{N}\}$, and suppose that $a_{1}\prec a_{2}\prec\cdots\prec a_{N}$ is $\top$-compatible.  For any $t\in[N]$ the order $a_{t}\prec a_{t+1}\prec\cdots\prec a_{N}\prec a_{1}\prec a_{2}\prec\cdots\prec a_{t-1}$ is $\top$-compatible.
\end{lemma}
\begin{proof}
	It suffices to consider the case $t=2$, the other cases follow then by repeated application.  It is immediate that we only need to consider those $g\leq_{\gen}\top$ of length two for which $a_{1}\in\gen_{g}$.  (For all other $g$ of length two, the restriction of $\prec$ to $\gen_{g}$ is not affected by the cyclic shift of the indices.)
	
	Say that $g$ is such an element, and $g=a_{1}a_{i_{s}}=a_{i_{2}}a_{1}=a_{i_{3}}a_{i_{2}}=\cdots=a_{i_{s}}a_{i_{s-1}}$ for some $s\geq 2$ and some $1<i_{2}<i_{3}<\cdots<i_{s}\leq N$.  It follows that $(a_{1},a_{i_{s}})$ is the unique $\prec$-rising reduced $\gen$-factorization of $g$ before the shift, and is no longer $\prec$-rising after the shift.  Moreover,  $(a_{i_{r}},a_{i_{r-1}})$ is not $\prec$-rising before and after the shift whenever $r\in\{3,4,\ldots,s\}$.  Finally $a_{i_{2}}a_{1}$ is not $\prec$-rising before the shift, but it is $\prec$-rising after the shift.  
	
	We conclude that the number of $\prec$-rising reduced $\gen$-factorizations of $g$ does not change under cyclically shifting the order $\prec$.
\end{proof}

\begin{lemma}\label{lem:compatible_min_max}
	Let $\prec$ be a $\top$-compatible order of $\gen_{\top}$, and let $\prec_{g}$ denote the restriction of $\prec$ to $\gen_{g}$ for some $g\leq_{\gen}\top$ with $\ell_{\gen}(g)=2$.  If $(a,b)$ is the unique $\prec_{g}$-rising reduced $\gen$-factorization of $g$, then $a$ is minimal and $b$ is maximal with respect to $\prec_{g}$. 
\end{lemma}
\begin{proof}
	By definition and Lemma~\ref{lem:compatible_restriction} there exists a unique $\prec_{g}$-rising $\gen$-factorization of $g$, say $(a,b)$.  Let $a_{\min}=\min\gen_{g}$, and $a_{\max}=\max\gen_{g}$.  Proposition~\ref{prop:subword_order} implies that we can write $g=a_{\min}x$ for some $x\in\gen_{g}$, and by minimality we find $a_{\min}\prec x$.  So $(a_{\min},x)$ is a rising factorization, and by uniqueness we get $a_{\min}=a$.  Analogously we can write $g=a_{\max}y$; and using Hurwitz action this can be rewritten $g=y' a_{\max}$ for some $y'\in\gen_{g}$.  By maximality we find $y'\prec a_{\max}$, which implies $a_{\max}=b$.
\end{proof}

We are now in the position to prove Theorem~\ref{thm:main_hurwitz}.

\begin{proof}[Proof of Theorem~\ref{thm:main_hurwitz}]
	Suppose that $\PP_{\top}$ is chain-connected and admits a $\top$-compatible order $\prec$ of $\gen_{\top}$.  Then Lemma~\ref{lem:compatible_rank_2_hurwitz} implies that $\PP_{\top}$ is locally Hurwitz-connected.  Theorem~\ref{thm:hurwitz_transitive} now implies that $\bg_{\ell_{\gen}}(\top)$ acts transitively on $\red{\gen}{\top}$.
\end{proof}

However, Example~\ref{ex:no_compatible_order} shows that there are cases where $\red{\gen}{\top}$ is Hurwitz-connected, but there does not exist a $\top$-compatible order of $\gen_{\top}$.  

\subsubsection{Compatible Orders and Lexicographic Shellability}

In this section we provide some evidence that $\top$-compatible orders are also closely related to the shellability of the factorization poset $\PP_{\top}$.  For any linear order on $\gen_{\top}$ we can consider the corresponding lexicographic order on $\red{\gen}{\top}$, which is itself a linear order.  

\begin{lemma}\label{lem:lexicographically_smallest_rising}
	Let $\prec$ be a linear order of $\gen_{\top}$ and denote by $\prec_{\text{lex}}$ the corresponding lexicographic order on $\red{\gen}{\top}$. Fix a Hurwitz orbit $\Omega$ inside $\red{\gen}{\top}$. Then the $\prec_{\text{lex}}$-smallest factorization among all the factorizations in $\Omega$ is $\prec$-rising.
\end{lemma}
\begin{proof}
	Let $(a_{1},a_{2},\ldots,a_{k})\in\red{\gen}{c}$ be a reduced $\gen$-factorization in the orbit $\Omega$.  If it is not $\prec$-rising, then there must be some $i\in[k-1]$ such that $a_{i+1}\prec a_{i}$.  It follows that
	\begin{displaymath}
		\sigma_{i}\cdot(a_{1},a_{2},\ldots,a_{k}) = (a_{1},a_{2},\ldots,a_{i-1},a_{i+1},a_{i+1}^{-1}a_{i}a_{i+1},a_{i+2},\ldots,a_{k})
	\end{displaymath}
	is a lexicographically smaller reduced $\gen$-factorization, and by construction it is also in the orbit $\Omega$.  The claim follows then by contraposition.
\end{proof}

It is a consequence of Lemma~\ref{lem:lexicographically_smallest_rising} that for every linear order $\prec$ of $\gen_{\top}$ every interval in $\PP_{\top}$ has at least one $\prec$-rising maximal chain with respect to the labeling $\lambda_{\top}$ defined in \eqref{eq:natural_labeling}.  We conjecture that there is exactly one $\prec$-rising maximal chain per interval if and only if $\prec$ is $\top$-compatible.

\begin{conjecture}\label{conj:compatible_el_shellable}
	The natural labeling $\lambda_{\top}$ from \eqref{eq:natural_labeling} is an EL-labeling of $\PP_{\top}$ with respect to some linear order $\prec$ of $\gen_{\top}$ if and only if $\PP_{\top}$ is totally chain-connected and $\prec$ is $\top$-compatible.
\end{conjecture}

We remark that one direction of Conjecture~\ref{conj:compatible_el_shellable} is trivially true.  If $\lambda_{\top}$ is an EL-labeling of $\PP_{\top}$ with respect to $\prec$, then every interval of $\PP_{\top}$ is shellable, and by Proposition~\ref{prop:chain_connectivity_necessary} chain-connected.  Since every rank-$2$ interval of $\PP_{\top}$ has a unique rising chain, it follows in view of Lemma~\ref{lem:bijection_chains_words} that there exists a unique $\prec$-rising reduced $\gen$-factorization for any element of $\PP_{\top}$ that has length $2$.  Hence $\prec$ is $\top$-compatible.

Conjecture~\ref{conj:compatible_el_shellable} does, however, \emph{not} suggest that $\PP_{\top}$ can only be EL-shellable if there exists a $\top$-compatible order of $\gen_{\top}$.  If there is no $\top$-compatible order of $\gen_{\top}$, we may only conclude that $\lambda_{\top}$ is not an EL-labeling (there may exist others, though).
Consider for instance the factorization poset from Example~\ref{ex:chain_nonhurwitz} again.  Since it is of rank $2$ and not Hurwitz-connected, Corollary~\ref{cor:equiv_compatible_loc_hurwitz} implies that it does not admit a $\top$-compatible order.  However, the labeling which assigns the label sequence $(1,2)$ to one maximal chain, and the label sequence $(2,1)$ to the remaining maximal chains is clearly an EL-labeling.

The next example shows that we cannot drop the assumption of total chain-connectivity in Conjecture~\ref{conj:compatible_el_shellable}.  

\begin{example}\label{ex:abelian_quotient_disconnected}
	Let $\grp$ be the quotient of the free abelian group on six generators $r,s,t,u,v,w$ given by the relation $rst=uvw$.  The factorization poset $\PP_{rst}$ is shown in Figure~\ref{fig:abelian_quotient_disconnected_poset}; its chain graph, depicted in Figure~\ref{fig:abelian_quotient_disconnected_graph}, has two connected components.  Since the generators all commute, any linear order on $\{r,s,t,u,v,w\}$ is $rst$-compatible, but we always find exactly two rising maximal chains.  

	\begin{figure}
		\centering
		\begin{subfigure}[t]{.48\textwidth}
			\centering
			\begin{tikzpicture}\small
				\def\x{1};
				\def\y{1};
				\draw(3.5*\x,1*\y) node(n1){$e$};
				\draw(1*\x,2*\y) node(n2){$r$};
				\draw(2*\x,2*\y) node(n3){$s$};
				\draw(3*\x,2*\y) node(n4){$t$};
				\draw(4*\x,2*\y) node(n5){$u$};
				\draw(5*\x,2*\y) node(n6){$v$};
				\draw(6*\x,2*\y) node(n7){$w$};			
				\draw(1*\x,3*\y) node(n8){$rs$};
				\draw(2*\x,3*\y) node(n9){$rt$};
				\draw(3*\x,3*\y) node(n10){$st$};
				\draw(4*\x,3*\y) node(n11){$uv$};
				\draw(5*\x,3*\y) node(n12){$uw$};
				\draw(6*\x,3*\y) node(n13){$vw$};
				\draw(3.5*\x,4*\y) node(n14){$rst$};
				\draw(n1) -- (n2);
				\draw(n1) -- (n3);
				\draw(n1) -- (n4);
				\draw(n1) -- (n5);
				\draw(n1) -- (n6);
				\draw(n1) -- (n7);
				\draw(n2) -- (n8);
				\draw(n2) -- (n9);
				\draw(n3) -- (n8);
				\draw(n3) -- (n10);
				\draw(n4) -- (n9);
				\draw(n4) -- (n10);
				\draw(n5) -- (n11);
				\draw(n5) -- (n12);
				\draw(n6) -- (n11);
				\draw(n6) -- (n13);
				\draw(n7) -- (n12);
				\draw(n7) -- (n13);
				\draw(n8) -- (n14);
				\draw(n9) -- (n14);
				\draw(n10) -- (n14);
				\draw(n11) -- (n14);
				\draw(n12) -- (n14);
				\draw(n13) -- (n14);
			\end{tikzpicture}
			\caption{An interval in the prefix order on the quotient of the free abelian group on six generators $r,s,t,u,v,w$ given by the relation $rst=uvw$.}
			\label{fig:abelian_quotient_disconnected_poset}
		\end{subfigure}
		\hspace*{.1\textwidth}
		\begin{subfigure}[t]{.4\textwidth}
			\centering
			\begin{tikzpicture}\small
				\def\x{1};
				\def\y{1};
				\def\s{1};
				\draw(2.5*\x,.75*\y) node{};
				\draw(1.5*\x,1*\y) node[scale=\s](n1){$rst$};
				\draw(1*\x,1.75*\y) node[scale=\s](n2){$rts$};
				\draw(1*\x,2.5*\y) node[scale=\s](n3){$trs$};
				\draw(1.5*\x,3.25*\y) node[scale=\s](n4){$tsr$};
				\draw(2*\x,2.5*\y) node[scale=\s](n5){$str$};
				\draw(2*\x,1.75*\y) node[scale=\s](n6){$srt$};
				\draw(3.5*\x,1*\y) node[scale=\s](n7){$uvw$};
				\draw(3*\x,1.75*\y) node[scale=\s](n8){$uwv$};
				\draw(3*\x,2.5*\y) node[scale=\s](n9){$wuv$};
				\draw(3.5*\x,3.25*\y) node[scale=\s](n10){$wvu$};
				\draw(4*\x,2.5*\y) node[scale=\s](n11){$vwu$};
				\draw(4*\x,1.75*\y) node[scale=\s](n12){$vuw$};
				\draw(n1) -- (n2) -- (n3) -- (n4) -- (n5) -- (n6) -- (n1);
				\draw(n7) -- (n8) -- (n9) -- (n10) -- (n11) -- (n12) -- (n7);
			\end{tikzpicture}
			\caption{The Hurwitz graph of the poset in Figure~\ref{fig:abelian_quotient_disconnected_poset}.}
			\label{fig:abelian_quotient_disconnected_graph}
		\end{subfigure}
		\caption{An factorization poset which admits a compatible generator order, but is neither chain-connected, nor EL-shellable.}
		\label{fig:abelian_quotient_disconnected}
	\end{figure}
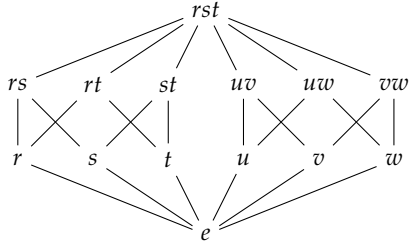
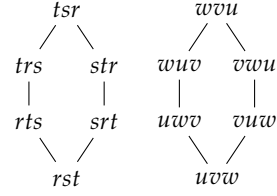
\end{example}

\begin{question}\label{qu:compatible_rising_chain_components}
	Let $\prec$ be a $\top$-compatible order.  Is there a connection between the number $\rise(\top;\prec)$ and the number of connected components of $\CC(\PP_{\top})$?  Is one quantity always smaller or equal to the other?  Are they in fact the same?
\end{question}

An affirmative answer to the next question would enable us to weaken the hypotheses of Conjecture~\ref{conj:compatible_el_shellable}.

\begin{question}\label{qu:compatible_totally_connected}
	Is every finite, chain-connected factorization poset that admits a compatible order totally chain-connected?
\end{question}

\subsection{The Well-Covered Property and EL-Labelings}
	\label{sec:well-covered}
Let us first state some further properties of factorization posets admitting compatible orders.  Fix a linear order $\prec$ of $\gen_{\top}$, and for $a\in\gen_{\top}$ define
\begin{displaymath}
	F_{\prec}(a;\top) \defs \bigl\{g\in\grp\mid a\lessdot_{\gen}g\leq_{\gen}\top\;\text{and there is}\;a'\in\gen_{\top}\;\text{with}\;a'\prec a\;\text{and}\;a'\lessdot_{\gen}g\bigr\}.
\end{displaymath}
In other words, $F_{\prec}(a;\top)$ consists of all upper covers of $a$ in $\PP_{\top}$ that also cover some $a'\prec a$.  

\begin{proposition}\label{prop:compatible_cover_set}
	Let $\prec$ be a $\top$-compatible order of $\gen_{\top}$.  For every $a\in\gen_{\top}$ and every $g\in\grp$ with $a\lessdot_{\gen}g$, the natural labeling $\lambda_{\top}$ satisfies $g\in F_{\prec}(a;\top)$ if and only if $\lambda_{\top}(a,g)\prec\lambda_{\top}(\id,a)$.
\end{proposition}
\begin{proof}
  Let $\gen_{\top}=\{a_{1},a_{2},\ldots,a_{N}\}$ where $a_{i}\prec a_{j}$ if and only if $i<j$, and let $\ell_{\gen}(\top)\geq 2$. (If $\ell_{\gen}(\top)\leq 1$, then the claim is vacuously satisfied.)  Let $g\in\grp$ be of length $2$ in $\PP_{\top}$, and pick $a_{j}\in\gen_{\top}$ such that $g=a_{j}a$ for some $a\in\gen_{\top}$.

  If $g\in F_{\prec}(a_{j};\top)$, then there must be some $a_{i}\in\gen_g$ with $i<j$.  It follows that $a_{j}$ is not minimal in $\gen_{g}$, and Lemma~\ref{lem:compatible_min_max} implies that $a_{j}\succ a$, which yields $\lambda_{\top}(a_{j},g)=a\prec a_{j}=\lambda_{\top}(\id,a_{j})$.

  If $g\notin F_{\prec}(a_{j};\top)$, then $a_{j}$ is the minimal generator in $\gen_{g}$, and Lemma~\ref{lem:compatible_min_max} implies $a_{j}\preceq a$.  Thus $\lambda_{\top}(a_{j},g)=a\succeq a_{j}=\lambda_{\top}(\id,a_{j})$.
\end{proof}

For every linear order $\prec$ of $\gen_{\top}$, the set $F_{\prec}(a;\top)$ is empty for $a=\min\gen_{\top}$.  Factorization posets in which this is the only case when $F_{\prec}(a;\top)$ is empty will be awarded a special name.  The point of this definition is that it provides a different perspective on the question for which linear orders $\lambda_{\top}$ is an EL-labeling.  

\begin{definition}\label{def:well_covered}
	A factorization poset $\PP_{\top}$ is \defn{well-covered} with respect to a linear order $\prec$ of $\gen_{\top}$ if $F_{\prec}(a;\top)$ is empty if and only if $a=\min\gen_{\top}$.  Moreover, $\PP_{\top}$ is \defn{totally well-covered} with respect to $\prec$ if for all $g\in P_{\top}$ the factorization poset $\PP_{g}$ is well-covered with respect to the appropriate restriction of $\prec$.  
\end{definition}

In other words, $\PP_{\top}$ is well-covered if and only if for every atom $a$ (except the smallest one with respect to $\prec$), we can find  a cover $g \gtrdot_{\gen} a$ such that $a$ is not the smallest atom in $\gen_g$.  As a consequence of Lemma~\ref{lem:bottom_intervals}, $\PP_{\top}$ is totally well-covered if and only if every interval of $\PP_{\top}$ is well-covered.

\begin{example}\label{ex:sym_4_ctd_2}
	Let us continue Example~\ref{ex:sym_4} once again, and fix the lexicographic order $\prec$ on $T$ from Example~\ref{ex:sym_4_ctd_1} again.  We now list the sets $F_{\prec}(t;\top)$ for any transposition $t$.\\
	\begin{center}\begin{tabular}{c||c}
		$t$ & $F_{\prec}(t;\top)$\\
		\hline\hline
		$(1\;2)$ & $\emptyset$\\
		$(1\;3)$ & $\bigl\{(1\;2\;3)\bigr\}$\\
		$(1\;4)$ & $\bigl\{(1\;2\;4),(1\;3\;4)\bigr\}$\\
		$(2\;3)$ & $\bigl\{(1\;2\;3),(1\;4)(2\;3)\bigr\}$\\
		$(2\;4)$ & $\bigl\{(1\;2\;4),(2\;3\;4)\bigr\}$\\
		$(3\;4)$ & $\bigl\{(1\;3\;4),(2\;3\;4),(1\;2)(3\;4)\bigr\}$\\
	\end{tabular}\end{center}
	Since $F_{\prec}(t;\top)$ is empty only for $t=(1\;2)=\min T_{\top}$ we conclude that $\PP_{\top}$ is well-covered with respect to $\prec$.
\end{example}

\begin{example}\label{ex:abelian_quotient_disconnected_ctd}
	Let us continue Example~\ref{ex:abelian_quotient_disconnected}.  If we fix the linear order $r\prec s\prec t\prec u\prec v\prec w$, then we observe that $F_{\prec}(u;rst)$ is empty even though $u$ is not minimal with respect to $\prec$.  By definition $\PP_{rst}$ is not well-covered with respect to $\prec$.  In fact, it is not well-covered with respect to any linear order on $\{r,s,t,u,v,w\}$.
\end{example}

\begin{proposition}\label{prop:well_covered_chain_connected}
	Let $\prec$ be a linear order of $\gen_{\top}$.  If $\PP_{\top}$ is totally well-covered with respect to $\prec$, then $\PP_{\top}$ is totally chain-connected.
\end{proposition}
\begin{proof}
	It suffices to prove that $\PP_{\top}$ is chain-connected.  The total variant then follows by restricting $\prec$ to any interval of $\PP_{\top}$ and repeating the argument.  We proceed by induction on $\ell_{\gen}(\top)$.  If $\ell_{\gen}(\top)=1$, then $\CC(\PP_{\top})$ consists of a single element, and is therefore trivially chain-connected.  So suppose that $\ell_{\gen}(\top)=n$, and the claim holds for all elements of length $<n$.  Let $a=\min\gen_{\top}$.

	First we prove that every maximal chain is in the same connected component of $\CC(\PP_{\top})$ as some maximal chain running through $a$. Let $C\in\MM(\PP_{\top})$, and suppose that $C\cap\gen_{\top}=\{b_{0}\}$.    If $b_{0}=a$, we are done. Otherwise, let $D_{0}=C\setminus\{\id,b_{0}\}$. Since $\PP_{\top}$ is well-covered we can find $g_{1}\in F_{\prec}(b_{0};\top)$ and $b_{1}\prec b_{0}$ with $b_{0}\lessdot_{\gen}g_{1}$ and $b_{1}\lessdot_{\gen}g_{1}$.  Let $D_{1}$ be some maximal chain in $[g_1,\top]_{\gen}$.  If $b_{1}=a$, then we stop; otherwise we repeat this construction until we obtain $b_{k}=a$ (recall that $\gen_\top$ is finite).

	By assumption, for every $i\in\{0,1,\ldots,k-1\}$, the interval $[b_{i},\top]_{\gen_{\top}}$ is totally well-covered with respect to the appropriate restriction of $\prec$.  By our induction hypothesis, we may thus find a sequence of maximal chains in $[b_{i},\top]_{\gen_{\top}}$ that connects $D_{i}\cup\{b_{i}\}$ and $D_{i+1}\cup\{b_{i}\}$.  Since $D_{i}\cup\{\id,b_{i}\}$ and $D_{i+1}\cup\{\id,b_{i}\}$ are adjacent in $\CC(\PP_{\top})$, this implies that $C$ and $D_{k}\cup\{\id,b_{k}\}$ belong to the same connected component of $\CC(\PP_{\top})$. 

	It remains to check that all the maximal chains running through $a$ are in the same connected component. Since $\PP_{\top}$ is totally well-covered, the interval $[a,\top]_{\gen}$ itself is well-covered, and therefore it is chain-connected by induction hypothesis.  This means that between any two maximal chains $D,D'$ of $[a,\top]_{\gen_{\top}}$ there exists a path in the corresponding chain graph connecting them, and this path corresponds to a path in $\CC(\PP_{\top})$ connecting $D\cup\{\id\}$ and $D'\cup\{\id\}$.
\end{proof}

Now fix a linear order $\prec$ of $\gen_{\top}$.  For $a\in\gen_{\top}$ let us denote by $\sqsubset^{a}$ the linear order of the atoms of $[a,\top]_{\gen}$ induced by $\prec$, \ie for $g,g'\in\grp$ with $a\lessdot_{\gen}g$ and $a\lessdot_{\gen}g'$ we have $g\sqsubset^{a}g'$ if and only if $a^{-1}g\prec a^{-1}g'$.  

\begin{lemma}\label{lem:compatible_recursive_prefix}
	Let $\prec$ be a $\top$-compatible order of $\gen_{\top}$.  For every $a\in\gen_{\top}$ there are no two upper covers $g,g'$ of $a$ such that $g\sqsubset^{a} g'$, but $g\notin F_{\prec}(a;\top)$ and $g'\in F_{\prec}(a;\top)$. 
\end{lemma}
\begin{proof}
	Suppose that there exists $a\in\gen_{\top}$ and $g,g'\leq_{\gen}\top$ with $a\lessdot_{\gen}g$ and $a\lessdot_{\gen}g'$ such that $g\sqsubset^{a}g'$, but $g\notin F_{\prec}(a;\top)$ and $g'\in F_{\prec}(a;\top)$.  By definition we find $a'\in\gen_{\top}$ with $a'\prec a$ and $a'\lessdot_{\gen}g'$.  Since $a\lessdot_{\gen}g$ we can write $g=ab$ for some $b\in\gen_{\top}$, and since $g\notin F_{\prec}(a;\top)$, we conclude $a\preceq b$.  Analogously we can find $b'\in\gen_{\top}$ such that $g'=ab'$.  Since $a$ is not minimal in $\gen_{g'}$ Lemma~\ref{lem:compatible_min_max} implies $b'\prec a$.  Since $g\sqsubset^{a}g'$ we obtain by construction that $a^{-1}g\prec a^{-1}g'$.  We thus obtain
	\begin{displaymath}
		a\preceq b = a^{-1}g \prec a^{-1}g' = b' \prec a,
	\end{displaymath}
	which is a contradiction.
\end{proof}

We proceed with the proof of the fact that $\PP_{\top}$ is totally well-covered with respect to a $\top$-compatible order $\prec$ if and only if $\lambda_{\top}$ is an EL-labeling with respect to $\prec$.

\begin{proposition}\label{prop:well_covered_shellable}
	Let $\prec$ be a $\top$-compatible order of $\gen_{\top}$. Suppose that $\PP_{\top}$ is totally well-covered with respect to $\prec$.  Then, for every $x,y\in\grp$ with $x\leq_{\gen}y\leq_{\gen}\top$ there exists a unique $\prec$-rising reduced $\gen$-factorization of $x^{-1}y$.
\end{proposition}
\begin{proof}
	Fix $x,y\in\grp$ with $x\leq_{\gen}y\leq_{\gen}\top$.  In view of Lemmas~\ref{lem:bottom_intervals} and \ref{lem:compatible_restriction} it suffices to consider the case $x=\id$ and $y=\top$.  We proceed by induction on $n=\ell_{\gen}(\top)$.  The cases $n\leq 1$ are trivial, so we can assume that $n>1$.
	
	Lemma~\ref{lem:lexicographically_smallest_rising} implies that the lexicographically smallest reduced $\gen$-factorization of $\top$ is $\prec$-rising. Denote this factorization by $F=(a_{1},a_{2},\ldots,a_{n})$.  Now let $\overline{F}=(\overline{a}_{1},\overline{a}_{2},\ldots,\overline{a}_{n})$ be any $\prec$-rising factorization.  If $\overline{a}_{1}=a_{1}$, then $(a_{2},a_{3},\ldots,a_{n})$ and $(\overline{a}_{2},\overline{a}_{3},\ldots,\overline{a}_{n})$ are both $\prec$-rising reduced $\gen$-factorization of $a_{1}^{-1}\top$. So by our induction hypothesis they are equal, and $\overline{F}=F$.

	Now assume $\overline{a}_{1}\neq a_{1}$, i.e., $\overline{a}_{1}\succ a_{1}$. Let $z\leq_{\gen}\top$ such that $\overline{a}_{1}\lessdot_{\gen}z$. We can thus write $z=\overline{a}_{1}b$ for some $b\in\gen_{\top}$.  By induction we obtain that $(\overline{a}_{2},\overline{a}_{3},\ldots,\overline{a}_{n})$ is the lexicographically smallest reduced $\gen$-factorization of $\overline{a}_{1}^{-1}\top$.  It follows that $\overline{a}_{2}\preceq b$, and therefore $\overline{a}_{1}\preceq b$, which in view of Proposition~\ref{prop:compatible_cover_set} implies that $z\notin F_{\prec}(\overline{a}_{1},\top)$.  Since $z$ was chosen arbitrarily we conclude $F_{\prec}(\overline{a}_{1},\top)=\emptyset$.  Since $\PP_{\top}$ is (totally) well-covered we conclude that $\overline{a}_{1}=\min\gen_{\top}$, which yields the contradiction $\overline{a}_{1}\preceq a_{1}$.
\end{proof}

\begin{theorem}\label{thm:shellable_well_covered}
	Let $\prec$ be a linear order of $\gen_{\top}$. Then $\lambda_{\top}$ is an EL-labeling of $\PP_{\top}$ with respect to $\prec$ if and only if $\prec$ is $\top$-compatible and $\PP_{\top}$ is totally well-covered with respect to $\prec$.
\end{theorem}
\begin{proof}
	If $\PP_{\top}$ is totally well-covered with respect to a $\top$-compatible order $\prec$, then Proposition~\ref{prop:well_covered_shellable} implies that for all $x\leq_{\gen}y\leq_{\gen}\top$ there exists a unique $\prec$-rising reduced $\gen$-factorization of $x^{-1}y$.  By Lemma~\ref{lem:lexicographically_smallest_rising} this must necessarily be the lexicographically smallest reduced $\gen$-factorization.  It follows that $\lambda_{\top}$ is an EL-labeling of $\PP_{\top}$.
  
	Now assume that $\lambda_{\top}$ is an EL-labeling of $\PP_{\top}$.  By definition, every interval of rank $2$ has a unique $\prec$-rising maximal chain, which implies that $\prec$ is $\top$-compatible.  Moreover, since $\lambda_{\top}$ restricts to an EL-labeling of any interval of $\PP_{\top}$ by restricting~$\prec$ accordingly, it suffices to prove that $\PP_{\top}$ is well-covered.

	Now pick $a_1\in\gen_{\top}$ with $a_1\neq\min\gen_{\top}$.  Let $(a_{2},a_{3},\ldots,a_{n})$ be the lexicographically smallest reduced $\gen$-factorization of $a_1^{-1}\top$.  Lemma~\ref{lem:lexicographically_smallest_rising} implies $a_{2}\prec a_{3}\prec\cdots\prec a_{n}$.  Since $a_1\neq\min\gen_{\top}$ we also know that $(a_1,a_{2},\ldots,a_{n})$ is not the lexicographically smallest reduced $\gen$-factorization of $\top$.  Since $\lambda_{\top}$ is an EL-labeling of $\PP_{\top}$ we obtain that $(a_1,a_{2},\ldots,a_{n})$ is not $\prec$-rising, and thus $a_1\succ a_{2}$.  Consider $g=a_1a_{2}$, so that $a_1\lessdot g$ and $a_2\lessdot g$.  Proposition~\ref{prop:compatible_cover_set} implies $g\in F_{\prec}(a_1;\top)$.  Since~$a_1$ was chosen arbitrarily we conclude that $F_{\prec}(a_1;\top)\neq\emptyset$ whenever $a_1\neq\min\gen_{\top}$, which precisely says that $\PP_{\top}$ is well-covered.
\end{proof}

We thus obtain the following equivalent statement of Conjecture~\ref{conj:compatible_el_shellable}.

\begin{conjecture}\label{conj:compatible_well_covered}
	If $\PP_{\top}$ is totally chain-connected, and $\prec$ is a $\top$-compatible order of $\gen_{\top}$, then $\PP_{\top}$ is totally well-covered with respect to $\prec$.
\end{conjecture}

It follows from Proposition~\ref{prop:well_covered_chain_connected} that every well-covered poset is chain-connected.  However, the example in Figure~\ref{fig:chain_nonhurwitz_poset} shows that factorization posets may be well-covered with respect to non-compatible atom orders.

\begin{proposition}\label{prop:conjecture_implication_2}
	Conjectures~\ref{conj:compatible_el_shellable} and \ref{conj:compatible_well_covered} are equivalent.
\end{proposition}
\begin{proof}
	Let us first assume that Conjecture~\ref{conj:compatible_el_shellable} is true.  If $\PP_{\top}$ is totally chain-connected and admits a $\top$-compatible order $\prec$, then Conjecture~\ref{conj:compatible_el_shellable} implies that~$\lambda_{\top}$ is an EL-labeling of $\PP_{\top}$ with respect to $\prec$.  Now, Theorem~\ref{thm:shellable_well_covered} implies that~$\PP_{\top}$ is well-covered with respect to $\prec$, which establishes Conjecture~\ref{conj:compatible_well_covered}.
	
	\medskip
	
	Now, conversely, let us assume that Conjecture~\ref{conj:compatible_well_covered} is true.  If $\PP_{\top}$ is totally chain-connected and admits a $\top$-compatible order $\prec$, then Conjecture~\ref{conj:compatible_well_covered} implies that $\PP_{\top}$ is totally well-covered with respect to $\prec$.  Now, Theorem~\ref{thm:shellable_well_covered} implies that~$\lambda_{\top}$ is an EL-labeling of $\PP_{\top}$ with respect to $\prec$.
	
	If $\lambda_{\top}$ is an EL-labeling of $\PP_{\top}$ with respect to some linear order $\prec$ of $\gen_{\top}$, then Theorem~\ref{thm:shellable_well_covered} implies that $\prec$ is $\top$-compatible, and $\PP_{\top}$ is totally well-covered with respect to $\prec$.  Now, Proposition~\ref{prop:well_covered_chain_connected} implies that $\PP_{\top}$ is totally chain-connected.
	
	We have thus established Conjecture~\ref{conj:compatible_el_shellable}.
\end{proof}

\begin{remark}\label{rem:well_covered_origin}
	The well-covered property is modeled after Condition~(ii) in \cite{bjorner83lexicographically}*{Definition~3.1}, which introduces the concept of a recursive atom order of a bounded graded poset.  Lemma~\ref{lem:compatible_recursive_prefix} implies that a $\top$-compatible order of $\gen_{\top}$ satisfies Condition~(i) of \cite{bjorner83lexicographically}*{Definition~3.1}.  Consequently, if Conjecture~\ref{conj:compatible_well_covered} were true, any $\top$-compatible order of $\gen_{\top}$ in a totally chain-connected factorization poset would be a recursive atom order.
		
	Moreover, the proofs of Propositions~\ref{prop:compatible_cover_set}, \ref{prop:well_covered_shellable} and Theorem~\ref{thm:shellable_well_covered} are essentially verbatim to parts of the proof of \cite{bjorner83lexicographically}*{Theorem~3.2}.  
\end{remark}

Let us conclude this section with the proof of Theorem~\ref{thm:main_result}.  

\begin{proof}[Proof of Theorem~\ref{thm:main_result}]
	Suppose that $\PP_{\top}$ admits a $\top$-compatible order of $\gen_{\top}$ and that it is totally well-covered.  Proposition~\ref{prop:well_covered_chain_connected} implies that $\PP_{\top}$ is (totally) chain-connected.  Moreover, Theorem~\ref{thm:main_hurwitz} implies that it is Hurwitz-connected, and Theorem~\ref{thm:shellable_well_covered} implies that it is shellable.
\end{proof}

\section{The Cycle Graph of $\PP_{\top}$}
	\label{sec:cycle_graph}
In this section we prove a particular case of Conjecture~\ref{conj:compatible_well_covered}.  We achieve this with the help of a graph-theoretic representation of a part of $\PP_{\top}$.  To that end, we associate a directed labeled graph constructed from cycles with $\PP_{\top}$; the \emph{cycle graph} of $\PP_{\top}$.  This structure retains only the information about the rank-$2$ subintervals of $\PP_{\top}$ but allows to tackle some cases of our conjectures systematically and in a nice graphical way.

\subsection{Definition and Properties of the Cycle Graph}
	\label{sec:def_cyclegraph}
Recall Assumption~\ref{disc:finiteness}:  $\gen_{\top}$ is assumed to be finite.  Let
\begin{equation*}
	\rtwo_{\top} \defs \bigl\{g\in\grp\mid g\leq_{\gen}\top\;\text{and}\;\ell_{\gen}(g)=2\bigr\}
\end{equation*}
be the set of length $2$ elements below $\top$.  For any $g\in\rtwo_{\top}$, the set of reduced $\gen$-factorizations of $g$ is partitioned into \defn{Hurwitz orbits}, \ie connected components of $\HH(g)$.  Since $\gen_{\top}$ is finite, each of these orbits consists of factorizations of the form
\begin{displaymath}
	g = a_{1}a_{2} = a_{2}a_{3} = \cdots = a_{k-1}a_{k} = a_{k}a_{1},
\end{displaymath}
for some $k\geq 1$ and some pairwise distinct generators $a_{1},a_{2},\ldots,a_{k}$.  If $k=1$, then $g=a_{1}^{2}$.  We can represent this Hurwitz orbit by drawing a cycle connecting $a_{1},a_{2},\ldots,a_{k},a_{1}$ in that order.  This gives rise to the following definition.

\begin{definition}\label{def:cycle_graph}
	Let $\PP_{\top}$ be a factorization poset.  The \defn{cycle graph} of $\PP_{\top}$, denoted by $\Gamma(\PP_{\top})$, is a directed labeled graph such that
	\begin{itemize}
		\item the set of vertices of $\Gamma(\PP_{\top})$ is $\gen_{\top}$;
		\item for $a,b\in\gen_{\top}$ we have that $a\rightarrow b$ is an oriented edge of $\Gamma(\PP_{\top})$ if and only if $ab\in\rtwo_{\top}$;
		\item for $a,b\in\gen_{\top}$ the label of $a\rightarrow b$ is given by the product $ab$.
	\end{itemize}
\end{definition}

\begin{lemma}\label{lem:cycle_graph_properties}
	$\Gamma(\PP_{\top})$ has the following properties.
	\begin{itemize}
		\item $\Gamma(\PP_{\top})$ has no multiple directed edges.
		\item $\Gamma(\PP_{\top})$ has a loop with vertex $a$ and label $g$ if and only if $g=a^{2}\in\rtwo_{\top}$.
		\item The set of edge labels of $\Gamma(\PP_{\top})$ is $\rtwo_{\top}$, and every element of $\rtwo_{\top}$ appears as the label of at least one edge of $\Gamma(\PP_{\top})$.
		\item For any $g\in\rtwo_{\top}$, the set of edges of $\Gamma(\PP_{\top})$ having label $g$ is a disjoint union of directed cycles, each corresponding to a connected component of $\HH(g)$.
	\end{itemize}
\end{lemma}
\begin{proof}
	This is immediate from the definition.
\end{proof}

Figure~\ref{fig:cycle_graphs} shows three cycle graphs of factorization posets we have encountered before. For brevity we have omitted the edge labels, and have instead used colors to indicate edges with the same label.  

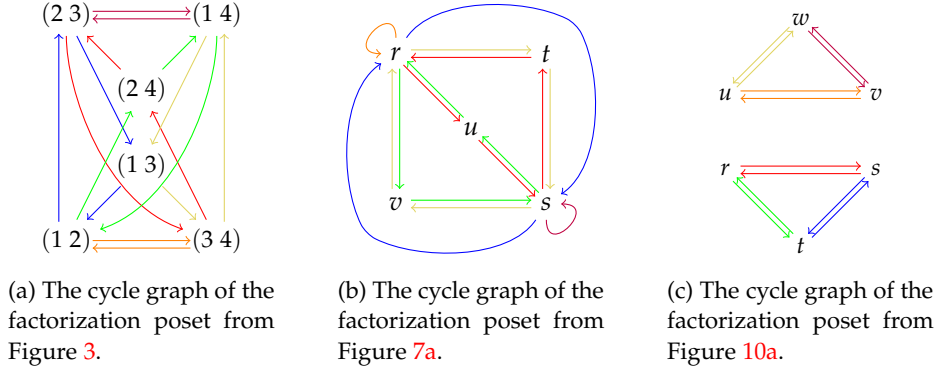
\begin{figure}
	\centering
	\begin{subfigure}[t]{.28\textwidth}
		\centering
		\begin{tikzpicture}\small
			\def\x{1};
			\def\y{1};
			\draw(1*\x,1*\y) node(n1){$(1\;2)$};
			\draw(2*\x,2*\y) node(n2){$(1\;3)$};
			\draw(3*\x,4*\y) node(n3){$(1\;4)$};
			\draw(1*\x,4*\y) node(n4){$(2\;3)$};
			\draw(2*\x,3*\y) node(n5){$(2\;4)$};
			\draw(3*\x,1*\y) node(n6){$(3\;4)$};
			\draw[->,purple](1.35*\x,4.05*\y) -- (2.65*\x,4.05*\y);
			\draw[->,purple](2.65*\x,3.95*\y) -- (1.35*\x,3.95*\y);
			\draw[->,orange](1.35*\x,1*\y) -- (2.65*\x,1*\y);
			\draw[->,orange](2.65*\x,.9*\y) -- (1.35*\x,.9*\y);
			\draw[->,blue](.9*\x,1.25*\y) -- (.9*\x,3.75*\y);
			\draw[->,blue](n4) -- (n2);
			\draw[->,blue](n2) -- (n1);
			\draw[->,red](n4) .. controls (1*\x,3*\y) and (1.5*\x,1.5*\y) .. (n6);
			\draw[->,red](n6) -- (n5);
			\draw[->,red](n5) -- (n4);
			\draw[->,green](n1) -- (n5);
			\draw[->,green](n5) -- (n3);
			\draw[->,green](n3) .. controls (3*\x,3*\y) and (2.5*\x,1.5*\y) .. (n1);
			\draw[->,yellow!70!gray](n3) -- (n2);
			\draw[->,yellow!70!gray](n2) -- (n6);
			\draw[->,yellow!70!gray](3.1*\x,1.25*\y) -- (3.1*\x,3.75*\y);
		\end{tikzpicture}
		\caption{The cycle graph of the factorization poset from Figure~\ref{fig:sym_4_poset}.}
		\label{fig:sym_4_cycle_graph}
	\end{subfigure}
	\hspace*{.05\textwidth}
	\begin{subfigure}[t]{.28\textwidth}
		\centering
		\begin{tikzpicture}
			\def\x{1};
			\def\y{1};
			\draw(1*\x,1*\x) node(n1){$v$};
			\draw(3*\x,1*\x) node(n2){$s$};
			\draw(2*\x,2*\x) node(n3){$u$};
			\draw(1*\x,3*\x) node(n4){$r$};
			\draw(3*\x,3*\x) node(n5){$t$};
			\draw[->,yellow!70!gray](1.2*\x,3.05*\y) -- (2.8*\x,3.05*\y);
			\draw[->,yellow!70!gray](2.8*\x,.95*\y) -- (1.2*\x,.95*\y);
			\draw[->,yellow!70!gray](3.05*\x,2.8*\y) -- (3.05*\x,1.2*\y);
			\draw[->,yellow!70!gray](.95*\x,1.2*\y) -- (.95*\x,2.8*\y);
			\draw[->,red](2.8*\x,2.95*\y) -- (1.2*\x,2.95*\y);
			\draw[->,red](1.1*\x,2.85*\y) -- (1.85*\x,2.1*\y);
			\draw[->,red](2.1*\x,1.85*\y) -- (2.85*\x,1.1*\y);
			\draw[->,red](2.95*\x,1.2*\y) -- (2.95*\x,2.8*\y);
			\draw[->,green](1.2*\x,1.05*\y) -- (2.8*\x,1.05*\y);
			\draw[->,green](1.9*\x,2.15*\y) -- (1.15*\x,2.9*\y);
			\draw[->,green](2.9*\x,1.15*\y) -- (2.15*\x,1.9*\y);
			\draw[->,green](1.05*\x,2.8*\y) -- (1.05*\x,1.2*\y);
			\draw[->,rounded corners,blue](n2) .. controls (2.5*\x,.25*\y) and (1.3*\x,.25*\y) .. (.6*\x,.6*\y) .. controls (.25*\x,1.3*\y) and (.25*\x,2.5*\y) .. (n4);
			\draw[->,rounded corners,blue](n4) .. controls (1.5*\x,3.75*\y) and (2.7*\x,3.75*\y) .. (3.4*\x,3.4*\y) .. controls (3.75*\x,2.7*\y) and (3.75*\x,1.5*\y) .. (n2);
			\path (n4) edge[out=90,in=180,looseness=.75,distance=1.5em,->,orange] (n4);
			\path (n2) edge[out=270,in=0,looseness=.75,distance=1.5em,->,purple] (n2);
		\end{tikzpicture}
		\caption{The cycle graph of the factorization poset from Figure~\ref{fig:hurwitz_nonshellable_poset}.}
		\label{fig:hurwitz_nonshellable_cycle_graph}
	\end{subfigure}
	\hspace*{.05\textwidth}
	\begin{subfigure}[t]{.28\textwidth}
		\centering
		\begin{tikzpicture}\small
			\def\x{1};
			\def\y{1};
			\draw(2*\x,1*\y) node(n1){$t$};
			\draw(1*\x,2*\y) node(n2){$r$};
			\draw(3*\x,2*\y) node(n3){$s$};
			\draw(1*\x,3*\y) node(n4){$u$};
			\draw(3*\x,3*\y) node(n5){$v$};
			\draw(2*\x,4*\y) node(n6){$w$};
			\draw[->,orange](1.2*\x,3.05*\y) -- (2.8*\x,3.05*\y);
			\draw[->,orange](2.8*\x,2.95*\y) -- (1.2*\x,2.95*\y);
			\draw[->,yellow!70!gray](1.85*\x,3.9*\y) -- (1.1*\x,3.15*\y);
			\draw[->,yellow!70!gray](1.15*\x,3.1*\y) -- (1.9*\x,3.85*\y);
			\draw[->,purple](2.9*\x,3.15*\y) -- (2.15*\x,3.9*\y);
			\draw[->,purple](2.1*\x,3.85*\y) -- (2.85*\x,3.1*\y);
			\draw[->,red](1.2*\x,2.05*\y) -- (2.8*\x,2.05*\y);
			\draw[->,red](2.8*\x,1.95*\y) -- (1.2*\x,1.95*\y);
			\draw[->,green](1.9*\x,1.15*\y) -- (1.15*\x,1.9*\y);
			\draw[->,green](1.1*\x,1.85*\y) -- (1.85*\x,1.1*\y);
			\draw[->,blue](2.85*\x,1.9*\y) -- (2.1*\x,1.15*\y);
			\draw[->,blue](2.15*\x,1.1*\y) -- (2.9*\x,1.85*\y);
		\end{tikzpicture}
		\caption{The cycle graph of the factorization poset from Figure~\ref{fig:abelian_quotient_disconnected_poset}.}
		\label{fig:abelian_quotient_disconnected_cycle_graph}
	\end{subfigure}
	\caption{Some cycle graphs.}
	\label{fig:cycle_graphs}
\end{figure}

Now assume that $\PP_{\top}$ admits a $\top$-compatible order of $\gen_{\top}$.  It follows in this case from Lemma~\ref{lem:compatible_rank_2_hurwitz} that for every $g\in\rtwo_{\top}$ the directed edges in $\Gamma(\PP_{\top})$ labeled by $g$ form a single cycle.  In fact the existence of a $\top$-compatible order puts a bound on the number of edges we need to remove from $\Gamma(\PP_{\top})$ to make it acyclic.  Generally, for any directed graph $\Gamma$ a \defn{feedback arc set} is a set of directed edges of $\Gamma$ whose removal makes $\Gamma$ acyclic.  Let $d(\Gamma)$ denote the minimal size of a feedback arc set of $\Gamma$.  

\begin{proposition}\label{prop:minimal_feedback_compatible}
	Any factorization poset $\PP_{\top}$ satisfies $d\bigl(\Gamma(\PP_{\top})\bigr)\geq \bigl\lvert\rtwo_{\top}\bigr\rvert$.  Moreover, $d\bigl(\Gamma(\PP_{\top})\bigr)=\bigl\lvert\rtwo_{\top}\bigr\rvert$ if and only if $\PP_{\top}$ admits a $\top$-compatible order of $\gen_{\top}$.
\end{proposition}
\begin{proof}
	Let us abbreviate $\Gamma=\Gamma(\PP_{\top})$.  It is immediate from the definition that $d(\Gamma)$ needs to be at least the size of any collection of edge-disjoint cycles of~$\Gamma$.  Moreover, every $g\in\rtwo_{\top}$ produces a collection of $k_{g}\geq 1$ cycles of $\Gamma$, say $C_{g}^{(i)}$ for $i=1,2,\dots, k_g$.  Since the edges in any $C_{g}^{(i)}$ are labeled by $g$, we find that $\bigl\{C_{g}^{(i)}\mid g\in\rtwo_{\top}, i\in[k_{g}]\bigr\}$ is a collection of edge-disjoint cycles of $\Gamma$, and we obtain
	\begin{displaymath}
		d(\Gamma) \geq \bigl\lvert\bigl\{C_{g}^{(i)}\mid g\in\rtwo_{\top}, i\in[k_{g}]\bigr\}\bigr\rvert \geq \bigl\lvert\rtwo_{\top}\bigr\rvert.
	\end{displaymath}
	
	Now suppose that $d(\Gamma)=\bigl\lvert\rtwo_{\top}\bigr\rvert$. This implies in particular that there is only one cycle labeled by $g$ for each $g\in\rtwo_{\top}$; which we denote simply by $C_g$.  Let $\Gamma'$ be an acyclic graph that is constructed from $\Gamma$ by removing $d(\Gamma)$ edges.  It follows that $\Gamma'$ is obtained by removing exactly one edge for each $C_{g}$.  Since $\Gamma'$ is acyclic, we can define the following partial order on $\gen_{\top}$: $a \preceq b$ if and only if $b \rightarrow a$ in $\Gamma'$. Then by construction, any linear extension of $\preceq$ is $\top$-compatible. 
	
	Conversely suppose that there exists a $\top$-compatible order $\prec$ on $\gen_{\top}$.  Then from Lemma~\ref{lem:compatible_rank_2_hurwitz}, the Hurwitz action of $\bg_2$ on $\red{\gen}{g}$ is transitive for any $g\in\rtwo_{\top}$. Let us write the unique Hurwitz orbit for $g$ as $g=a_{1}a_{2} = a_{2}a_{3} = \cdots = a_{k-1}a_{k} = a_{k}a_{1}$; this corresponds in $\Gamma$ to the unique cycle $C_g$ labeled by $g$. Since $\prec$ is a $\top$-compatible order, we can find some $i\in[k]$ such that $a_{i}\succ a_{i+1}\succ\cdots\succ a_{i+k-1}$ (indices taken mod $k$).  Now let us remove from $\Gamma$ the edge from $a_{i-1}$ to $a_{i}$ in $C_{g}$.  (In particular, we remove all the loops.)  After doing that for each $g\in \rtwo_{\top}$, every edge $a\rightarrow b$ in the resulting graph $\Gamma'$ satisfies $a\succ b$.  This implies that there can be no cycles left in $\Gamma'$.  This yields $d(\Gamma)=\bigl\lvert\rtwo_{\top}\bigr\rvert$.
\end{proof}

If $\prec$ is $\top$-compatible, let $\Gamma_{\prec}(\PP_{\top})$ denote the \defn{reduced cycle graph}, \ie the graph that is constructed from $\PP(\top)$ by removing the edges corresponding to the $\prec$-rising reduced $\gen$-factorizations of $g\in\rtwo_{\top}$.  Figure~\ref{fig:reduced_cycle_graphs} shows reduced versions of the cycle graphs in Figures~\ref{fig:sym_4_cycle_graph} and \ref{fig:abelian_quotient_disconnected_cycle_graph}.  The cycle graph in Figure~\ref{fig:hurwitz_nonshellable_cycle_graph} has six different edge labels, but it is quickly verified that we need to remove at least nine edges to make this graph acyclic.  (These are the two loops plus one edge for each of the seven back-and-forth pairs.)  Proposition~\ref{prop:minimal_feedback_compatible} then implies that the corresponding factorization poset does not admit a compatible order, which we have already verified in Example~\ref{ex:no_compatible_order}.  

\begin{figure}
	\centering
	\begin{subfigure}[t]{.45\textwidth}
		\centering
		\begin{tikzpicture}\small
			\def\x{1};
			\def\y{1};
			\draw(1*\x,1*\y) node(n1){$(1\;2)$};
			\draw(2*\x,2*\y) node(n2){$(1\;3)$};
			\draw(3*\x,4*\y) node(n3){$(1\;4)$};
			\draw(1*\x,4*\y) node(n4){$(2\;3)$};
			\draw(2*\x,3*\y) node(n5){$(2\;4)$};
			\draw(3*\x,1*\y) node(n6){$(3\;4)$};
			\draw[->,purple](1.35*\x,4.05*\y) -- (2.65*\x,4.05*\y);
			\draw[->,orange](2.65*\x,.9*\y) -- (1.35*\x,.9*\y);
			\draw[->,blue](n4) -- (n2);
			\draw[->,blue](n2) -- (n1);
			\draw[->,red](n6) -- (n5);
			\draw[->,red](n5) -- (n4);
			\draw[->,green](n5) -- (n3);
			\draw[->,green](n3) .. controls (3*\x,3*\y) and (2.5*\x,1.5*\y) .. (n1);
			\draw[->,yellow!70!gray](n3) -- (n2);
			\draw[->,yellow!70!gray](3.1*\x,1.25*\y) -- (3.1*\x,3.75*\y);
		\end{tikzpicture}
		\caption{A reduced cycle graph of the cycle graph from Figure~\ref{fig:sym_4_cycle_graph} corresponding to the order $(1\;2)\prec(1\;3)\prec(1\;4)\prec(2\;3)\prec(2\;4)\prec(3\;4)$.}
		\label{fig:sym_4_red_cycle_graph}
	\end{subfigure}
	\hspace*{.05\textwidth}
	\begin{subfigure}[t]{.45\textwidth}
		\centering
		\begin{tikzpicture}\small
			\def\x{1};
			\def\y{1};
			\draw(2*\x,1*\y) node(n1){$t$};
			\draw(1*\x,2*\y) node(n2){$r$};
			\draw(3*\x,2*\y) node(n3){$s$};
			\draw(1*\x,3*\y) node(n4){$u$};
			\draw(3*\x,3*\y) node(n5){$v$};
			\draw(2*\x,4*\y) node(n6){$w$};
			\draw[->,orange](2.8*\x,2.95*\y) -- (1.2*\x,2.95*\y);
			\draw[->,yellow!70!gray](1.85*\x,3.9*\y) -- (1.1*\x,3.15*\y);
			\draw[->,purple](2.1*\x,3.85*\y) -- (2.85*\x,3.1*\y);
			\draw[->,red](2.8*\x,1.95*\y) -- (1.2*\x,1.95*\y);
			\draw[->,green](1.9*\x,1.15*\y) -- (1.15*\x,1.9*\y);
			\draw[->,blue](2.15*\x,1.1*\y) -- (2.9*\x,1.85*\y);
		\end{tikzpicture}
		\caption{The reduced cycle graph of the cycle graph from Figure~\ref{fig:abelian_quotient_disconnected_cycle_graph} corresponding to the order $r\prec s\prec t\prec u\prec v\prec w$.}
		\label{fig:abelian_quotient_disconnected_red_cycle_graph}
	\end{subfigure}
	\caption{Some reduced cycle graphs.}
	\label{fig:reduced_cycle_graphs}
\end{figure}
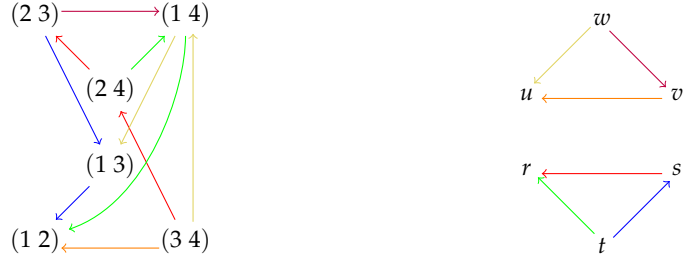

From the proof of Proposition~\ref{prop:minimal_feedback_compatible} we obtain the following corollary.

\begin{corollary}\label{cor:compatible_linear_extension}
	Let $\prec$ be a $\top$-compatible order of $\gen_{\top}$.  Let $(\gen_{\top},\sqsubseteq)$ be the dual of the poset induced by $\Gamma_{\prec}(\PP_{\top})$.  Then, $\prec$ is a linear extension of $\sqsubseteq$, and moreover every linear extension of $\sqsubseteq$ is $\top$-compatible.
\end{corollary}

A directed graph on the vertex set $V$ is \defn{connected} if for every $v_{1},v_{2}\in V$ there exists a directed path from $v_{1}$ to $v_{2}$.  

\begin{lemma}\label{lem:reduced_graph_linear_connected}
	The order induced by $\Gamma_{\prec}(\PP_{\top})$ is linear (and therefore equal to $\prec$) if and only if $\Gamma_{\prec}(\PP_{\top})$ is connected.
\end{lemma}
\begin{proof}
	This follows again from the construction, since every two vertices $a$ and $b$ of $\Gamma_{\prec}(\PP_{\top})$ are comparable in $(\gen_{\top},\sqsubseteq)$ if and only if they are connected by a directed path in $\Gamma_{\prec}(\PP_{\top})$.
\end{proof}

We can also characterize the well-covered property inside the reduced cycle graph.

\begin{proposition}\label{prop:reduced_graph_linear_well_covered}
  For any $\top$-compatible order $\prec$, $\PP_{\top}$ is well-covered if and only if the reduced cycle graph $\Gamma_{\prec}(\PP_{\top})$ has a unique sink.
  In particular, if $\Gamma_{\prec}(\PP_{\top})$ induces a linear order, then $\PP_{\top}$ is well-covered.
\end{proposition}
\begin{proof}
  Let $s_0$ be the minimal element of $\gen_\top$ for the order $\prec$. Clearly $s_0$ is a sink. 
  A vertex $b$ of $\Gamma_{\prec}(\PP_{\top})$ is not a sink if and only if there exists a vertex $a$ with $b\rightarrow a$, which is equivalent to $a\prec b$, with both $a$ and $b$ covered by some $g\in\rtwo_{\top}$ (corresponding to the label of that edge).  So $b$ is not a sink in $\Gamma_{\prec}(\PP_{\top})$ if and only if there exists $g\in\rtwo_{\top}$ such that $g\in F_{\prec}(b,\top)$, i.e. exactly when $F_{\prec}(b,\top)\neq \emptyset$. By definition, this means that $\PP_{\top}$ is well-covered if and only if $\Gamma_{\prec}$ has only $s_0$ as a sink.

The second statement follows naturally since an acyclic graph inducing a linear order has necessarily a unique sink.
\end{proof}

Then the following conjecture, described in terms of the cycle graph, implies Conjecture~\ref{conj:compatible_well_covered}.

\begin{conjecture}\label{conj:compatible_linear_cycle_graph}
	If $\PP_{\top}$ is totally chain-connected and admits a $\top$-compatible order $\prec$ of $\gen_{\top}$, then for any $g\in \PP_{\top}$, the reduced cycle graph $\Gamma_{\prec}(\PP_g)$ induces a linear order.
\end{conjecture}

\subsection{Proof of a Specific Case}
	\label{sec:proof_of_specific_case}
In the remainder of this section we prove a particular case of Conjecture~\ref{conj:compatible_linear_cycle_graph}.  Recall from Assumption~\ref{disc:finiteness} that we assume $\gen_{\top}$ to be finite.

\begin{theorem}\label{thm:generator_single_cycle}
	Let  $\PP_{\top}(\grp,\gen)$ be a factorization poset and denote $n:=\ell_{\gen}(\top)$. Assume that $n\geq 3$, and assume there exists an element $a$ in $\gen_{\top}$ that is covered in $\PP_{\top}$ by a unique element $g$ (other than possibly $a^{2}$). Then Conjecture~\ref{conj:compatible_linear_cycle_graph} holds.
\end{theorem}

We assume in the following the hypotheses of Conjecture~\ref{conj:compatible_linear_cycle_graph}, \ie $\PP_{\top}$ is totally chain-connected and there exists a $\top$-compatible order of $\gen_{\top}$.  In particular, Theorem~\ref{thm:main_hurwitz} implies that $\PP_{g}$ is Hurwitz-connected for every $g\leq_{\gen}\top$.  

Let us denote the unique element in $\rtwo_{\top}$ from Theorem~\ref{thm:generator_single_cycle} by $g_{a}$, and let $b\in\gen_{\top}$ be the unique generator satisfying $g_{a}=ab$.  The size of the cycle containing $a$ is therefore equal to $\lvert\gen_{g_{a}}\rvert$.  

In terms of the cycle graph the condition in Theorem~\ref{thm:generator_single_cycle} translates as follows: there is a unique cycle containing the vertex $a$ (not counting a possible loop), and this cycle is labeled by $g_{a}$.  In particular, thanks to Proposition~\ref{prop:subword_order}, we get the following crucial property.
\begin{equation}\label{eq:facto_starting_with_a}
	\text{Any factor in a reduced $\gen$-factorization of $\top$ starting with $a$ is either $a$ or $b$.}	
\end{equation}
Note that this does not imply $\gen_{\top}=\{a,b\}$; there may be factorizations of $\top$ not containing $a$ at all. Consider now a factorization $(a,x_2,x_3, \dots, x_n)\in \red{\gen}{\top}$. Let us consider three disjoint cases.

\begin{enumerate}[(i)]
	\item \emph{All the $x_i$'s, for $i\in \{2,3,\ldots,n\}$, are equal to $a$.} This means that $\top=a^{n}$ and by Hurwitz-connectivity we obtain $\gen_{\top}=\{a\}$, which contradicts the hypotheses of Theorem~\ref{thm:generator_single_cycle}.
	\item \emph{There is at least one $a$ and one $b$ in the $x_i$'s for $i\in \{2,3,\ldots,n\}$.} Then, we may apply a Hurwitz move on a sequence of factors $(a,b)$ (or $(b,a)$) and obtain the element $aba^{-1}$ (or $a^{-1}ba$) as a factor appearing after $a$ in a factorization of $\top$. Thus, by Property~\eqref{eq:facto_starting_with_a}, this element is either $a$ or $b$. It cannot be equal to~$a$ (since $b\neq a$), so we deduce that $ab=ba$.  So $\gen_{g_{a}}=\{a,b\}$, and $\top$ can be written $a^{k}b^{l}$ for some positive integers $k,l$. By Hurwitz-connectivity, we conclude that $\gen_{\top}=\{a,b\}$.  The cycle graph of $\PP_{\top}$ in this particular case consists of a single cycle of length two, and possibly one or two loops; the reduced cycle graph is just a single directed edge, which clearly induces a linear order of~$\gen_{\top}$. This concludes the proof of Theorem~\ref{thm:generator_single_cycle} in this case.
	\item \emph{All the $x_i$'s, for $i\in \{2,3,\ldots,n\}$, are equal to $b$.} This means that $\top=a b^{n-1}$, and this is the only nontrivial case of the theorem, to which we devote the rest of the section.
\end{enumerate}

The proof of Theorem~\ref{thm:generator_single_cycle} is now reduced to the two theorems below: Theorem~\ref{thm:big_cycle} which deals with the case $\lvert\gen_{g_{a}}\rvert\geq 4$, and Theorem~\ref{thm:particular_top} which tackles the remaining case $\lvert\gen_{g_{a}}\rvert=3$.  In both cases we apply the following reasoning.  We assume knowledge of the Hurwitz orbit of $\red{\gen}{g_{a}}$ and we let the appropriate braid group act on the factorization $ab^{n-1}$ in order to exhibit other elements in $\gen_{\top}$ and $\rtwo_{\top}$ step by step.  With the help of Lemma~\ref{prop:reverse_same_orbit_size} we deduce the size of the corresponding Hurwitz orbits, and we continue until we find no new elements.  There exists a \textsc{Sage}-script that may assist with this process, which can be obtained from \url{https://www.math.tu-dresden.de/\textasciitilde hmuehle/files/sage/hurwitz.sage}.  In many cases (or when $\lvert\gen_{g_{a}}\rvert$ is large enough) we find out that there does not exist any $\top$-compatible order.  In the remaining cases, we check explicitly that the reduced cycle graph induces a linear order.

\begin{theorem}\label{thm:big_cycle}
  Let  $\PP_{\top}(\grp,\gen)$ be a factorization poset satisfying the assumptions of Theorem~\ref{thm:generator_single_cycle}. Assume moreover that $\PP_\top$ is totally chain-connected.

  If $abb\leq_{\gen}\top$ for some $b\in\gen_{\top}\setminus\{a\}$, and if $\lvert\gen_{g_{a}}\rvert\geq 4$, then there does not exist a $\top$-compatible order of $\gen_{\top}$.
\end{theorem}
\begin{proof}
	Assume that there exists a $\top$-compatible order. Since $\PP_{\top}$ is totally chain-connected, Theorem~\ref{thm:main_hurwitz} implies in particular that $\PP_{abb}$ is Hurwitz-connected.

	Suppose that $ab=bc=\cdots=da$, and rename $a=a_{0}$ and $b=b_{0}$.  We can thus write
	\begin{displaymath}
		a_{0}b_{0}b_{0} = b_{0}c_{0}b_{0} = b_{0}b_{0}a_{1} = b_{0}a_{1}b_{1} = a_{1}b_{1}b_{1},
	\end{displaymath}
	for appropriate elements $c_{0}, a_{1}, b_{1}$.
	
	If we continue this process, the finiteness of $\gen_{\top}$ implies that we can find a minimal $k\geq 1$ with $a_{k}=a_{0}$ and $b_{k}=b_{0}$.  If $k=1$, then we may conclude from $b_0 a_1 = a_1 b_1$ that $a$ and $b$ commute.  We thus obtain the contradiction $\bigl\lvert\gen_{g_{a}}\bigr\rvert=2$.  We therefore have $k\geq 2$.
	
	Moreover, we observe that for $i\in\{0,1,\ldots,k-1\}$ we have $a_{i}b_{i}=b_{i}c_{i}$ and $c_{i}b_{i}=a_{i+1}b_{i+1}$.  Let us denote by $C_{i}$ the cycle in $\Gamma(\PP_{abb})$ labeled by $a_{i}b_{i}$.  Lemma~\ref{prop:reverse_same_orbit_size} implies $\lvert C_{i}\rvert=\lvert C_{i+1}\rvert$.
	
	In particular for each $i\in\{0,1,\ldots,k-1\}$ the graph $\Gamma(\PP_{abb})$ contains the $2$-cycle $b_{i}\overset{a_{i}b_{i}}{\longrightarrow}c_{i}\overset{a_{i+1}b_{i+1}}{\longrightarrow}b_{i}$.  We also observe that there exists a cycle 
	\begin{displaymath}
		a_{0}\overset{a_{0}b_{0}}{\longrightarrow}b_{0}\overset{a_{1}b_{1}}{\longrightarrow}a_{1}\overset{a_{1}b_{1}}{\longrightarrow}b_{1}\overset{a_{2}b_{2}}{\longrightarrow}\cdots\overset{a_{k-1}b_{k-1}}{\longrightarrow}a_{k-1}\overset{a_{k-1}b_{k-1}}{\longrightarrow}b_{k-1}\overset{a_{0}b_{0}}{\longrightarrow}a_{0}.
	\end{displaymath}
	We have therefore exhibited $k+1$ cycles that are mutually edge-disjoint by construction, whereas there are only $k$ elements of length $2$ in $\PP_{abb}$.
	We conclude that any feedback arc set of $\Gamma(\PP_{abb})$ needs to contain two edges with the same label, which in light of Proposition~\ref{prop:minimal_feedback_compatible} implies that we cannot find an $abb$-compatible order of $\gen_{abb}$.  Consequently we cannot find a $\top$-compatible order of $\gen_{\top}$.
\end{proof}

\begin{example}\label{ex:big_cycle_example}
	Let us illustrate Theorem~\ref{thm:big_cycle} in the case $\lvert\gen_{g_{a}}\rvert=5$ and $k=3$.  Let $\top=abb$, and write $a_{0}=a$ and $b_{0}=b$.  We obtain
	\begin{align*}
		a_{0}b_{0}b_{0} & = b_{0}c_{0}b_{0} = b_{0}b_{0}a_{1} = b_{0}a_{1}b_{1} = a_{1}b_{1}b_{1}\\
		& = b_{1}c_{1}b_{1} = b_{1}b_{1}a_{2} = b_{1}a_{2}b_{2} = a_{2}b_{2}b_{2}\\
		& = b_{2}c_{2}b_{2} = b_{2}b_{2}a_{0} = b_{2}a_{0}b_{0} = a_{0}b_{0}b_{0}.
	\end{align*}
	Figure~\ref{fig:big_cycle_example_graph} shows the cycle graph of the factorization poset $\PP_{\top}$ displayed in Figure~\ref{fig:big_cycle_example_poset}.
	The reader is invited to verify that there does not exist a $\top$-compatible order of $\gen_{\top}=\{a_{0},b_{0},c_{0},a_{1},b_{1},c_{1},a_{2},b_{2},c_{2}\}$.  By construction, however, $\bg_{3}$ acts transitively on $\red{\gen_{\top}}{\top}$, and moreover $\bg_{2}$ acts transitively on $\red{\gen_{\top}}{g}$ for any $g\leq_{\gen}\top$ with $\ell_{\gen_{\top}}(g)=2$.
	
	\begin{figure}
		\centering
		\begin{tikzpicture}\small
			\def\x{1};
			\def\y{1};
			\draw(3*\x,6*\y) node(na){$a_{0}$};
			\draw(5*\x,5*\y) node(nb){$b_{0}$};
			\draw(4*\x,4*\y) node(nc){$c_{0}$};
			\draw(2*\x,4*\y) node(nd){$c_{2}$};
			\draw(1*\x,5*\y) node(ne){$b_{2}$};
			\draw(5*\x,2*\y) node(nf){$a_{1}$};
			\draw(3*\x,1*\y) node(ng){$b_{1}$};
			\draw(3*\x,3*\y) node(nh){$c_{1}$};
			\draw(1*\x,2*\y) node(ni){$a_{2}$};
			\draw[->,blue](na) -- (nb);
			\draw[->,blue](4.85*\x,4.9*\y) -- (4.1*\x,4.15*\y);
			\draw[->,blue](nc) -- (nd);
			\draw[->,blue](1.9*\x,4.15*\y) -- (1.15*\x,4.9*\y);
			\draw[->,blue](ne) -- (na);
			\draw[->,green](4.15*\x,4.1*\y) -- (4.9*\x,4.85*\y);
			\draw[->,green](5.1*\x,4.8*\y) -- (nf);
			\draw[->,green](nf) -- (ng);
			\draw[->,green](3.05*\x,1.2*\y) -- (3.05*\x,2.8*\y);
			\draw[->,green](nh) -- (nc);
			\draw[->,red](2.95*\x,2.8*\y) -- (2.95*\x,1.2*\y);
			\draw[->,red](ng) -- (ni);
			\draw[->,red](ni) -- (.9*\x,4.8*\y);
			\draw[->,red](1.1*\x,4.85*\y) -- (1.85*\x,4.1*\y);
			\draw[->,red](nd) -- (nh);
			\draw[->,orange](nh) .. controls (2.5*\x,3*\y) and (1*\x,4.5*\y) .. (ne);
			\draw[->,orange](ne) -- (nb);
			\draw[->,orange](nb) .. controls (5*\x,4.5*\y) and (3.5*\x,3*\y) .. (nh);
			\draw[->,purple](5.1*\x,4.8*\y) -- (ng);
			\draw[->,purple](ng) -- (nd);
			\draw[->,purple](nd) -- (nb);
			\draw[->,yellow!70!gray](ne) -- (nc);
			\draw[->,yellow!70!gray](nc) -- (ng);
			\draw[->,yellow!70!gray](ng) -- (.9*\x,4.8*\y);
			\path (nb) edge[out=90,in=0,looseness=.75,distance=1.5em,->] (nb);
			\path (ne) edge[out=180,in=90,looseness=.75,distance=1.5em,->] (ne);
			\path (ng) edge[out=315,in=225,looseness=.75,distance=1.5em,->] (ng);
		\end{tikzpicture}
		\caption{The cycle graph $\Gamma(\PP_{abb})$ from Example~\ref{ex:big_cycle_example}, corresponding to the case $\lvert\gen_{g_{a}}\rvert=5$ and $k=3$ of Theorem~\ref{thm:big_cycle}.  (For simplicity the loops are all drawn black even though their labels are distinct.)}
		\label{fig:big_cycle_example_graph}
	\end{figure}

	\begin{figure}
		\centering
		\begin{tikzpicture}\small
			\def\x{1};
			\def\y{2};
			\draw(5*\x,1.25*\y) node(n1){$\id$};
			\draw(1*\x,2*\y) node(n2){$b_{0}$};
			\draw(2*\x,2*\y) node(n3){$a_{0}$};
			\draw(3*\x,2*\y) node(n4){$c_{0}$};
			\draw(4*\x,2*\y) node(n5){$a_{1}$};
			\draw(5*\x,2*\y) node(n6){$b_{2}$};
			\draw(6*\x,2*\y) node(n7){$c_{2}$};
			\draw(7*\x,2*\y) node(n8){$c_{1}$};
			\draw(8*\x,2*\y) node(n9){$b_{1}$};
			\draw(9*\x,2*\y) node(n10){$a_{2}$};
			\draw(1*\x,3*\y) node(n11){$b_{0}b_{0}$};
			\draw(2*\x,3*\y) node(n12){$a_{0}b_{0}$};
			\draw(3*\x,3*\y) node(n13){$c_{1}b_{2}$};
			\draw(4*\x,3*\y) node(n14){$b_{1}c_{2}$};
			\draw(5*\x,3*\y) node(n15){$c_{0}b_{0}$};
			\draw(6*\x,3*\y) node(n16){$b_{2}b_{2}$};
			\draw(7*\x,3*\y) node(n17){$b_{1}b_{2}$};
			\draw(8*\x,3*\y) node(n18){$b_{1}b_{1}$};
			\draw(9*\x,3*\y) node(n19){$c_{1}b_{1}$};
			\draw(5*\x,3.75*\y) node(n20){$a_{0}b_{0}b_{0}$};
			\draw(n1) -- (n2);
			\draw(n1) -- (n3);
			\draw(n1) -- (n4);
			\draw(n1) -- (n5);
			\draw(n1) -- (n6);
			\draw(n1) -- (n7);
			\draw(n1) -- (n8);
			\draw(n1) -- (n9);
			\draw(n1) -- (n10);
			\draw(n2) -- (n11);
			\draw(n2) -- (n12);
			\draw(n2) -- (n13);
			\draw(n2) -- (n14);
			\draw(n2) -- (n15);
			\draw(n3) -- (n12);
			\draw(n4) -- (n12);
			\draw(n4) -- (n15);
			\draw(n4) -- (n17);
			\draw(n5) -- (n15);
			\draw(n6) -- (n12);
			\draw(n6) -- (n13);
			\draw(n6) -- (n16);
			\draw(n6) -- (n17);
			\draw(n6) -- (n19);
			\draw(n7) -- (n12);
			\draw(n7) -- (n14);
			\draw(n7) -- (n19);
			\draw(n8) -- (n13);
			\draw(n8) -- (n15);
			\draw(n8) -- (n19);
			\draw(n9) -- (n14);
			\draw(n9) -- (n15);
			\draw(n9) -- (n17);
			\draw(n9) -- (n18);
			\draw(n9) -- (n19);
			\draw(n10) -- (n19);
			\draw(n11) -- (n20);
			\draw(n12) -- (n20);
			\draw(n13) -- (n20);
			\draw(n14) -- (n20);
			\draw(n15) -- (n20);
			\draw(n16) -- (n20);
			\draw(n17) -- (n20);
			\draw(n18) -- (n20);
			\draw(n19) -- (n20);
		\end{tikzpicture}
		\caption{The factorization poset of Example~\ref{ex:big_cycle_example}, corresponding to the cycle graph from Figure~\ref{fig:big_cycle_example_graph}.}
		\label{fig:big_cycle_example_poset}
	\end{figure}
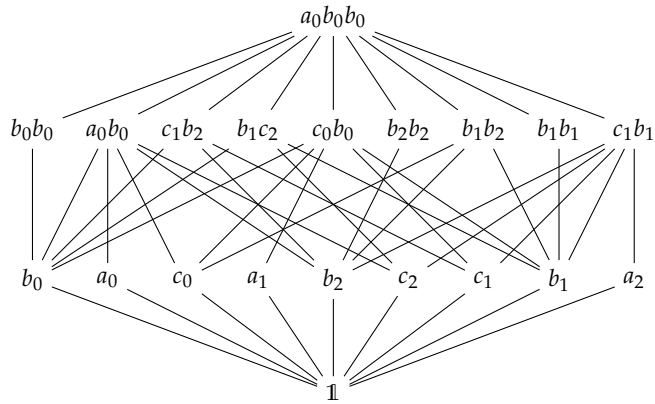
\end{example}

\begin{remark}
	If $\lvert\gen_{g_{a}}\rvert=4$, then we have $c_{0}=c_{1}=\cdots=c_{k-1}$.  Indeed, observe that we can write $c_{0}b_{0}=b_{0}a_{1}=a_{1}b_{1}=b_{1}c_{1}$.  Lemma~\ref{prop:reverse_same_orbit_size} implies that this is the whole Hurwitz orbit, which forces $c_{0}=c_{1}$. 
\end{remark}

\begin{theorem}\label{thm:particular_top}
	Let $\PP_{\top}(\grp,\gen)$ be a factorization poset satisfying the assumptions of Theorem~\ref{thm:generator_single_cycle}.  If $\top=ab^{n-1}$ and $\bigl\lvert\gen_{g_{a}}\bigr\rvert=3$, then Conjecture~\ref{conj:compatible_linear_cycle_graph} holds.
\end{theorem}
\begin{proof}
	Assume the hypotheses of Conjecture~\ref{conj:compatible_linear_cycle_graph}: $\PP_{\top}$ is totally chain-connected, and there exists a $\top$-compatible order $\prec$ of $\gen_{\top}$.  In particular, Theorem~\ref{thm:main_hurwitz} implies that $\PP_{\top}$ is Hurwitz-connected.
	
	If $a^{2}\leq_{\gen}\top$, then Proposition~\ref{prop:subword_order} implies that there exists a reduced $\gen$-factoriza\-tion of $\top$ that starts with $a^{2}$, which is connected by a sequence of Hurwitz moves to the factorization $ab^{n-1}$.  In that case we can verify that there needs to exist an element in $\rtwo_{\top}$ different from both $g_{a}$ and $a^{2}$ that covers $a$, which is a contradiction.  Hence we conclude $a^{2}\not\leq_{\gen}\top$.
	
	Without loss of generality let $g_{a} = ab = bc = ca$.  If $n=3$, then Lemma~\ref{prop:reverse_same_orbit_size} implies that there is a unique possible cycle graph, which is shown in Figure~\ref{fig:particular_top_3_graph}.  The corresponding factorization poset is shown in Figure~\ref{fig:particular_top_3_poset}.  There are four $\top$-compatible orders of $\gen_{\top}$ (namely $a\prec c\prec d\prec b$ and its cyclic shifts, see Lemma~\ref{lem:compatible_cyclic}), and we can check that in each case the reduced cycle graph $\Gamma_{\prec}(\PP_{\top})$ induces a linear order.
	
	\begin{figure}
		\centering
		\begin{subfigure}[t]{.4\textwidth}
			\centering
			\begin{tikzpicture}\small
				\def\x{1};
				\def\y{1};
				\draw(2.5*\x,1*\y) node(n1){$\id$};
				\draw(1*\x,2*\y) node(n2){$a$};
				\draw(2*\x,2*\y) node(n3){$b$};
				\draw(3*\x,2*\y) node(n4){$c$};
				\draw(4*\x,2*\y) node(n5){$d$};
				\draw(1*\x,3*\y) node(n6){$bb$};
				\draw(2*\x,3*\y) node(n7){$ab$};
				\draw(3*\x,3*\y) node(n8){$bc$};
				\draw(4*\x,3*\y) node(n9){$cc$};
				\draw(2.5*\x,4*\y) node(n10){$abb$};
				\draw(n1) -- (n2);
				\draw(n1) -- (n3);
				\draw(n1) -- (n4);
				\draw(n1) -- (n5);
				\draw(n2) -- (n7);
				\draw(n3) -- (n6);
				\draw(n3) -- (n7);
				\draw(n3) -- (n8);
				\draw(n4) -- (n7);
				\draw(n4) -- (n8);
				\draw(n4) -- (n9);
				\draw(n5) -- (n8);
				\draw(n6) -- (n10);
				\draw(n7) -- (n10);
				\draw(n8) -- (n10);
				\draw(n9) -- (n10);
			\end{tikzpicture}
			\caption{The factorization poset of $\PP_{\top}$ in the case $n=3$ of Theorem~\ref{thm:particular_top}.}
			\label{fig:particular_top_3_poset}
		\end{subfigure}
		\hspace*{.1\textwidth}
		\begin{subfigure}[t]{.4\textwidth}
			\centering
			\begin{tikzpicture}\small
				\def\x{1.5};
				\def\y{1.5};
				\draw(1*\x,2*\y) node(na){$a$};
				\draw(2*\x,2*\y) node(nb){$b$};
				\draw(1*\x,1*\y) node(nc){$c$};
				\draw(2*\x,1*\y) node(nd){$d$};
				\draw[->,red](na) -- (nb);
				\draw[->,red](1.9*\x,1.85*\y) -- (1.15*\x,1.1*\y);
				\draw[->,red](nc) -- (na);
				\draw[->,blue](1.1*\x,1.15*\y) -- (1.85*\x,1.9*\y);
				\draw[->,blue](nb) -- (nd);
				\draw[->,blue](nd) -- (nc);
				\path (nb) edge[out=90,in=0,looseness=.75,distance=1.5em,->] (nb);
				\path (nc) edge[out=270,in=180,looseness=.75,distance=1.5em,->] (nc);
			\end{tikzpicture}
			\caption{The cycle graph of the factorization poset of Figure~\ref{fig:particular_top_3_poset}.  (For simplicity the loops are all drawn black even though their labels are distinct.)}
			\label{fig:particular_top_3_graph}
		\end{subfigure}
		\caption{Illustration of the case $n=3$ in Theorem~\ref{thm:particular_top}.}
		\label{fig:particular_top_3}
	\end{figure}
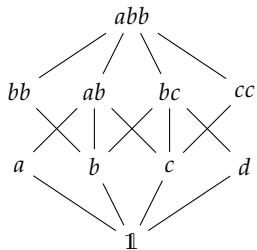
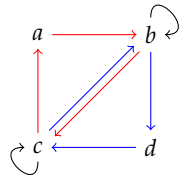
	
	If $n=4$, then Lemma~\ref{prop:reverse_same_orbit_size} implies once again that there is a unique possible cycle graph, which is shown in Figure~\ref{fig:particular_top_4_graph}.  The corresponding factorization poset is shown in Figure~\ref{fig:particular_top_4_poset}.  There are six $\top$-compatible orders of $\gen_{\top}$ (namely $a\prec c\prec f\prec d\prec e\prec b$ and its cyclic shifts), and we can check that in each case the reduced cycle graph $\Gamma_{\prec}(\PP_{\top})$ induces a linear order.

	\begin{figure}
		\centering
		\begin{subfigure}[t]{.4\textwidth}
			\centering
			\begin{tikzpicture}\small
				\def\x{1};
				\def\y{1};
				\draw(4*\x,1*\y) node(n1){$\id$};
				\draw(1*\x,2*\y) node(n2){$a$};
				\draw(2*\x,2*\y) node(n3){$c$};
				\draw(3*\x,2*\y) node(n4){$b$};
				\draw(4*\x,2*\y) node(n5){$f$};
				\draw(5*\x,2*\y) node(n6){$d$};
				\draw(6*\x,2*\y) node(n7){$e$};
				\draw(1*\x,3*\y) node(n8){$ab$};
				\draw(2*\x,3*\y) node(n9){$cc$};
				\draw(3*\x,3*\y) node(n10){$bb$};
				\draw(4*\x,3*\y) node(n11){$cb$};
				\draw(5*\x,3*\y) node(n12){$cd$};
				\draw(6*\x,3*\y) node(n13){$db$};
				\draw(7*\x,3*\y) node(n14){$dd$};
				\draw(2*\x,4*\y) node(n15){$ccc$};
				\draw(3*\x,4*\y) node(n16){$abb$};
				\draw(4*\x,4*\y) node(n17){$bbb$};
				\draw(5*\x,4*\y) node(n18){$bdd$};
				\draw(6*\x,4*\y) node(n19){$bbe$};
				\draw(7*\x,4*\y) node(n20){$ddd$};
				\draw(4*\x,5*\y) node(n21){$abbb$};
				\draw(n1) -- (n2);
				\draw(n1) -- (n3);
				\draw(n1) -- (n4);
				\draw(n1) -- (n5);
				\draw(n1) -- (n6);
				\draw(n1) -- (n7);
				\draw(n2) -- (n8);
				\draw(n3) -- (n8);
				\draw(n3) -- (n9);
				\draw(n3) -- (n11);
				\draw(n3) -- (n12);
				\draw(n4) -- (n8);
				\draw(n4) -- (n10);
				\draw(n4) -- (n11);
				\draw(n4) -- (n13);
				\draw(n5) -- (n12);
				\draw(n6) -- (n11);
				\draw(n6) -- (n12);
				\draw(n6) -- (n14);
				\draw(n6) -- (n13);
				\draw(n7) -- (n13);
				\draw(n8) -- (n16);
				\draw(n9) -- (n15);
				\draw(n9) -- (n16);
				\draw(n9) -- (n18);
				\draw(n10) -- (n16);
				\draw(n10) -- (n17);
				\draw(n10) -- (n19);
				\draw(n11) -- (n16);
				\draw(n11) -- (n18);
				\draw(n11) -- (n19);
				\draw(n12) -- (n18);
				\draw(n13) -- (n19);
				\draw(n14) -- (n18);
				\draw(n14) -- (n19);
				\draw(n14) -- (n20);
				\draw(n15) -- (n21);
				\draw(n16) -- (n21);
				\draw(n17) -- (n21);
				\draw(n18) -- (n21);
				\draw(n19) -- (n21);
				\draw(n20) -- (n21);
			\end{tikzpicture}
			\caption{The factorization poset of $\PP_{\top}$ in the case $n=4$ of Theorem~\ref{thm:particular_top}.}
			\label{fig:particular_top_4_poset}
		\end{subfigure}
		\hspace*{.1\textwidth}
		\begin{subfigure}[t]{.4\textwidth}
			\centering
			\begin{tikzpicture}\small
				\def\x{1};
				\def\y{1};
				\draw(3*\x,3*\y) node(na){$a$};
				\draw(4*\x,2*\y) node(nb){$b$};
				\draw(2*\x,2*\y) node(nc){$c$};
				\draw(3*\x,1*\y) node(nd){$d$};
				\draw(5*\x,1*\y) node(ne){$e$};
				\draw(1*\x,1*\y) node(nf){$f$};
				\draw[->,blue](na) -- (nb);
				\draw[->,blue](3.9*\x,1.95*\y) -- (2.1*\x,1.95*\y);
				\draw[->,blue](nc) -- (na);
				\draw[->,red](nb) -- (ne);
				\draw[->,red](ne) -- (nd);
				\draw[->,red](3.1*\x,1.15*\y) -- (3.85*\x,1.9*\y);
				\draw[->,green](nd) -- (nf);
				\draw[->,green](nf) -- (nc);
				\draw[->,green](2.15*\x,1.9*\y) -- (2.9*\x,1.15*\y);
				\draw[->,yellow!70!gray](3.9*\x,1.85*\y) -- (3.15*\x,1.1*\y);
				\draw[->,yellow!70!gray](2.85*\x,1.1*\y) -- (2.1*\x,1.85*\y);
				\draw[->,yellow!70!gray](2.1*\x,2.05*\y) -- (3.9*\x,2.05*\y);
				\path (nb) edge[out=90,in=0,looseness=.75,distance=1.5em,->] (nb);
				\path (nc) edge[out=180,in=90,looseness=.75,distance=1.5em,->] (nc);
				\path (nd) edge[out=315,in=225,looseness=.75,distance=1.5em,->] (nd);
			\end{tikzpicture}
			\caption{The cycle graph of the factorization poset of Figure~\ref{fig:particular_top_4_poset}.  (For simplicity the loops are all drawn black even though they represent distinct elements.)}
			\label{fig:particular_top_4_graph}
		\end{subfigure}
		\caption{Illustration of the case $n=4$ in Theorem~\ref{thm:particular_top}.}
		\label{fig:particular_top_4}
	\end{figure}
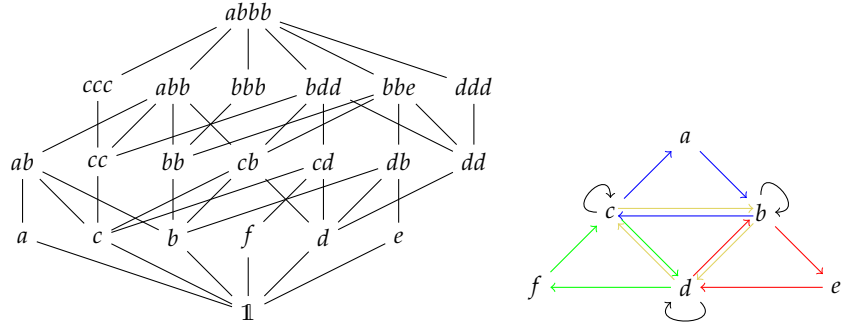
	
	Next let $n=5$, and let $a_{0}=a$, $b_{0}=b$ and $c_{0}=c$.  We have $\top=a_{0}b_{0}b_{0}b_{0}b_{0}=b_{0}c_{0}b_{0}b_{0}b_{0}$, which in view of Proposition~\ref{prop:subword_order} implies that $c_{0}b_{0}\in\rtwo_{\top}$.  Lemma~\ref{prop:reverse_same_orbit_size} implies that there is some $d\in\gen_{\top}$ with $c_{0}b_{0}=b_{0}d=dc_{0}$.  We continue to find $b_{0}c_{0}b_{0}b_{0}b_{0}=b_{0}b_{0}db_{0}b_{0}$, and as before we find $db_{0}\in\rtwo_{\top}$, which forces the existence of $b_{1}\in\gen_{\top}$ with $db_{0}=b_{0}b_{1}=b_{1}d$.  Now we have $b_{0}b_{0}db_{0}b_{0}=b_{0}b_{0}b_{0}b_{1}b_{0}$ so that $b_{1}b_{0}\in\rtwo_{\top}$ and we find $a_{1}\in\gen_{\top}$ with $b_{1}b_{0}=b_{0}a_{1}=a_{1}b_{1}$.  We finish this iteration to obtain
	\begin{displaymath}
		\top = b_{0}b_{0}b_{0}b_{1}b_{0} = b_{0}b_{0}b_{0}b_{0}a_{1} = b_{0}b_{0}b_{0}a_{1}b_{1} = b_{0}b_{0}a_{1}b_{1}b_{1} = b_{0}a_{1}b_{1}b_{1}b_{1} = a_{1}b_{1}b_{1}b_{1}b_{1}.
	\end{displaymath}
	Since $\gen_{\top}$ is finite we can find some $k>1$ with $a_{k}=a_{0}$ and $b_{k}=b_{0}$.  This also implies that $c_{0}$ is in fact $b_{k-1}$, and we have $b_{i-1}b_{i}=db_{i-1}=b_{i}d$ for $i\in[k]$.
	
	For $k=2$ the cycle graph consists of two cycles $a\rightarrow b\rightarrow c\rightarrow a$, and $a\rightarrow c\rightarrow b\rightarrow a$, which contradicts the assumption that $g_{a}$ is the unique element covering $a$.  But even without that assumption we would still obtain a contradiction to the fact that $\prec$ is $\top$-compatible.
	
	For $k=3$ we can find the following sequence of equalities coming from repeated Hurwitz moves:
	\begin{align*}
		a_{0}b_{0}b_{0}b_{0}b_{0} & = b_{0}b_{2}b_{0}b_{0}b_{0} = b_{0}db_{2}b_{0}b_{0} = db_{2}b_{2}b_{0}b_{0} = db_{2}db_{2}b_{0} = db_{1}b_{2}b_{2}b_{0}\\
		& = db_{1}b_{2}b_{0}d = \mathbf{db_{1}db_{2}d} = db_{0}b_{1}b_{2}d = db_{0}b_{2}dd = da_{0}b_{0}dd = \mathbf{da_{0}db_{2}d}.
	\end{align*}
	The highlighted words imply $a_{0}=b_{1}$, which contradicts the assumption that $g_{a}$ is the unique element covering $a$.
	
	For $k\geq 4$ we observe that $\Gamma(\PP_{\top})$ has the following $2k+1$ cycles: $b_{i}\overset{b_{i+1}d}{\longrightarrow}b_{i+1}\overset{a_{i+1}b_{i+1}}{\longrightarrow}b_{i}$ and $b_{i}\overset{b_{i}d}{\longrightarrow}d\overset{b_{i+1}d}{\longrightarrow}b_{i}$ for $i\in\{0,1,\ldots,k-1\}$ as well as
	\begin{displaymath}
		a_{0}\overset{a_{0}b_{0}}{\longrightarrow}b_{0}\overset{a_{1}b_{1}}{\longrightarrow}a_{1}\overset{a_{1}b_{1}}{\longrightarrow}b_{1}\overset{a_{2}b_{2}}{\longrightarrow}a_{2}\overset{a_{2}b_{2}}{\longrightarrow}\cdots b_{k-1}\overset{a_{0}b_{0}}{\longrightarrow}a_{0}.
	\end{displaymath}
	It follows that any feedback arc set of $\PP_{\top}$ needs to contain at least two edges with the same label, which in view of Proposition~\ref{prop:minimal_feedback_compatible} contradicts the assumption that $\prec$ is $\top$-compatible.  In fact, if $k$ is even, then we can explicitly find an element in $\rtwo_{\top}$ on whose set of reduced $\gen$-factorizations $\bg_{2}$ does not act transitively.  See Figure~\ref{fig:particular_top_5_4} for an illustration of the case $k=4$.
	
	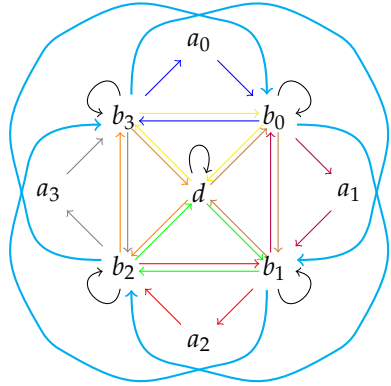
\begin{figure}
		\centering
		\begin{tikzpicture}
			\def\x{1};
			\def\y{1};
			\draw(1*\x,1*\y) node(n1){$b_{2}$};
			\draw(3*\x,1*\y) node(n2){$b_{1}$};
			\draw(2*\x,2*\y) node(n3){$d$};
			\draw(1*\x,3*\y) node(n4){$b_{3}$};
			\draw(3*\x,3*\y) node(n5){$b_{0}$};
			\draw(2*\x,0*\y) node(n6){$a_{2}$};
			\draw(0*\x,2*\y) node(n7){$a_{3}$};
			\draw(4*\x,2*\y) node(n8){$a_{1}$};
			\draw(2*\x,4*\y) node(n9){$a_{0}$};
			\draw[->,blue](2.8*\x,2.95*\y) -- (1.2*\x,2.95*\y);
			\draw[->,blue](n4) -- (n9);
			\draw[->,blue](n9) -- (n5);
			\draw[->,yellow](1.2*\x,3.05*\y) -- (2.8*\x,3.05*\y);
			\draw[->,yellow](2.85*\x,2.9*\y) -- (2.1*\x,2.15*\y);
			\draw[->,yellow](1.9*\x,2.15*\y) -- (1.15*\x,2.9*\y);
			\draw[->,red](1.2*\x,1.05*\y) -- (2.8*\x,1.05*\y);
			\draw[->,red](n2) -- (n6);
			\draw[->,red](n6) -- (n1);
			\draw[->,green](2.8*\x,.95*\y) -- (1.2*\x,.95*\y);
			\draw[->,green](1.15*\x,1.1*\y) -- (1.9*\x,1.85*\y);
			\draw[->,green](2.1*\x,1.85*\y) -- (2.85*\x,1.1*\y);
			\draw[->,purple](2.95*\x,1.2*\y) -- (2.95*\x,2.8*\y);
			\draw[->,purple](n5) -- (n8);
			\draw[->,purple](n8) -- (n2);
			\draw[->,brown](3.05*\x,2.8*\y) -- (3.05*\x,1.2*\y);
			\draw[->,brown](2.9*\x,1.15*\y) -- (2.15*\x,1.9*\y);
			\draw[->,brown](2.15*\x,2.1*\y) -- (2.9*\x,2.85*\y);
			\draw[->,gray](1.05*\x,2.8*\y) -- (1.05*\x,1.2*\y);
			\draw[->,gray](n1) -- (n7);
			\draw[->,gray](n7) -- (n4);
			\draw[->,orange](.95*\x,1.2*\y) -- (.95*\x,2.8*\y);
			\draw[->,orange](1.1*\x,2.85*\y) -- (1.85*\x,2.1*\y);
			\draw[->,orange](1.85*\x,1.9*\y) -- (1.1*\x,1.15*\y);
			\draw[->,rounded corners,thick,cyan] (3.3*\x,2.9*\y) .. controls (3.5*\x,2.9*\y) and (4*\x,2.9*\y) .. (4.25*\x,2.5*\y) .. controls (4.5*\x,1.25*\y) and (4.75*\x,1.25*\y) .. (4*\x,0*\y) .. controls (2.75*\x,-.75*\y) and (2.75*\x,-.5*\y) .. (1.5*\x,-.25*\y) .. controls (1.1*\x,0*\y) and (1.1*\x,.5*\y) .. (1.1*\x,.7*\y);
			\draw[->,rounded corners,thick,cyan] (.7*\x,1.1*\y) .. controls (.5*\x,1.1*\y) and (0*\x,1.1*\y) .. (-.25*\x,1.5*\y) .. controls (-.5*\x,2.75*\y) and (-.75*\x,2.75*\y) .. (0*\x,4*\y) .. controls (1.25*\x,4.75*\y) and (1.25*\x,4.5*\y) .. (2.5*\x,4.25*\y) .. controls (2.9*\x,4*\y) and (2.9*\x,3.5*\y) .. (2.9*\x,3.3*\y);
			\draw[->,rounded corners,thick,cyan] (1.1*\x,3.3*\y) .. controls (1.1*\x,3.5*\y) and (1.1*\x,4*\y) .. (1.5*\x,4.25*\y) .. controls (2.75*\x,4.5*\y) and (2.75*\x,4.75*\y) .. (4*\x,4*\y) .. controls (4.75*\x,2.75*\y) and (4.5*\x,2.75*\y) .. (4.25*\x,1.5*\y) .. controls (4*\x,1.1*\y) and (3.5*\x,1.1*\y) .. (3.3*\x,1.1*\y);
			\draw[->,rounded corners,thick,cyan] (2.9*\x,.7*\y) .. controls (2.9*\x,.5*\y) and (2.9*\x,0*\y) .. (2.5*\x,-.25*\y) .. controls (1.25*\x,-.5*\y) and (1.25*\x,-.75*\y) .. (0*\x,0*\y) .. controls (-.75*\x,1.25*\y) and (-.5*\x,1.25*\y) .. (-.25*\x,2.5*\y) .. controls (0*\x,2.9*\y) and (.5*\x,2.9*\y) .. (.7*\x,2.9*\y);
			\path (n1) edge[out=260,in=190,looseness=.75,distance=1.5em,->] (n1);
			\path (n2) edge[out=280,in=350,looseness=.75,distance=1.5em,->] (n2);
			\path (n4) edge[out=100,in=170,looseness=.75,distance=1.5em,->] (n4);
			\path (n5) edge[out=80,in=10,looseness=.75,distance=1.5em,->] (n5);
			\path (n3) edge[out=110,in=70,looseness=.75,distance=1.5em,->] (n3);
		\end{tikzpicture}
		\caption{The cycle graph $\Gamma(\PP_{abbbb})$ corresponding to the case $k=4$ of Theorem~\ref{thm:particular_top}.  (For simplicity the loops are all drawn black even though they represent distinct elements.)  The thick blue edges represent a rank-$2$ element which affords a non-transitive Hurwitz action.}
		\label{fig:particular_top_5_4}
	\end{figure}
	
	If $n>5$, then we observe that $\PP_{\top}$ contains an interval isomorphic to $\PP_{ab^{4}}$.  We have seen in the previous paragraph, that the restriction of $\prec$ to this interval cannot be $ab^{4}$-compatible, which implies that $\prec$ cannot be $\top$-compatible.  This, however, contradicts our assumption on $\prec$.
\end{proof}

We can now conclude the proof of Theorem~\ref{thm:generator_single_cycle}.

\begin{proof}[Proof of Theorem~\ref{thm:generator_single_cycle}]
	Suppose that $\PP_{\top}$ is totally chain-connected and there exists a $\top$-compatible order of $\gen_{\top}$.  In particular, Theorem~\ref{thm:main_hurwitz} implies that $\PP_{\top}$ is Hurwitz-connected.  We distinguish three cases according to the size of $\gen_{g_{a}}$.  
	
	(i) Suppose $\bigl\lvert\gen_{g_{a}}\bigr\rvert=2$.  This means that $\gen_{g_{a}}=\{a,b\}$, and $a$ and $b$ commute.  By Hurwitz-connectivity we conclude $\gen_{\top}=\{a,b\}$, and the claim follows trivially.
	
	(ii) Suppose $\bigl\lvert\gen_{g_{a}}\bigr\rvert=3$.  Let $\gen_{g_{a}}=\{a,b,c\}$, and say that $ab=bc=ca$.  If $aa\leq_{\gen}\top$, then we have either $\top=a^n$ or $aab\leq_{\gen}\top$.  In the first case, we obtain by Hurwitz-connectivity the contradiction $\gen_{\top}=\{a\}$.  In the second case we have $aab=aca$, which forces the existence of $ca\in\rtwo_{\top}$ with $a\lessdot_{\gen}ca$, which contradicts our assumption that $g_{a}$ is the unique upper cover of $a$ (except possibly $a^{2}$).  We conclude $aa\not\leq_{\gen}\top$.  Thus $\top=ab^{n-1}$, and the claim follows from Theorem~\ref{thm:particular_top}.
	
	(iii) Suppose $\bigl\lvert\gen_{g_{a}}\bigr\rvert\geq 4$.  Analogously to (ii) we conclude $\top=ab^{n-1}$, and the claim follows from Theorem~\ref{thm:big_cycle}.
\end{proof}

\section*{Acknowledgements}
We thank the anonymous referee for the many comments and suggestions that helped improve both content and exposition of this paper.


\end{document}